\documentclass[11pt,reqno]{amsart}
\usepackage{amsmath, amsfonts, amsthm, amssymb, mathrsfs, amscd, graphicx, cite, color}
\usepackage{tikz}
\usepackage{latexsym,hyperref}

\makeatletter
\@namedef{subjclassname@2020}{%
  \textup{2020} Mathematics Subject Classification}
\makeatother

\usepackage[margin=1.1in]{geometry}

\usepackage{mathtools}

\usepackage{hyperref}
\hypersetup{hidelinks}

\allowdisplaybreaks

\newtheorem{theorem}{Theorem}[section]
\newtheorem{lemma}{Lemma}[section]

\newtheorem{proposition}{Proposition}[section]

\theoremstyle{definition}

\theoremstyle{remark}
\newtheorem{remark}{Remark}[section]

\newcommand{\bbr}{\mathbb R}

\def\charf {\mbox{{\text 1}\kern-.30em {\text l}}}

\def\di{\displaystyle}

\title[Vanishing viscosity limit for same family interacting shocks]
{Vanishing viscosity limit to two interacting shocks from the same family for the compressible Navier-Stokes equations}

\author[L.-A. Li]{Lin-An Li}
\address{School of Mathematical Sciences, Laboratory of Mathematics and Complex Systems, MOE, Beijing Normal University, Beijing 100875, P. R. China.}
\email{linanli@amss.ac.cn}

\author[D. Wang]{Dehua Wang}
\address{Department of Mathematics, University of Pittsburgh, Pittsburgh, PA 15260, USA.}
\email{dhwang@pitt.edu}

\author[Y. Wang]{Yi Wang}
\address{State Key Laboratory of Mathematical Sciences and Institute of Applied Mathematics, Academy of Mathematics and Systems Science, Chinese Academy of Sciences, Beijing 100190, P. R.~China; School of Mathematical Sciences, University of Chinese Academy of Sciences, Beijing 100049, P. R.~China}
\email{wangyi@amss.ac.cn}

\begin{document}

\begin{abstract}

We investigate the vanishing viscosity limit for the one-dimensional compressible Navier-Stokes equations in the regime of two interacting shock waves from the same characteristic family of the underlying Euler equations. Unlike the interaction of two shocks from distinct families, which produces two outgoing shocks in their respective original families, the collision of two same-family shocks generates an outgoing shock in the same family together with a rarefaction wave in the other family. This configuration is more singular because the collision occurs on a larger time scale that depends inversely on the wave strength, making it challenging to justify the vanishing viscosity limit not only in the processes before and after the shock collision but, more crucially, at the collision point itself. Furthermore, the simultaneous emergence of both shock and rarefaction waves after the collision introduces an additional layer of difficulty.
To overcome these obstacles, we first employ the anti-derivative method before the collision, which allows us to fix the locations of both viscous shocks precisely up to the collision point. However, uniform estimates with respect to the viscosity cannot be closed all the way to the collision time; we therefore introduce a carefully constructed approximate collision time to obtain the required higher-order energy bounds.
After the collision, we apply a weighted relative entropy method, combined with time-dependent shifts, to handle the composite wave structure arising from the coexistence of shock and rarefaction waves.
Our main result establishes that, for suitably small wave strengths and viscosity coefficients, there exists a family of global smooth solutions to the Navier-Stokes equations that converge to the entropy solution of the Euler equations. This approach resolves the singular nature of same-family shock interactions and provides a rigorous justification of the vanishing viscosity limit with an explicit convergence rate in this physically significant regime. The techniques developed in this paper are expected to be applicable to other related problems in vanishing viscosity theory.

\end{abstract}

\keywords{Vanishing viscosity limit, compressible Navier-Stokes equations, interacting shocks, same family shocks, rarefaction waves.}
\subjclass[2020]{76N10, 35Q35, 35Q30, 35Q31, 76N06}	
\date{\today}

\maketitle 
\tableofcontents
%
%
\section{Introduction}  
\setcounter{equation}{0}

The convergence of solutions of the Navier-Stokes equations to those of the Euler equations as the viscosity coefficient tends to zero is a fundamental problem in fluid dynamics. This vanishing viscosity limit is important both mathematically and physically, as it bridges the gap between the dissipative description of fluids and the ideal inviscid model. However, the limit process is singular when the inviscid solution contains discontinuities such as shocks.
The problem is further complicated when multiple shocks interact, as the interaction introduces additional nonlinear phenomena and singularities.

In this paper, we investigate this problem for the one-dimensional (1D) compressible Navier-Stokes equations in the Lagrangian coordinates:
\begin{equation}  \label{NS}
\begin{cases}
\displaystyle v_t - u_x = 0,       \qquad \qquad \qquad x \in \bbr, ~t \in\bbr^+, \\
\displaystyle u_t + p(v)_x = \varepsilon\left( \frac{u_x}{v}\right)_x,
\end{cases}
\end{equation}
where $v=v(t,x) >0$ denotes the specific volume, $u=u(t,x)$ the velocity, and $p=p(t,x)$ the pressure which is given by the $\gamma$-law $p(v)=bv^{-\gamma}$ with $\gamma>1$ and $b$ being positive fluid constants. Without loss of generality, we take $b=1$ throughout this paper.
The constant $\varepsilon>0$ denotes the physical viscosity coefficient. Formally, as $\varepsilon \to 0+$, the system \eqref{NS} converges to the following compressible Euler equations (a.k.a. the $p$-system):
\begin{equation}  \label{ES}
	\begin{cases}
		\displaystyle v_t - u_x = 0,       \quad \qquad \qquad x \in \bbr, ~t \in\bbr^+, \\
		\displaystyle u_t + p(v)_x = 0.
	\end{cases}
\end{equation}
We consider the vanishing viscosity limit from \eqref{NS} to \eqref{ES} in the regime of  two interacting shocks from the same characteristic family. Without loss of generality, the following three-piecewise constant initial data to \eqref{ES} are imposed:
\begin{equation} \label{ESI}
	(v, u)(0, x) = \begin{cases}
		(v_-, u_-),     \quad x < 0, \\
		(v_*, u_*),     \quad 0< x < 1, \\
		(v_+, u_+),     \quad x > 1.
	\end{cases}
\end{equation}

Concerning the problem of the vanishing viscosity limit for the one-dimensional compressible Navier-Stokes equations with two-piecewise Riemann initial data for the associated compressible Euler equations, a considerable amount of related research exists. As early as 1950, Hopf \cite{Hopf1950} investigated the vanishing viscosity limit of shock/rarefaction waves for 1D scalar viscous conservation laws. When considering elementary waves to the systems, for a single shock wave, Hoff and Liu \cite{HL} first verified the vanishing viscosity limit from \eqref{NS} to \eqref{ES} with Riemann shock data and then Goodman and Xin \cite{GX} proved the vanishing viscosity limit for piecewise smooth non-interacting shocks to the system of conservation laws with artificial viscosity, and subsequently, Yu \cite{Y} established the inviscid limit including both initial and shock layers. Wang \cite{W} and Wang \cite{W-2} generalized the results in \cite{GX} to the isentropic and full compressible Navier-Stokes equations, respectively, for the vanishing physical viscosity limits.
Recently, Kang and Vasseur \cite{KV} proved the vanishing viscosity limit for shock waves in the compressible Navier-Stokes equations with density-dependent viscosity under $L^2$ weak topology by using the $a$-contraction method, and the uniqueness of shock in the sense of vanishing physical viscosity limit can also be established.
For the vanishing viscosity limit of the rarefaction wave, contact discontinuity and their superpositions even with shocks, one can refer to \cite{X-1, JNS, XZ, HLW, L-W-W, M, VW, ZPT, ZPWT, HWY-1, HJW, HWY-2, HWWY, KV2, HWWW} and the related references therein.

Note that the aforementioned results on the vanishing viscosity limit for elementary wave patterns and their composites are either Riemann data or non-interacting shocks. For general small BV data with possibly infinitely many and interacting shocks, Bianchini and Bressan \cite{BB} proved the vanishing viscosity limit for 1D system of conservation laws with artificial viscosity; and recently, Chen, Kang, and Vasseur \cite{CKV} verified the vanishing  viscosity limit of 1D isentropic compressible Navier-Stokes equations with density-dependent viscosity in $L^2$ weak topology.
In this paper, we consider the vanishing viscosity limit of 1D compressible Navier-Stokes equations \eqref{NS} with constant viscosity coefficients, which is different from  \cite{BB, CKV}. In particular, the purpose of the present paper is to investigate the vanishing physical viscosity limit of \eqref{NS} to \eqref{ES} in the regime of two interacting shocks arising from the same characteristic family.

For two interacting shocks from different characteristic families, Huang, Wang, Wang, and Yang \cite{HWWY2015} established the vanishing physical viscosity limit of the compressible Navier-Stokes equations \eqref{NS} by using anti-derivative techniques, which was extended by Shi, Yong, and Zhang \cite{SYZ} to the case of non-isentropic Euler equations with artificial viscosities.
Very recently, Bressan, Caravenna, and Shen \cite{Bressan2026} characterized the local asymptotic patterns of vanishing viscosity approximations for scalar conservation laws with strictly convex flux near the singularity of two interacting shocks.

In our paper, we assume that both shocks are of the second characteristic family,  abbreviated as two 2-shocks, without loss of generality. Since the wave speed on the left of the same-family shock is larger than that on the right, the left shock will inevitably overtake the right one, leading to a collision and interaction. After the interaction, the two shocks merge to generate a stronger 2-shock of the same family and a weaker 1-rarefaction wave from the first characteristic family.

\subsection{Shock interactions for the inviscid Euler equations}

When there are no vacuum states,  the system \eqref{ES} is strictly hyperbolic with two real distinct eigenvalues
$$ \Lambda_1 = -\sqrt{-p'(v)}, \qquad \Lambda_2 = \sqrt{-p'(v)}. $$
To illustrate the shock interactions in \eqref{ES}-\eqref{ESI}, we start from the Riemann solutions to the inviscid Euler equations \eqref{ES}, which include at most two elementary waves, each associated with one of the two distinct characteristic families: a 1-rarefaction or 1-shock from the first characteristic family, and a 2-rarefaction or 2-shock from the second characteristic family. For any right end state $(v_R, u_R)$ with $v_R>0$, in the phase space $\{(v,u)|v>0,u\in \bbr\}$, the $i$-rarefaction curve $\mathcal{R}_i(v_R, u_R) \ (i=1,2)$ is defined by
\begin{equation} \label{rareC}
	\mathcal{R}_i(v_R, u_R) :=\left\lbrace (v, u) | u<u_R, u=u_R-\int_{v_R}^{v}\Lambda_i(s)ds\right\rbrace,
\end{equation}
and the $i$-shock curve $\mathcal{S}_i(v_R, u_R) \  (i=1,2)$ is defined by
\begin{equation} \label{shockC}
	\mathcal{S}_i(v_R, u_R) :=\left\lbrace  (v, u)| \exists s\in\mathbb{R}, {\rm \ s.t.\ }\begin{cases}
		 -s(v_R - v) - (u_R -u) =0, \\ -s(u_R-u) + (p(v_R)-p(v))=0,
	\end{cases} \Lambda_i(v) > s > \Lambda_i(v_R) \right\rbrace.
\end{equation}
We consider two incoming $2$-shocks in \eqref{ES}-\eqref{ESI},  denoted by $S_l, S_r$, respectively, where $S_l: (v_-, u_-) \in \mathcal{S}_2(v_*, u_*),$ is the 2-shock which connects $(v_-, u_-)$ as the left state and $(v_*, u_*)$ as the right state with the shock speed $s_l > 0$, while $S_r: (v_*, u_*) \in \mathcal{S}_2(v_+, u_+),$ is the 2-shock connecting the left state $(v_*, u_*)$ and the right state $(v_+, u_+)$ with the shock speed $s_r > 0$; see Figure \ref{fig:shock-interaction1} below.
Since the shock $S_l$ propagates faster than $S_r$, the left shock $S_l$ will inevitably overtake the right one $S_r$ and collide at some point $Q = (t_0, x_0)$ with $t_0=\frac{1}{s_l-s_r} > 0$.
After the collision time $t_0$, there is another intermediate state $(v^*, u^*)$ with $ (v^*, u^*) \in \mathcal{S}_2(v_+, u_+) $ and $(v_-, u_-) \in \mathcal{R}_1(v^*, u^*)$, which means that the overtaking interaction of two incoming 2-shocks $S_l$ and $S_r$ forms an outgoing 2-shock $S_2: (v^*, u^*) \in \mathcal{S}_2(v_+, u_+) $ followed by a 1-rarefaction wave $R_1: (v_-, u_-) \in \mathcal{R}_1(v^*, u^*)$.
\begin{figure}[h]
	\begin{tikzpicture}[scale = 1.3]
		\draw[->] (-1,0) -- (4,0) node[right] {$x$};
		\draw[->] (0,0) -- (0,4) node[above] {$t$};
		
		\coordinate (O) at (0,0);
		\coordinate (X0T0) at (2,2);
		
		\draw (O) -- (X0T0) node[midway,left] {$S_l$};
		\draw (1.8,0) -- (X0T0) node[midway,right] {$S_r$};
		\draw[thick] (X0T0) -- +(1,1.5) node[right] {$S_2$};
		\draw[thick] (X0T0) -- +(-1,1) node[above left] {$R_1$};
		\draw[thick] (X0T0) -- +(-0.8,1.2);
		\draw[thick] (X0T0) -- +(-1.2,0.8);
		
		\node[above right] at (0,2) {$(v_-,u_-)$};
		\node[above right] at (2.5,2) {$(v_+,u_+)$};
		\node[above] at (2,2.6) {$(v^*,u^*)$};
		\node[below right] at (2,2) {$(x_0,t_0)$};
		\node[below] at (0,0) {$0$};
		\node[below] at (1.8,0) {$1$};
		\node[above] at (-1,0) {$(v_-,u_-)$};
		\node[above] at (1,0) {$(v_*,u_*)$};
		\node[above] at (3,0) {$(v_+,u_+)$};
		
		\draw[dashed] (-1,2) -- (4,2);	
	\end{tikzpicture}
	\centering
	\caption{Entropy solution $(V, U)(t,x)$}
	\label{fig:shock-interaction1}
\end{figure}
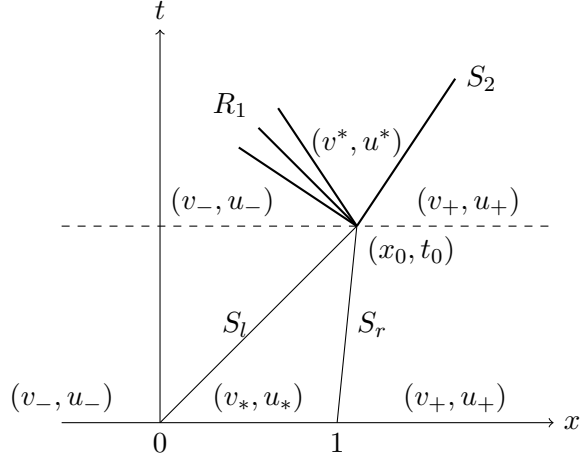

More precisely, let $\delta_l := v_* - v_-$ be the strength of the left shock $S_l$, and $\delta_r := v_+ - v_*$ be the strength of the right shock $S_r$. In the present paper, we assume that $\delta_l \sim \delta_r \sim \delta$ to be suitably small with the same order.
Then the strength of the outgoing 2-shock $S_2$ is 
$$\delta_2 := v_+ - v^* = \delta_l + \delta_r + O(1) \delta_l\delta_r(\delta_l + \delta_r),$$
and the resulting 1-rarefaction wave $R_1$ has strength of third order:
$$\delta_1 := v^* - v_- = O(1) \delta_l\delta_r(\delta_l + \delta_r);$$
see \cite{Liu, Smoller1994, Bressan2000, CEJ} for the details.

Denote by $(V, U)(t, x)$  the entropy solution to \eqref{ES}-\eqref{ESI} for two interacting 2-shocks. Before the collision time $ t_0 $, $ (V, U)(t, x) $ is given by
\begin{equation} \label{shock-shock}
	(V, U)(t, x) =
	\begin{cases} 
		(v_-, u_-), & x < s_l t, \quad t \leq t_0, \\
		(v_*, u_*), & s_l t < x < s_r t + 1, \quad t \leq t_0, \\
		(v_+, u_+), & x > s_r t + 1, \quad t \leq t_0.
	\end{cases}
\end{equation}
Alternatively, \eqref{shock-shock} can be rewritten as
\begin{equation} \label{shock-shock-1}
	(V, U)(t, x) = (v^S_l, u^S_l)(x - s_l t) + (v^S_r, u^S_r)(x - s_r t -1) - (v_*, u_*),  \qquad t\leq t_0,
\end{equation}
where
\begin{equation*}
	(v^S_l, u^S_l)\left( x - s_l t \right) = \begin{cases}
		(v_-, u_-),     \quad x - s_l t < 0, \\
		(v_*, u_*),     \quad x - s_l t > 0,
	\end{cases}
\end{equation*}
and
\begin{equation*}
	(v^S_r, u^S_r)\left( x - s_r t - 1 \right) = \begin{cases}
		(v_*, u_*),     \quad x - s_r t - 1 < 0, \\
		(v_+, u_+),     \quad x - s_r t - 1 > 0,
	\end{cases}
\end{equation*}
and the shock speeds  $ s_l, s_r $ are uniquely determined by the Rankine-Hugoniot conditions
\[
\begin{cases} 
	s_l(v_* - v_-) = -(u_* - u_-), \\
	s_l(u_* - u_-) = p(v_*) - p(v_-),
\end{cases}
\quad \text{and} \quad
\begin{cases} 
	s_r(v_+ - v_*) = -(u_+ - u_*), \\
	s_r(u_+ - u_*) = p(v_+) - p(v_*),
\end{cases}
\]
and Lax entropy conditions
$$0 < \Lambda_2(v_+) < s_r < \Lambda_2(v_*) < s_l < \Lambda_2(v_-).$$
By direct calculations,  
the two 2-shocks $S_l$ and $S_r$ collide at the point $ Q (t_0, x_0) := \left( \frac{1}{s_l - s_r}, \frac{s_l}{s_l - s_r} \right) $.
After the collision time $t_{0}$, we need to solve the Riemann problem of the $p$-system \eqref{ES} at time $t = t_0$ with the Riemann data
\begin{equation} \label{ESI-t0}
(v, u)(t=t_{0}, x)=\begin{cases}
	(v_-, u_-), & x < x_0, \\
	(v_+, u_+), & x > x_0,
\end{cases}
\end{equation}
which consists of the superposition of 
a stronger outgoing 2-shock $S_2: (v^*, u^*) \in \mathcal{S}_2(v_+, u_+) $ and a weaker 1-rarefaction wave $R_1: (v_-,u_-) \in \mathcal{R}_1(v^*, u^*)$. 

Precisely, denote by $(v^S, u^S)(t,x):= (v^S, u^S)\left( \frac{x - x_0}{t - t_0} \right)$  the outgoing 2-shock $S_2$ given by 
\begin{equation} \label{shock2}
	(v^S, u^S)\left( \frac{x - x_0}{t - t_0}\right) = (v^S, u^S)((x - x_0) - s_2 (t - t_0)) = \begin{cases}
		(v^*, u^*),     \quad (x - x_0) - s_2 (t - t_0) < 0, \\
		(v_+, u_+),     \quad (x - x_0) - s_2 (t - t_0) > 0,
	\end{cases}
\end{equation}
with the shock speed $s_2=\sqrt{-\frac{p(v_+)-p(v^*)}{v_+-v^*}}>0$.
Denote by $(v^R,u^R)\left( \frac{x - x_0}{t - t_0} \right)$   the outgoing 1-rarefaction wave $R_1$ given explicitly by 
\begin{align} \label{rare1}
	\begin{aligned}
		&\Lambda_1(v^R) \left( \frac{x - x_0}{t - t_0} \right)  = w^R\left( \frac{x - x_0}{t - t_0} \right), \\
		&Z_1(v^R,u^R)\left( \frac{x - x_0}{t - t_0} \right) = Z_1(v_-,u_-) = Z_1(v^*, u^*),
	\end{aligned}	
\end{align}
where  $Z_i(v,u) := u + \int^{v} \Lambda_i(s)ds\ (i=1,2)$ is the $i$-Riemann invariant of the Euler system \eqref{ES} and
\begin{equation*}
	w^R\left( \frac{x - x_0}{t - t_0} \right) = \begin{cases}
		\Lambda_1(v_-), \qquad x - x_0 < \Lambda_1(v_-)(t - t_0), \\
		\frac{x - x_0}{t - t_0}, \qquad\quad \Lambda_1(v_-)(t - t_0) \leq x - x_0 \leq \Lambda_1(v^*)(t - t_0), \\
		\Lambda_1(v^*) , \qquad  x - x_0 > \Lambda_1(v^*)(t - t_0),
	\end{cases}
\end{equation*}
solves the Riemann problem of the inviscid Burgers equation:
\begin{equation} \label{burgers-riemann}
	\begin{cases}
		\displaystyle w_t + ww_{x} = 0,    \qquad \qquad \qquad \qquad \qquad \qquad x \in \bbr, ~t > t_0, \\
		\displaystyle w(t_0, x) = w_0^R(x) = \begin{cases}
			\Lambda_1(v_-),  \quad x < x_0, \\
			\Lambda_1(v^*),  \quad x > x_0.
		\end{cases}
	\end{cases}
\end{equation}

Therefore, after the collision time $t_0$, 
\begin{equation} \label{rare-shock}
	(V, U)(t, x) = \left( v^R, u^R \right)\left( \frac{x - x_0}{t - t_0} \right) + \left(v^S, u^S\right) \left(\frac{x - x_0}{t - t_0}\right) - \left( v^*,  u^*\right), \qquad t>t_0.
\end{equation}

\subsection{Main result}

Now we can state the main result as follows.
\begin{theorem} \label{theorem1}
Let $(V, U)(t, x)$ in Figure 1 be the entropy solution of the Euler system \eqref{ES}-\eqref{ESI} defined in \eqref{shock-shock-1} and \eqref{rare-shock}. There exist positive constants $\varepsilon_0$ and $\delta_0$, such that for any $\varepsilon \in (0, \varepsilon_0)$ and $\delta_l \sim \delta_r \leq \delta_0$, the Cauchy problem of compressible Navier-Stokes equations \eqref{NS} 
admits a family of global smooth solutions $(v^\varepsilon, u^\varepsilon)(t, x)$ satisfying
\begin{align*}
	\begin{cases}
		(v^\varepsilon - V, u^\varepsilon - U) \in C^0([0, +\infty), L^2(\bbr)), \\
		(v^\varepsilon_x, u^\varepsilon_x) \in C^0([0, +\infty), L^2(\bbr))\cap  L^2(0, +\infty, H^1(\bbr)). \\
	\end{cases}
\end{align*}
Moreover, as $\varepsilon \rightarrow 0+$, 
$(v^\varepsilon, u^\varepsilon)(t, x)$ converges to $(V, U)(t, x)$ in $L^p(\mathbb{R})$ for any  $p\in [2, +\infty)$ with the following 
convergence rate 
\begin{align*}
	&\|(v^\varepsilon - V, u^\varepsilon - U)(t, \cdot)\|_{L^p(\bbr)} \\
	&\leq \begin{cases} C (\delta_l + \delta_r)^\frac{1}{2} \varepsilon^\frac{1}{p},  &\text{\rm if} \quad 0 \leq t \leq t_0, \\
		C (\delta_l + \delta_r)^\frac{1}{2} \varepsilon^\frac{1}{p} + C (\delta_l + \delta_r) (t - t_0)^\frac{1}{2p} \varepsilon^\frac{1}{2p},  &\text{\rm if}\quad t \geq t_0,
	\end{cases} 		
\end{align*}
where the positive constant $C$ is independent of $\varepsilon, \delta_l, \delta_r$ and $t$, but may depend on $v_\pm, v_*, u_\pm, u_*$.
\end{theorem}

\begin{remark}
	The methodology developed in this paper may be extended to two interacting shocks from the same or different characteristic families for non-isentropic compressible Navier-Stokes equations, although the wave interactions will be much more complicated and the contact discontinuity will be involved, which are left for future study.
\end{remark}

\begin{remark}
	For simplicity of presentation, we state and prove our main results in the Lagrangian coordinates; the corresponding results in the Eulerian coordinates hold as well.
\end{remark}

\begin{remark}
	We note that Chen and Perepelitsa \cite{CP} proved the convergence of solutions of the Navier-Stokes equations to a weak solution of the Euler equations by using the compensated compactness method. Our result in Theorem \ref{theorem1} shows that for a given entropy solution defined in \eqref{shock-shock-1} and \eqref{rare-shock}, we can construct a family of smooth solutions of the Navier-Stokes equations that tend to this entropy solution of the Euler equations with the convergence rate as the viscosity vanishes.
\end{remark}

\subsection{Main difficulties and the key ideas of the proof}

Compared with the case of two interacting shocks from different characteristic families in \cite{HWWY2015}, the collision time for two shocks from the same family is in a larger time scale that depends inversely on the wave strength, that is $t_0= \frac{1}{s_l - s_r}\sim\frac{1}{\delta_l}\sim\frac{1}{\delta_r}$, while the collision time $t_0=O(1)$ for the two shocks from different characteristic families. Therefore, the overtaking interaction considered in the present paper is stronger and more singular, which poses several critical difficulties for justifying the vanishing viscosity limit not only in the processes before and after shock collisions, but more crucially, at the collision point:

\begin{itemize}
\item {\bf Pre-collision phase:} The standard energy estimates cannot be closed uniformly in $\varepsilon$ up to the collision time $t_0$, as the cumulative effect of the wave interaction over the long time interval $[0, t_0 ]$ produces uncontrollable error terms.

\item {\bf Post-collision phase:} After collision, the wave pattern simultaneously contains both an outgoing shock and a rarefaction wave, requiring a delicate treatment of the composite wave structure.

\item {\bf Collision point:} The collision itself is highly singular, and the transition from two shocks to a shock-rarefaction composite must be carefully resolved.
\end{itemize}
To resolve these issues, we develop novel ingredients in our proof as follows.

\begin{itemize}
\item {\bf Approximation of collision time:} We decompose the whole time interval $\mathbb{R}^+$ into two distinct subintervals based on the collision time and introduce different techniques on each. Precisely, before the collision, we employ the anti-derivative method to ensure both viscous shock locations can be fixed up to the collision point $t_0$. However, the uniform $H^2$-estimates of the anti-derivative perturbations with respect to the viscosity can not be closed on the full pre-collision time interval $[0, t_0]$. Instead, the closure of these estimates is only guaranteed on the carefully constructed subinterval $\left[0, t_0 - A\frac{\varepsilon}{\delta^2}\right]$ with $A>0$ taken to be a suitably large constant. This motivates us to introduce an approximate collision time $t_0 - A\frac{\varepsilon}{\delta^2}$, so that we are able to complete the required energy bounds on the subinterval $\left[0, t_0 - A\frac{\varepsilon}{\delta^2}\right]$ in the space $H^2(\mathbb{R})$.

\item {\bf Treatment of composite wave:} After the approximate collision time $t_0 - A\frac{\varepsilon}{\delta^2}$, we always use the composite wave of the viscous shock and smooth rarefaction wave as the approximate wave pattern. Due to the appearance of both shock and rarefaction waves, we apply the relative entropy framework combined with the $a$-contraction method and the time-dependent shifts to obtain the uniform estimates with respect to the viscosity on the time interval $\left[t_0 - A\frac{\varepsilon}{\delta^2}, +\infty\right)$, which is partially inspired by Kang, Vasseur and Wang \cite{KVW} for the time-asymptotic stability of the superposition of viscous shock and rarefaction waves. 

\item {\bf Convergence estimates on the whole time interval $[0, +\infty)$:} In the rigorous verification of the vanishing viscosity limit, we decompose the full time interval $\mathbb{R}^+$ into three distinct subintervals. On the two outer intervals $\left[0, t_0 - A\frac{\varepsilon}{\delta^2}\right]$ and $[t_0, +\infty)$, the constructed approximate wave pattern is well matched to the exact entropy solution. By contrast, such a matching property no longer holds on the intermediate transition interval $\left[t_0 - A\frac{\varepsilon}{\delta^2}, t_0\right]$. Nevertheless, this transition interval has a width of order $O(\varepsilon)$, and through a delicate analysis, we show that the error between the approximate composite wave (consisting of viscous shock and smooth rarefaction waves) and the exact entropy solution \eqref{shock-shock-1} tends to zero with the desired convergence rate as $\varepsilon\to 0$. 
\end{itemize}
Finally, we succeed in justifying the vanishing viscosity limit on the whole time interval $\mathbb{R}^+$ and obtain a uniform convergence rate with respect to the viscosity coefficient in $L^p$ norm with $p\in [2, +\infty)$.

The rest of this paper is organized as follows.
In Section 2, we first construct the viscous shock waves and the smooth approximate rarefaction wave associated with the Euler system \eqref{ES}, and collect several known properties of these wave structures. Then we construct the approximate wave patterns before and after collision, and the shift function for the viscous shock.
Then in Section 3, we prove the global existence of a solution to \eqref{NS}  based on the uniform  {\it a priori} estimates before and after collision.
In Section 4, we prove our main Theorem \ref{theorem1}. 
In Section 5, we prove the {\it a priori} estimates before the collision time by the anti-derivative techniques.
In Section 6, we carry out the detailed {\it a priori} estimates after the collision by using  $a$-contraction method with time-dependent shifts based on the relative entropy.

{\bf Notation.} In this paper, $\|\cdot\|_l \ (l\in \mathbb{N})$ denotes the usual Sobolev norm for $H^l(\bbr)$, $\|\cdot\| := \|\cdot\|_0$ for $L^2(\bbr) := H^0(\bbr)$ and $\|\cdot\|_{L^p} \ (1\leq p\leq +\infty)$ the $L^p(\bbr)$ norm. And
$C$   denotes  a generic positive constant that does not depend on $\varepsilon, \delta_l, \delta_r$ or $t$, but may depend on $(v_\pm, u_\pm)$ and $(v_*, u_*)$.

\bigskip

\section{Approximate Wave Profile}
\setcounter{equation}{0}

In this section, we approximate the inviscid shock and rarefaction waves and their superpositions in the context of compressible Navier-Stokes equations \eqref{NS}.

First, for the sake of computational and expository convenience, we introduce a variable transformation:
\begin{equation} \label{new variable}
	\tau := t - t_0, \quad y := x - x_0.
\end{equation}
Under new variables $(\tau, y)$, the initial time is $\tau = -t_0$, and the collision time is $\tau = 0$.
Subsequently, we reformulate the problem in the new variables $(\tau, y)$.
For notational simplicity, we denote the solution $(v^\varepsilon, u^\varepsilon)(t,x)$ under the new variables $(\tau, y)$ as $(v, u)(\tau, y)$. Then the new unknown functions $(v, u)(\tau, y)$ satisfy the system
\begin{equation} \label{NS new}
	\begin{cases}
		\displaystyle v_\tau - u_y = 0,       \qquad \qquad \qquad y \in \bbr, ~\tau \geq -t_0, \\
		\displaystyle u_\tau + p(v)_y = \varepsilon \left(\frac{u_y}{v}\right)_y.
	\end{cases}
\end{equation}
And the entropy solution $(V, U)(t, x)$ of \eqref{ES}-\eqref{ESI} under the new variables $(\tau, y)$ is  denoted by $(V, U)(\tau, y)$. 
Before the collision time, i.e. $-t_0 \leq \tau \leq 0$, by  \eqref{shock-shock}  and  \eqref{shock-shock-1}, 
 one has
\begin{equation} \label{shock-shock new}
	(V, U)(\tau, y) = (v^S_l, u^S_l)(y-s_l\tau) + (v^S_r, u^S_r)(y-s_r\tau) - (v_*, u_*),
\end{equation}
where
\begin{equation*}
	(v^S_l, u^S_l)\left( y - s_l \tau \right) = \begin{cases}
		(v_-, u_-),     \quad y - s_l \tau < 0, \\
		(v_*, u_*),     \quad y - s_l \tau > 0,
	\end{cases}
\end{equation*}
and
\begin{equation*}
	(v^S_r, u^S_r)\left( y - s_r \tau \right) = \begin{cases}
		(v_*, u_*),     \quad y - s_r \tau < 0, \\
		(v_+, u_+),     \quad y - s_r \tau > 0.
	\end{cases}
\end{equation*}
After the collision time, i.e. $\tau > 0$, by \eqref{shock2}, \eqref{rare1}, and \eqref{rare-shock}, one has
\begin{align}
	\begin{aligned} \label{rare-shock new}
		(V, U)(\tau, y) &= \left( v^R, u^R \right)\left( \frac{y}{\tau} \right) + \left(v^S, u^S\right) \left(\frac{y}{\tau}\right) - \left( v^*,  u^*\right) \\
		&= \left( v^R, u^R \right)\left( \frac{y}{\tau} \right) + \left(v^S, u^S\right)(y - s_2\tau) - \left( v^*,  u^*\right),
	\end{aligned}	
\end{align}
where $\left( v^R, u^R \right)\left( \frac{y}{\tau} \right)$ is given by \eqref{rare1} and $\left(v^S, u^S\right)(y - s_2\tau)$ is given by \eqref{shock2}.

In the course of the proof, we observe that the uniform {\it a priori} estimates with respect to the viscosity can not be closed on the full pre-collision time interval $[-t_0, 0]$. Instead, the closure of these estimates is only guaranteed on the carefully constructed subinterval $\left[-t_0, - A\frac{\varepsilon}{\delta^2}\right]$ with $A>0$ taken to be a suitably large constant. This motivates us to introduce an approximate collision time $\tau = - A\frac{\varepsilon}{\delta^2}$ for the collision time $\tau =0$, so that we are able to complete the required energy bounds on the subinterval $\left[-t_0, - A\frac{\varepsilon}{\delta^2}\right]$.

In the rest of the paper, we carry out the estimates and analysis under the new variable $(\tau, y)$.
To begin with, it is necessary to provide smooth approximations for the wave pattern $(V, U)(\tau, y)$ before and after the approximate collision time. Specifically, the smooth approximations of the shock waves $(v^S_l, u^S_l)\left( y - s_l \tau \right), (v^S_r, u^S_r)\left( y - s_r \tau \right), \left(v^S, u^S\right)(y - s_2\tau)$ and the rarefaction wave $\left( v^R, u^R \right)\left( \frac{y}{\tau} \right)$ are presented as follows.

\subsection{Viscous shock wave}

We first consider the left 2-viscous shock wave $(v^s_l, u^s_l)(y-s_l\tau)$ connecting the states $(v_-, u_-)$ and $(v_*, u_*)$ with $(v_-, u_-) \in S_2(v_*, u_*)$, which satisfies the following system:
\begin{equation}  \label{shock smoothL}
	\begin{cases}
		\displaystyle -s_lv^s_{ly} - u^s_{ly} = 0, \\
		\displaystyle - s_lu^s_{ly} + p(v^s_l)_y = \varepsilon \left(\frac{u^s_{ly}}{v^s_l}\right)_y,
	\end{cases}
\end{equation}
along with the far-field conditions:
$$(v^s_l, u^s_l)(-\infty)=(v_-, u_-), \qquad (v^s_l, u^s_l)(+\infty)=(v_*, u_*).$$

Similarly, the right 2-viscous shock wave $(v^s_r, u^s_r)(y-s_r\tau)$ connecting $(v_*, u_*)$ and $(v_+, u_+)$ with $(v_*, u_*) \in S_2(v_+, u_+)$ satisfies
\begin{equation}  \label{shock smoothR}
	\begin{cases}
		\displaystyle -s_r v^s_{ry} - u^s_{ry} = 0, \\
		\displaystyle - s_r u^s_{ry} + p(v^s_r)_y = \varepsilon\left(\frac{u^s_{ry}}{v^s_r}\right)_y,
	\end{cases}
\end{equation}
along with the far-field conditions:
$$(v^s_r, u^s_r)(-\infty)=(v_*, u_*), \qquad (v^s_r, u^s_r)(+\infty)=(v_+, u_+).$$

Finally, the outgoing 2-viscous shock wave $(v^s, u^s)(y-s_2\tau)$ connecting $(v^*, u^*)$ and $(v_+, u_+)$ with $(v^*, u^*) \in S_2(v_+, u_+)$ satisfies
\begin{equation}  \label{shock smooth2}
	\begin{cases}
		\displaystyle -s_2v^s_y - u^s_y = 0, \\
		\displaystyle - s_2u^s_y + p(v^s)_y = \varepsilon \left(\frac{u^s_y}{v^s}\right)_y,
	\end{cases}
\end{equation}
along with the far-field conditions:
$$(v^s, u^s)(-\infty)=(v^*, u^*), \qquad (v^s, u^s)(+\infty)=(v_+, u_+).$$

The properties of the aforementioned 2-viscous shock waves are summarized in the following lemma. We refer the reader to \cite{KM} for the detailed proof.

\begin{lemma}\label{lemma-shock}
	
	For any given end state $(v_-, u_-) \in S_2(v_*, u_*)$, let $\delta_l := v_* - v_- \sim u_- - u_*$ denote the shock wave strength. Then, there exists a unique (up to a constant shift) solution $(v^s_l, u^s_l)(y-s_l \tau)$ to \eqref{shock smoothL}. Furthermore, this solution satisfies
	$$
	v^s_{ly} > 0, \quad u^s_{ly} = -s_l v^s_{ly} < 0,
	$$
	and
	\begin{align*}
		|(v^s_l-v_-, u^s_l-u_-)(y-s_l \tau)| &\leq C\delta_l e^{-C\frac{\delta_l}{\varepsilon}|y-s_l \tau|}, && \forall y-s_l \tau<0,\\[1mm]
		|(v^s_l-v_*, u^s_l-u_*)(y-s_l \tau)| &\leq C\delta_l e^{-C\frac{\delta_l}{\varepsilon}|y-s_l \tau|}, && \forall y-s_l \tau>0,\\[1mm]
		|(v^s_{ly}, u^s_{ly})(y-s_l \tau)| &\leq C\frac{\delta_l^2}{\varepsilon} e^{-C\frac{\delta_l}{\varepsilon}|y-s_l \tau|}, && \forall y\in\mathbb{R},\\[1mm]
		|(v^s_{lyy}, u^s_{lyy})(y-s_l \tau)| &\leq C\frac{\delta_l}{\varepsilon}|(v^s_{ly}, u^s_{ly})(y-s_l \tau)|, && \forall y\in\mathbb{R}.
	\end{align*}
	
	Similarly, for any given end state $(v_*, u_*) \in S_2(v_+, u_+)$, let $\delta_r := v_+ - v_* \sim u_* - u_+$ denote the shock wave strength. Then, there exists a unique (up to a constant shift) solution $(v^s_r, u^s_r)(y-s_r \tau)$ to \eqref{shock smoothR}. Furthermore, this solution satisfies
	$$
	v^s_{ry} > 0, \quad u^s_{ry} = -s_r v^s_{ry} < 0,
	$$
	and
	\begin{align*}
		|(v^s_r-v_*, u^s_r-u_*)(y-s_r \tau)| &\leq C\delta_r e^{-C\frac{\delta_r}{\varepsilon}|y-s_r \tau|}, && \forall y-s_r \tau<0,\\[1mm]
		|(v^s_r-v_+, u^s_r-u_+)(y-s_r \tau)| &\leq C\delta_r e^{-C\frac{\delta_r}{\varepsilon}|y-s_r \tau|}, && \forall y-s_r \tau>0,\\[1mm]
		|(v^s_{ry}, u^s_{ry})(y-s_r \tau)| &\leq C\frac{\delta_r^2}{\varepsilon} e^{-C\frac{\delta_r}{\varepsilon}|y-s_r \tau|}, && \forall y\in\mathbb{R},\\[1mm]
		|(v^s_{ryy}, u^s_{ryy})(y-s_r \tau)| &\leq C\frac{\delta_r}{\varepsilon}|(v^s_{ry}, u^s_{ry})(y-s_r \tau)|, && \forall y\in\mathbb{R}.
	\end{align*}
	
	Finally, for any given end state $(v^*, u^*) \in S_2(v_+, u_+)$, let $\delta_2 := v_+ - v^* \sim u^* - u_+$ denote the shock wave strength. Then, there exists a unique (up to a constant shift) solution $(v^s,u^s)(y-s_2 \tau)$ to \eqref{shock smooth2}. Furthermore, this solution satisfies
	$$
	v_y^s > 0, \quad u^s_y = -s_2 v_y^s < 0,
	$$
	and
	\begin{align*}
		|(v^s-v^*, u^s-u^*)(y-s_2 \tau)| &\leq C\delta_2 e^{-C\frac{\delta_2}{\varepsilon}|y-s_2 \tau|}, && \forall y-s_2 \tau<0,\\[1mm]
		|(v^s-v_+, u^s-u_+)(y-s_2 \tau)| &\leq C\delta_2 e^{-C\frac{\delta_2}{\varepsilon}|y-s_2 \tau|}, && \forall y-s_2 \tau>0,\\[1mm]
		|(v^s_y, u^s_y)(y-s_2 \tau)| &\leq C\frac{\delta_2^2}{\varepsilon} e^{-C\frac{\delta_2}{\varepsilon}|y-s_2 \tau|}, && \forall y\in\mathbb{R},\\[1mm]
		|(v^s_{yy}, u^s_{yy})(y-s_2 \tau)| &\leq C\frac{\delta_2}{\varepsilon}|(v^s_y, u^s_y)(y-s_2 \tau)|, && \forall y\in\mathbb{R}.
	\end{align*}
\end{lemma}

\subsection{Smooth approximate rarefaction wave}

We now construct the smooth approximate 1-rarefaction wave $(v^r, u^r)(\tau, y)$ by
\begin{align}
	\begin{aligned} \label{rare smooth}
		&\Lambda_1(v^r)(\tau, y) = w^r\left(\tau + A\frac{\varepsilon}{\delta^2}, y + s_2 A\frac{\varepsilon}{\delta^2} \right), \\
		&Z_1(v^r, u^r)(\tau, y) = Z_1(v_-, u_-) = Z_1(v^*, u^*),
	\end{aligned}
\end{align}
where $w^r(\tau, y)$ is the smooth solution to the Burgers equation:
\begin{equation*}
	\begin{cases}
		\displaystyle w_\tau + w w_y = 0, \\
		\displaystyle w(0, y) = w^r_0(y) = \frac{\Lambda_1(v^*) + \Lambda_1(v_-)}{2} + \frac{\Lambda_1(v^*) - \Lambda_1(v_-)}{2} \tanh \frac{y}{\kappa},
	\end{cases}
\end{equation*}
with $\kappa > 0$ being a small parameter depending on $\varepsilon$. In fact, throughout this paper, we choose $\kappa = \varepsilon$. One can verify that the approximate rarefaction wave $(v^r, u^r)(\tau, y)$ constructed above satisfies the following compressible Euler system:
\begin{equation} \label{euler rare}
	\begin{cases}
		\displaystyle v^r_\tau - u^r_y = 0,      \qquad \qquad y \in \bbr, ~\tau \geq -A\frac{\varepsilon}{\delta^2}, \\
		\displaystyle u^r_\tau + p(v^r)_y = 0,
	\end{cases}
\end{equation}
along with the initial data $(v_0^r, u_0^r)(y) := (v^r, u^r)\left(-A\frac{\varepsilon}{\delta^2}, y\right)$.

The main properties of the approximate 1-rarefaction wave are summarized in the following lemma (we refer to \cite{X-1} for the detailed proof).
\begin{lemma} \label{lemma-rare}
	Let $\delta_1 := v^* - v_- \sim u^* - u_-$ denote the strength of the rarefaction wave. Then, the smooth approximate 1-rarefaction wave $(v^r, u^r)(\tau, y)$ constructed in \eqref{rare smooth} satisfies the following properties:
	\begin{enumerate}
		\item $u^r_y = \frac{2}{\gamma + 1}w_y > 0$ and $v^r_y = \frac{1}{\sqrt{-p'(v^r)}}u^r_y > 0$ for all $y \in \bbr$ and $\tau \geq -A\frac{\varepsilon}{\delta^2}$.
		
		\item For all $\tau > -A\frac{\varepsilon}{\delta^2}$ and $p \in [1, + \infty]$, the following estimates hold:
		\begin{align*}
			\|(v^r_y, u^r_y)\|_{L^p} &\leq C \min\left( \delta_1\kappa^{-1+\frac{1}{p}}, \delta_1^\frac{1}{p} \left(\tau + A\frac{\varepsilon}{\delta^2}\right)^{-1 + \frac{1}{p}} \right), \\
			\|(v^r_{yy}, u^r_{yy})\|_{L^p} &\leq C \min\left( \delta_1\kappa^{-2 + \frac{1}{p}}, \kappa^{-1+\frac{1}{p}} \left(\tau + A\frac{\varepsilon}{\delta^2}\right)^{-1} \right), \\
			|u^r_{yy}(\tau, y)| &\leq C u^r_y(\tau, y).
		\end{align*}
		
		\item For any $y + s_2 A\frac{\varepsilon}{\delta^2} \geq 0$ and $\tau \geq -A\frac{\varepsilon}{\delta^2}$, it holds that:
		\begin{align*}
			|(v^r, u^r)(\tau, y) - (v^*, u^*)| &\leq C \delta_1 e^{-\frac{2}{\kappa}\left(|y + s_2 A\frac{\varepsilon}{\delta^2}| + |\Lambda_1(v^*)| \left(\tau + A\frac{\varepsilon}{\delta^2}\right)\right)}, \\
			|(v^r_y, u^r_y)(\tau, y)| &\leq C \delta_1 \kappa^{-1} e^{-\frac{2}{\kappa}\left(|y + s_2 A\frac{\varepsilon}{\delta^2}| + |\Lambda_1(v^*)| \left(\tau + A\frac{\varepsilon}{\delta^2}\right)\right)}.
		\end{align*}
		
		\item For any $y + s_2 A\frac{\varepsilon}{\delta^2} \leq \Lambda_1(v_-) \left(\tau + A\frac{\varepsilon}{\delta^2}\right)$ and $\tau \geq -A\frac{\varepsilon}{\delta^2}$, it holds that:
		\begin{align*}
			|(v^r, u^r)(\tau, y) - (v_-, u_-)| &\leq C \delta_1 e^{-\frac{2}{\kappa}\left|y + s_2 A\frac{\varepsilon}{\delta^2} - \Lambda_1(v_-) \left(\tau + A\frac{\varepsilon}{\delta^2}\right)\right|}, \\
			|(v^r_y, u^r_y)(\tau, y)| &\leq C \delta_1 \kappa^{-1} e^{-\frac{2}{\kappa}\left|y + s_2 A\frac{\varepsilon}{\delta^2} - \Lambda_1(v_-) \left(\tau + A\frac{\varepsilon}{\delta^2}\right)\right|}.
		\end{align*}
	\end{enumerate}
\end{lemma}

\subsection{Approximate wave pattern}

Before the approximate collision time $-A\frac{\varepsilon}{\delta^2}$, the approximate wave pattern $(\bar{v}, \bar{u})(\tau, y)$ consisting of the superposition of the left and right 2-viscous shock waves can be written as
\begin{equation} \label{approxi wave before}
	(\bar{v}, \bar{u})(\tau, y) := (v^s_l, u^s_l)(y-s_l\tau) + (v^s_r, u^s_r)(y-s_r\tau) - (v_*, u_*),
\end{equation}
which satisfies the system
\begin{equation}  \label{approxi wave before equ}
	\begin{cases}
		\displaystyle \bar{v}_\tau - \bar{u}_y = 0,      \quad\qquad\qquad \qquad \qquad\qquad y \in \bbr, ~\tau \in \left[-t_0, -A\frac{\varepsilon}{\delta^2}\right], \\
		\displaystyle \bar{u}_\tau + p(\bar{v})_y = \varepsilon \left(\frac{\bar{u}_y}{\bar{v}}\right)_y + F_{1y} + F_{2y},
	\end{cases}
\end{equation}
where
\begin{equation} \label{approxi wave before error}
	F_1 = \varepsilon \left(\frac{u^s_{ly}}{v^s_l} + \frac{u^s_{ry}}{v^s_r} - \frac{\bar{u}_y}{\bar{v}}\right), \qquad F_2 = p(\bar{v}) - p(v^s_l) - p(v^s_r) + p(v_*).
\end{equation}

Having introduced the approximate wave pattern \eqref{approxi wave before} before the approximate collision time, we now present the initial data for the Navier-Stokes equations \eqref{NS new} as follows:
\begin{equation} \label{NSI}
	(v, u)(-t_0, y) = (\bar{v}, \bar{u})(-t_0, y).
\end{equation}

After the approximate collision time $-A\frac{\varepsilon}{\delta^2}$, the asymptotic analysis of the Navier-Stokes solutions \eqref{NS new} near the composite 1-rarefaction and 2-shock wave requires introducing a time-dependent shift $X(\tau)$ to the viscous shock profile, i.e., $(v^s,u^s)(y-s_2\tau-X(\tau))$. Specifically, let 
$$ ((v^s)^{-X}, (u^s)^{-X})(y-s_2\tau) := (v^s,u^s)(y-s_2\tau - X(\tau)) $$
denote the shifted 2-viscous shock. In terms of the $(\tau, y)$ variables, this profile satisfies the system:
\begin{equation}  \label{shock smooth shift}
	\begin{cases}
		\displaystyle (v^s)^{-X}_\tau - (u^s)^{-X}_y + \dot{X}(\tau)(v^s)^{-X}_y = 0,       \qquad\qquad \qquad \qquad y \in \bbr, ~\tau \geq -A\frac{\varepsilon}{\delta^2}, \\
		\displaystyle (u^s)^{-X}_\tau + p((v^s)^{-X})_y + \dot{X}(\tau)(u^s)^{-X}_y = \varepsilon \left(\frac{(u^s)^{-X}_y}{(v^s)^{-X}}\right)_y,
	\end{cases}
\end{equation}
along with the far-field conditions:
$$((v^s)^{-X}, (u^s)^{-X})(-\infty)=(v^*, u^*), \qquad ((v^s)^{-X}, (u^s)^{-X})(+\infty)=(v_+, u_+).$$

Therefore, the approximate wave pattern $(\tilde{v}, \tilde{u})(\tau, y)$ after the approximate collision time $-A\frac{\varepsilon}{\delta^2}$ consisting of the superposition of the smooth approximate 1-rarefaction wave and the 2-viscous shock wave shifted by $X(\tau)$ can be written as
\begin{equation} \label{approxi wave after}
	(\tilde{v}, \tilde{u})(\tau, y) := (v^r(\tau, y) + (v^s)^{-X}(y-s_2\tau) - v^*, u^r(\tau, y) + (u^s)^{-X}(y-s_2\tau) - u^*),
\end{equation}
which satisfies the system
\begin{equation}  \label{approxi wave after equ}
	\begin{cases}
		\displaystyle \tilde{v}_\tau - \tilde{u}_y + \dot{X}(\tau)(v^s)^{-X}_y = 0,       \qquad\qquad \qquad \qquad\qquad y \in \bbr, ~\tau \geq -A\frac{\varepsilon}{\delta^2}, \\
		\displaystyle \tilde{u}_\tau + p(\tilde{v})_y + \dot{X}(\tau)(u^s)^{-X}_y = \varepsilon \left(\frac{\tilde{u}_y}{\tilde{v}}\right)_y + F_{3y} + F_{4y},
	\end{cases}
\end{equation}
where
\begin{equation} \label{approxi wave after error}
	F_3 = \varepsilon\frac{(u^s)^{-X}_y}{(v^s)^{-X}} - \varepsilon\frac{\tilde{u}_y}{\tilde{v}}, \qquad F_4 = p(\tilde{v}) - p(v^r) - p((v^s)^{-X}) + p(v^*).
\end{equation}

\subsection{Construction of shift}
The shift function $X(\tau)$ mentioned above can be constructed as follows
\begin{equation} \label{X}
	\begin{cases}
		\di \dot{X}(\tau) = -\frac{M}{\delta_2}\left(\int_{\bbr}a^{-X}(u^s)^{-X}_y(u-\tilde{u}) dy + \int_{\bbr}\frac{a^{-X}}{s_2}p'((v^s)^{-X})(v^s)^{-X}_y(u-\tilde{u}) dy \right), \\
		X\left(-A\frac{\varepsilon}{\delta^2}\right) = 0,
	\end{cases}
\end{equation}
where $a$ is a weight function defined in \eqref{weight}, and the constants are given by $M=\frac{5\alpha^*}{4(s^*)^2}$ with $\alpha^*=\frac{p''(v^*)}{2s^*}$ and $s^* = \sqrt{-p'(v^*)}$.

\bigskip
%
%
\section{Global Existence and Uniform Estimates} 
\setcounter{equation}{0}

In this section, we mainly present the global existence and uniform estimates of the solutions before and after the approximate collision time $-A\frac{\varepsilon}{\delta^2}$, along with the {\it a priori} estimates.
The wave interaction estimates and the weight function are also introduced in this section.

Before the approximate collision time $-A\frac{\varepsilon}{\delta^2}$, from \eqref{NS new}, \eqref{approxi wave before equ} and \eqref{NSI}, we set the perturbation of $(v, u)(\tau, y)$ around the approximate wave profile $(\bar{v}, \bar{u})(\tau, y)$ by
\begin{equation*}
\phi(\tau, y) := (v - \bar{v})(\tau, y), \quad \psi(\tau, y) := (u - \bar{u})(\tau, y), 
\end{equation*}
which satisfies the following system:
\begin{equation} \label{per before}
    \begin{cases}
    \displaystyle \phi_\tau - \psi_y = 0,        \qquad\qquad\quad\qquad\qquad \qquad \qquad\qquad y \in \bbr, ~\tau \in \left[-t_0, -A\frac{\varepsilon}{\delta^2}\right], \\
    \displaystyle \psi_\tau + (p(v) -p(\bar{v}))_y = \varepsilon \left(\frac{u_y}{v} - \frac{\bar{u}_y}{\bar{v}}\right)_y - F_{1y} - F_{2y},
    \end{cases}
\end{equation}
supplemented with the initial data
\begin{equation} \label{per beforeI}
  (\phi, \psi)(-t_0, y)=0.
\end{equation}

Define the anti-derivative variables
\begin{equation*}
	\Phi(\tau, y) = \int_{-\infty}^{y} \phi(\tau, y) dy, \quad \Psi(\tau, y) := \int_{-\infty}^{y} \psi(\tau, y) dy,
\end{equation*}
which satisfy the system
\begin{equation} \label{anti-per}
	\begin{cases}
		\displaystyle \Phi_\tau - \Psi_y = 0, \\
		\displaystyle \Psi_\tau + p(v) -p(\bar{v}) = \varepsilon \left(\frac{u_y}{v} - \frac{\bar{u}_y}{\bar{v}}\right) - F_1 - F_2.
	\end{cases}
\end{equation}
Given the initial data \eqref{per beforeI}, it is straightforward to see that the initial value of the anti-derivative system is
\begin{equation} \label{anti-perI}
	(\Phi, \Psi)(-t_0, y)=0.
\end{equation}
For convenience, we linearize \eqref{anti-per} around the approximate profile $(\bar{v}, \bar{u})$ to obtain
\begin{equation} \label{anti-per line}
	\begin{cases}
		\displaystyle \Phi_\tau - \Psi_y = 0, \\
		\displaystyle \Psi_\tau + p'(\bar{v})\Phi_y = \frac{\varepsilon}{\bar{v}}\Psi_{yy} + Q_1 + Q_2 - F_1 - F_2,
	\end{cases}
\end{equation}
where
\begin{equation} \label{anti-per error}
	Q_1 = \varepsilon \left(\frac{1}{v} - \frac{1}{\bar{v}}\right)\left( \Psi_{yy} + \bar{u}_y \right), \qquad Q_2 = -(p(v) -p(\bar{v}) - p'(\bar{v})\phi).
\end{equation}
We seek the solution to \eqref{anti-per line} in the functional space defined as follows:
\begin{equation*}
	\bar{\Pi}(I) = \left\lbrace (\Phi, \Psi) \in C(I, H^2), \Psi_y \in L^2(I, H^2) \right\rbrace,
\end{equation*}
where $I \subset \bbr$ is any interval.

In order to prove our main result, Theorem \ref{theorem1}, it is sufficient to prove the following results.

\subsection{Proof of global existence and uniform estimates before collision}
Before the approximate collision time $-A\frac{\varepsilon}{\delta^2}$, we can obtain the following global existence and uniform estimates.

\begin{theorem}[Global existence and uniform estimates before collision] \label{global estimate before}
	Given a suitably large constant $A$, there exist positive constants $\bar{\delta}_0$ and $\bar{\varepsilon}_0$, both independent of $\varepsilon$, such that if the wave strengths satisfy $\delta_l \sim \delta_r \sim \delta \leq \bar{\delta}_0$ and the viscosity coefficient satisfies $\varepsilon \leq \bar{\varepsilon}_0$, the perturbation problem \eqref{anti-per}-\eqref{anti-perI} admits a unique global solution $(\Phi, \Psi)(\tau, y) \in \bar{\Pi}\left( \left[-t_0, -A\frac{\varepsilon}{\delta^2}\right] \right)$ satisfying
	\begin{equation} \label{L2a before}
		\| (\Phi, \Psi)(\tau) \|^2 + \int_{-t_0}^{\tau}\int_{\bbr} |\bar{u}_y|\Psi^2 \,dy \,d\tau + \varepsilon \int_{-t_0}^{\tau} \| \Psi_y \|^2 \,d\tau
		\leq  C e^{-cA}\delta^{-1}\varepsilon^3,
	\end{equation}
	\begin{equation} \label{L2 before}
		\| (\phi, \psi)(\tau) \|^2 + \varepsilon \int_{-t_0}^{\tau} \| \psi_y \|^2 \,d\tau
		\leq  C e^{-cA}\delta\varepsilon,
	\end{equation}	
	\begin{equation} \label{H1 before}
		\|(\phi_y, \psi_y)(\tau)\|^2 + \varepsilon \int_{-t_0}^{\tau}\|\psi_{yy}\|^2 \,d\tau
		\leq C e^{-cA} \delta\varepsilon^{-1}.
	\end{equation}
	Moreover, it holds that
	\begin{equation} \label{L infty before}
		\|(\phi, \psi)(\tau)\|_{L^\infty} \leq C \|(\phi, \psi)(\tau)\|^{\frac{1}{2}} \|(\phi_y, \psi_y)(\tau)\|^{\frac{1}{2}} \leq C e^{-\frac{cA}{2}}\delta^{\frac{1}{2}}.
	\end{equation}
	Here, the positive constants $C, c$ and $A$ are independent of $\varepsilon, \delta_l, \delta_r$, and $\tau$, but may depend on the states $v_\pm, u_\pm$ and $v_*, u_*$.
\end{theorem}

The local existence and uniqueness of the system \eqref{anti-per}-\eqref{anti-perI} are well established; the detailed proof can be found in \cite{GX, W}, hence we omit it here for brevity. Therefore, to prove Theorem \ref{global estimate before}, it suffices to close the following \textit{a priori} assumption:
\begin{equation} \label{priori assump before}
	\|\Psi(\tau)\|_{L^\infty(\bbr)} \leq e^{-c\sqrt{A}}\varepsilon, \qquad \|\phi(\tau)\|_{L^\infty(\bbr)} \leq e^{-c\sqrt{A}}\delta^{\frac{1}{2}},
\end{equation}
for $\tau \in [-t_0, \bar{\tau}]$ with some $\bar{\tau} \in \left[-t_0, -A\frac{\varepsilon}{\delta^2}\right]$, where $[-t_0, \bar{\tau}]$ is the time interval on which the solution is assumed to exist. Through rigorous mathematical analysis, we derive the following uniform \textit{a priori} estimates.

\begin{proposition}[Uniform \textit{a priori} estimates before collision] \label{priori estimate before}
	Suppose that the Cauchy problem \eqref{anti-per}-\eqref{anti-perI} admits a solution $(\Phi, \Psi)(\tau, y) \in \bar{\Pi}\left( [-t_0, \bar{\tau}] \right)$ for some $\bar{\tau} \in \left[-t_0, -A\frac{\varepsilon}{\delta^2}\right]$ and satisfies the \textit{a priori} assumption \eqref{priori assump before}. Then there exist positive constants $\bar{\delta}_0$ and $\bar{\varepsilon}_0$, independent of $\varepsilon$, such that if $\delta_l \sim \delta_r \sim \delta \leq \bar{\delta}_0$ and $\varepsilon \leq \bar{\varepsilon}_0$, the inequalities \eqref{L2a before}-\eqref{L infty before} hold.
\end{proposition}

\noindent\textbf{Proof of Proposition \ref{priori estimate before}:} 
The detailed derivations of the necessary lemmas are deferred to Section 5. Here, we outline how to deduce Proposition \ref{priori estimate before} from these lemmas.

First, combining the estimates from Lemmas \ref{lemma anti} and \ref{lemma anti1} in Section 5, we have
\begin{align}
	\begin{aligned} \label{estimate anti-0}
		&\| (\Phi, \Psi)(\tau) \|^2 + \int_{-t_0}^{\tau}\int_{\bbr} |\bar{u}_y|\Psi^2 \,dy \,d\tau + \varepsilon \int_{-t_0}^{\tau} \| \Psi_y \|^2 \,d\tau \\
		&\leq C e^{-cA}\delta^{-1}\varepsilon^3 + C\left(\delta^2 + e^{-c\sqrt{A}}\right)\left( \| \Psi(\tau) \|^2 + \varepsilon \int_{-t_0}^{\tau} \| \Psi_y \|^2 \,d\tau \right) + C e^{-c\sqrt{A}}\varepsilon^3 \int_{-t_0}^{\tau} \|\psi_y\|^2 \,d\tau.
	\end{aligned}
\end{align}
By choosing a suitably small constant $\delta$ and a suitably large constant $A$, we can absorb the terms on the right-hand side of \eqref{estimate anti-0} to obtain:
\begin{align}
	\begin{aligned} \label{estimate anti-1}
		&\| (\Phi, \Psi)(\tau) \|^2 + \int_{-t_0}^{\tau}\int_{\bbr} |\bar{u}_y|\Psi^2 \,dy \,d\tau + \varepsilon \int_{-t_0}^{\tau} \| \Psi_y \|^2 \,d\tau \\
		&\leq C e^{-cA}\delta^{-1}\varepsilon^3 + C e^{-c\sqrt{A}}\varepsilon^3 \int_{-t_0}^{\tau} \|\psi_y\|^2 \,d\tau,
	\end{aligned}
\end{align}
as well as
\begin{equation} \label{estimate anti1-1}
	\varepsilon \|\Phi_y(\tau)\|^2 + \int_{-t_0}^{\tau} \|\Phi_y\|^2 \,d\tau
	\leq C e^{-cA} \delta^{-1} \varepsilon^2 + C e^{-c\sqrt{A}} \delta\varepsilon^2 \int_{-t_0}^{\tau} \|\psi_y\|^2 \,d\tau.
\end{equation}
Subsequently, coupling the bounds \eqref{estimate anti-1} and \eqref{estimate anti1-1} with \eqref{estimate 0before} from Lemma \ref{lemma 0before}, and applying the smallness of $\delta$ alongside the largeness of $A$, we deduce:
\begin{equation} \label{estimate anti-2}
	\| (\Phi, \Psi)(\tau) \|^2 + \int_{-t_0}^{\tau}\int_{\bbr} |\bar{u}_y|\Psi^2 \,dy \,d\tau + \varepsilon \int_{-t_0}^{\tau} \| \Psi_y \|^2 \,d\tau
	\leq  C e^{-cA}\delta^{-1}\varepsilon^3,
\end{equation}
\begin{equation} \label{estimate anti1-2}
	\varepsilon \|\Phi_y(\tau)\|^2 + \int_{-t_0}^{\tau} \|\Phi_y\|^2 \,d\tau
	\leq C e^{-cA} \delta^{-1} \varepsilon^2,
\end{equation}
and
\begin{equation} \label{estimate0before-1}
	\|(\phi, \psi)(\tau)\|^2 + \varepsilon \int_{-t_0}^{\tau} \|\psi_y\|^2 \,d\tau \leq C e^{-cA}\delta\varepsilon.
\end{equation}
With these bounds established, we readily obtain the derivative estimates:
\begin{equation} \label{estimate phi1before-1}
	\| \phi_y(\tau) \|^2 + \varepsilon^{-1}\int_{-t_0}^{\tau} \|\phi_y\|^2 \,d\tau 
	\leq C e^{-cA}\delta\varepsilon^{-1},	  
\end{equation}
\begin{equation} \label{estimate psi1before-1}
	\| \psi_y(\tau) \|^2 + \varepsilon \int_{-t_0}^{\tau}\|\psi_{yy}\|^2 \,d\tau \leq C e^{-cA}\delta\varepsilon^{-1}.
\end{equation}
This completes the proof of the \textit{a priori} estimates stated in Proposition \ref{priori estimate before}. \qed

\begin{remark}
	It is crucial to note that, the uniform estimates with respect to the viscosity in Proposition \ref{priori estimate before} can not be closed throughout the entire interval $[-t_0, 0]$ before the interacting time, but can be closed on the subinterval $\left[-t_0, -A\frac{\varepsilon}{\delta^2}\right]$ with $A>0$ being a suitably large constant. In other words, we need to introduce a carefully constructed approximate collision time, namely $- A\frac{\varepsilon}{\delta^2}$, to achieve the closure. This is fundamentally different from the interaction of shock waves between different characteristic families in \cite{HWWY2015}, where the {\it a priori} assumptions can be closed directly on the whole interval $[-t_0, 0]$ without resorting to an approximation at the collision time $\tau=0$.
\end{remark}

\subsection{Wave interaction estimates before collision}
We now establish the wave interaction estimates before the collision, which will be utilized to bound the error terms in Section 5.

\begin{lemma}\label{lem:shock-interact}
	There exist positive constants $\bar\delta_0$ and $C$ such that for any $\delta_l, \delta_r \in (0, \bar\delta_0)$ and all $\tau \in [-t_0, 0]$, the following estimates hold for $i=l, r$:
	\begin{align*}
		|v^s_{iy}| |\bar{v} - v^s_i| &\leq C \frac{\delta_i\delta_l\delta_r}{\varepsilon} \left(e^{\frac{C \delta_l^2}{\varepsilon} \tau} + e^{\frac{C \delta_r^2}{\varepsilon} \tau} \right), \qquad \forall y \in \mathbb{R}, \\
		\int_{\bbr} |\bar{v} - v^s_l| |\bar{v} - v^s_r| \,dy &\leq C \delta_r\varepsilon e^{\frac{C \delta_l^2}{\varepsilon} \tau} + C \delta_l\varepsilon e^{\frac{C \delta_r^2}{\varepsilon} \tau}, \\
		\int_{\bbr} |v^s_{iy}| |\bar{v} - v^s_i| \,dy &\leq C \delta_l \delta_r \left( e^{\frac{C \delta_l^2}{\varepsilon} \tau} + e^{\frac{C \delta_r^2}{\varepsilon} \tau} \right), \\
		\int_{\bbr} |v^s_{iy}|^2 |\bar{v} - v^s_i|^2 \,dy &\leq C \frac{\delta_i\delta_l^2 \delta_r^2}{\varepsilon} \left( e^{\frac{C \delta_l^2}{\varepsilon} \tau} + e^{\frac{C \delta_r^2}{\varepsilon} \tau} \right), \\
		\int_{\bbr} |v^s_{ly}| |v^s_{ry}| \,dy &\leq C \frac{\delta_l^2\delta_r}{\varepsilon}e^{\frac{C \delta_l^2}{\varepsilon} \tau} + C \frac{\delta_l\delta_r^2}{\varepsilon} e^{\frac{C \delta_r^2}{\varepsilon} \tau}, \\
		\int_{\bbr} |v^s_{ly} v^s_{ry}|^2 \,dy &\leq C \frac{\delta_l^4\delta_r^3}{\varepsilon^3}e^{\frac{C \delta_l^2}{\varepsilon} \tau} + C \frac{\delta_l^3\delta_r^4}{\varepsilon^3} e^{\frac{C \delta_r^2}{\varepsilon} \tau}.
	\end{align*}
\end{lemma}

\begin{proof}
	We only consider the case $i=l$, as the case $i=r$ can be shown analogously. It follows from Lemma \ref{lemma-shock} that
	\begin{equation}\label{vly}
		|v^s_{ly}(y - s_l \tau)| \leq C \frac{\delta_l^2}{\varepsilon} e^{- \frac{C \delta_l}{\varepsilon} |y - s_l \tau|}, \qquad \forall y \in \mathbb{R}, ~\tau \in [-t_0, 0],
	\end{equation}
	\begin{equation}\label{vr}
		|\bar{v} - v^s_l| = |v^s_r - v_*| \leq \begin{cases}
			C \delta_r e^{- \frac{C \delta_r}{\varepsilon} |y - s_r \tau|}, & \text{for } y \leq s_r \tau, \\[1mm]
			C \delta_r, & \text{for } y \geq s_r \tau,
		\end{cases}
	\end{equation}
	and
	\begin{equation}\label{vl}
		|\bar{v} - v^s_r| = |v^s_l - v_*| \leq \begin{cases}
			C \delta_l e^{- \frac{C \delta_l}{\varepsilon} |y - s_l \tau|}, & \text{for } y \leq s_l \tau, \\[1mm]
			C \delta_l, & \text{for } y \geq s_l \tau.
		\end{cases}
	\end{equation}
	Combining \eqref{vly}, \eqref{vr}, and \eqref{vl} yields
	\begin{equation} \label{vlvr}
		|\bar{v} - v^s_l| |\bar{v} - v^s_r| \leq \begin{cases}
			C \delta_l\delta_r e^{- \frac{C \delta_r}{\varepsilon} |y - s_r \tau|} , & \text{for } y \leq s_*\tau \leq s_r \tau, \\[1mm]
			C \delta_l\delta_r e^{- \frac{C \delta_l}{\varepsilon} |y - s_l \tau|}, & \text{for } y \geq s_*\tau \geq s_l \tau,
		\end{cases}
	\end{equation}
	\begin{equation} \label{vlyvr}
		|v^s_{ly}| |\bar{v} - v^s_l| \leq \begin{cases}
			C \frac{\delta_l^2\delta_r}{\varepsilon} e^{- \frac{C \delta_r}{\varepsilon} |y - s_r \tau|} , & \text{for } y \leq s_*\tau \leq s_r \tau, \\[1mm]
			C \frac{\delta_l^2\delta_r}{\varepsilon} e^{- \frac{C \delta_l}{\varepsilon} |y - s_l \tau|}, & \text{for } y \geq s_*\tau \geq s_l \tau,
		\end{cases}
	\end{equation}
	and
	\begin{equation} \label{vlyvr1}
		|v^s_{ly}|^{\frac{1}{2}} |\bar{v} - v^s_l| \leq \begin{cases}
			C \frac{\delta_l\delta_r}{\sqrt{\varepsilon}} e^{- \frac{C \delta_r}{\varepsilon} |y - s_r \tau|}, & \text{for } y \leq s_*\tau \leq s_r \tau, \\[1mm]
			C \frac{\delta_l\delta_r}{\sqrt{\varepsilon}} e^{- \frac{C \delta_l}{\varepsilon} |y - s_l \tau|}, & \text{for } y \geq s_*\tau \geq s_l \tau.
		\end{cases}
	\end{equation}
	Noting that for $\tau \leq 0$,
	\begin{equation} \label{shock-sep}
		\begin{aligned} 
			y - s_r \tau \leq (s_* - s_r) \tau \leq \frac{1}{2}(s_* - s_r) \tau \leq 0, \qquad & \text{for } y \leq s_*\tau, \\	 		
			y - s_l \tau \geq (s_* - s_l) \tau \geq \frac{1}{2}(s_* - s_l) \tau \geq 0, \qquad & \text{for } y \geq s_*\tau,
		\end{aligned}	 		
	\end{equation}
	substituting \eqref{shock-sep} into \eqref{vlvr} and \eqref{vlyvr} provides
	\begin{equation} \label{vlvr-1}
		|\bar{v} - v^s_l| |\bar{v} - v^s_r| \leq \begin{cases}
			C \delta_l\delta_r e^{- \frac{C \delta_r}{\varepsilon} |y - s_r \tau|} e^{- \frac{C \delta_r}{\varepsilon} (s_r - s_*) \tau} , & \text{for } y \leq s_*\tau \leq s_r \tau, \\[1mm]
			C \delta_l\delta_r e^{- \frac{C \delta_l}{\varepsilon} |y - s_l \tau|} e^{- \frac{C \delta_l}{\varepsilon} (s_* - s_l) \tau}, & \text{for } y \geq s_*\tau \geq s_l \tau,
		\end{cases}
	\end{equation}
	\begin{equation} \label{vlyvr-1}
		|v^s_{ly}| |\bar{v} - v^s_l| \leq \begin{cases}
			C \frac{\delta_l^2\delta_r}{\varepsilon} e^{- \frac{C \delta_r}{\varepsilon} |y - s_r \tau|} e^{- \frac{C \delta_r}{\varepsilon} (s_r - s_*) \tau} , & \text{for } y \leq s_*\tau \leq s_r \tau, \\[1mm]
			C \frac{\delta_l^2\delta_r}{\varepsilon} e^{- \frac{C \delta_l}{\varepsilon} |y - s_l \tau|} e^{- \frac{C \delta_l}{\varepsilon} (s_* - s_l) \tau}, & \text{for } y \geq s_*\tau \geq s_l \tau.
		\end{cases}
	\end{equation}
	Hence, \eqref{vlvr-1} and \eqref{vlyvr-1} directly yield the first two desired estimates:
	\begin{align*}
		|v^s_{ly}| |\bar{v} - v^s_l| &\leq C \frac{\delta_l^2\delta_r}{\varepsilon} \left(e^{\frac{C \delta_l^2}{\varepsilon} \tau} + e^{\frac{C \delta_r^2}{\varepsilon} \tau} \right), \\
		\int_{\bbr} |\bar{v} - v^s_l| |\bar{v} - v^s_r| \,dy &\leq C \delta_r\varepsilon e^{\frac{C \delta_l^2}{\varepsilon} \tau} + C \delta_l\varepsilon e^{\frac{C \delta_r^2}{\varepsilon} \tau}.
	\end{align*}
	
	Similarly, applying \eqref{shock-sep} to \eqref{vlyvr1} gives
	$$|v^s_{ly}|^\frac{1}{2} |\bar{v} - v^s_l| \leq C \frac{\delta_l\delta_r}{\sqrt{\varepsilon}} \left(e^{\frac{C \delta_l^2}{\varepsilon} \tau} + e^{\frac{C \delta_r^2}{\varepsilon} \tau} \right).$$
	Since Lemma \ref{lemma-shock} ensures $\int_{\bbr} |v^s_{ly}|^{\frac{1}{2}} \,dy \leq C \sqrt{\varepsilon}$, it follows that
	$$
	\int_{\bbr} |v^s_{ly}| |\bar{v} - v^s_l| \,dy \leq C \delta_l \delta_r \left( e^{\frac{C \delta_l^2}{\varepsilon} \tau} + e^{\frac{C \delta_r^2}{\varepsilon} \tau} \right),
	$$
	which establishes the third estimate. 
	
	Finally, by Lemma \ref{lemma-shock} and \eqref{shock-sep}, we have
	\begin{equation*}
		|v^s_{ly}| |v^s_{ry}| \leq \begin{cases}
			C \frac{\delta_l^2\delta_r^2}{\varepsilon^2} e^{- \frac{C \delta_l}{\varepsilon} |y - s_l \tau|} e^{ \frac{C \delta_r^2}{\varepsilon} \tau} , & \text{for } y \leq s_*\tau \leq s_r \tau, \\[1mm]
			C \frac{\delta_l^2\delta_r^2}{\varepsilon^2} e^{- \frac{C \delta_r}{\varepsilon} |y - s_r \tau|} e^{ \frac{C \delta_l^2}{\varepsilon} \tau}, & \text{for } y \geq s_*\tau \geq s_l \tau.
		\end{cases}
	\end{equation*}
	Integrating this inequality over $\bbr$ yields the fifth estimate:
	$$
	\int_{\bbr} |v^s_{ly}| |v^s_{ry}| \,dy \leq C \frac{\delta_l^2\delta_r}{\varepsilon}e^{\frac{C \delta_l^2}{\varepsilon} \tau} + C \frac{\delta_l\delta_r^2}{\varepsilon} e^{\frac{C \delta_r^2}{\varepsilon} \tau}.
	$$
	The remaining estimates (involving squared terms) can be deduced in an identical manner.
\end{proof}

After the approximate collision time $-A\frac{\varepsilon}{\delta^2}$, from \eqref{NS new} and \eqref{approxi wave after equ}, we set the perturbation of $(v, u)(\tau, y)$ around the approximate wave profile 
$(\tilde{v}, \tilde{u})(\tau, y)$ by
\begin{equation*}
	\tilde{\phi}(\tau, y) := (v - \tilde{v})(\tau, y), \quad \tilde{\psi}(\tau, y) := (u - \tilde{u})(\tau, y), 
\end{equation*}
which satisfies the following system:
\begin{equation} \label{per after}
	\begin{cases}
		\displaystyle \tilde{\phi}_\tau - \tilde{\psi}_y - \dot{X}(\tau)(v^s)_y^{-X} = 0,   \qquad\qquad \qquad \qquad\qquad y \in \bbr, ~\tau \geq -A\frac{\varepsilon}{\delta^2}, \\
		\displaystyle \tilde{\psi}_\tau + (p(v) -p(\tilde{v}))_y - \dot{X}(\tau)(u^s)_y^{-X} = \varepsilon\left(\frac{u_y}{v} - \frac{\tilde{u}_y}{\tilde{v}}\right)_y - F_{3y} - F_{4y},
	\end{cases}
\end{equation}
supplemented with the initial data
\begin{equation} \label{per afterI}
	(\tilde{\phi}, \tilde{\psi})\left(-A\frac{\varepsilon}{\delta^2}, y\right) = (\tilde{\phi}_0, \tilde{\psi}_0)(y).
\end{equation}

Unlike the previous part where the initial data at $\tau = - t_0$ is well-chosen such that the perturbation $(\phi, \psi)(\tau, y)\big|_{\tau = -t_0} = (0, 0)$, here the initial data $(v, u)\left(-A\frac{\varepsilon}{\delta^2}, y\right)$ for the system \eqref{NS new} is determined by the solution obtained before the collision time and can not be chosen any more. 

However, we still need the smallness of the perturbation around $(\tilde{v}, \tilde{u})$ at $\tau =-A\frac{\varepsilon}{\delta^2}$. Because we already have
\begin{equation} \label{hit before}
	\left\|(v - \bar{v}, u - \bar{u})\left(-A\frac{\varepsilon}{\delta^2}\right)\right\|^2 \leq C (\delta_l + \delta_r)\varepsilon, ~ \left\| ((v - \bar{v})_y, (u - \bar{u})_y)\left(-A\frac{\varepsilon}{\delta^2}\right) \right\|^2 \leq C(\delta_l + \delta_r)\varepsilon^{-1},
\end{equation}
in order to prove the smallness of $(v - \tilde{v}, u - \tilde{u})\left(-A\frac{\varepsilon}{\delta^2}, y\right)$, it suffices to estimate the difference between the approximate wave profiles before and after the approximate collision time $-A\frac{\varepsilon}{\delta^2}$, i.e.,
\begin{align*}
	&(\tilde{v} - \bar{v})\left(-A\frac{\varepsilon}{\delta^2}, y\right) \\
	&= \left(v^r\left(-A\frac{\varepsilon}{\delta^2}, y\right) + v^s\left(y + s_2A\frac{\varepsilon}{\delta^2}\right) - v^*\right) - \left(v^s_l\left(y + s_lA\frac{\varepsilon}{\delta^2}\right) + v^s_r\left(y + s_rA\frac{\varepsilon}{\delta^2}\right) - v_*\right) \\
	&=\begin{cases} \left(v^r\left(-A\frac{\varepsilon}{\delta^2}, y\right) - v_-\right) + \left(v^s\left(y + s_2A\frac{\varepsilon}{\delta^2}\right) - v^*\right)& \\
		\quad - \left(v^s_l\left(y + s_lA\frac{\varepsilon}{\delta^2}\right) - v_-\right) - \left(v^s_r\left(y + s_rA\frac{\varepsilon}{\delta^2}\right) - v_*\right),& \text{for } y + s_2A\frac{\varepsilon}{\delta^2} \leq 0, \\
		\left(v^r\left(-A\frac{\varepsilon}{\delta^2}, y\right) - v^*\right) + \left(v^s\left(y + s_2A\frac{\varepsilon}{\delta^2}\right) - v_+\right)& \\
		\quad - \left(v^s_l\left(y + s_lA\frac{\varepsilon}{\delta^2}\right) - v_*\right) - \left(v^s_r\left(y + s_rA\frac{\varepsilon}{\delta^2}\right) - v_+\right),& \text{for } y + s_2A\frac{\varepsilon}{\delta^2} \geq 0,\end{cases}
\end{align*}
where we notice that $X\left(-A\frac{\varepsilon}{\delta^2}\right) = 0$.

\begin{lemma}
	It holds that
	\begin{equation} \label{hit}
		\begin{aligned} 
			\left\|(\tilde{v} - \bar{v}, \tilde{u} - \bar{u})\left(-A\frac{\varepsilon}{\delta^2}\right)\right\|^2 &\leq C (\delta_l + \delta_r)\varepsilon, \\
			\left\| \left((\tilde{v} - \bar{v})_y, (\tilde{u} - \bar{u})_y\right)\left(-A\frac{\varepsilon}{\delta^2}\right) \right\|^2 &\leq C(\delta_l + \delta_r)^3\varepsilon^{-1}.
		\end{aligned}
	\end{equation}
\end{lemma}

\begin{proof}
	By Lemma \ref{lemma-rare}, we have
	\begin{align*}
		\left|v^r\left(-A\frac{\varepsilon}{\delta^2}, y\right) - v_-\right| &\leq C \delta_1 e^{-\frac{2}{\kappa}\left|y + s_2A\frac{\varepsilon}{\delta^2}\right|}, && \text{for } y + s_2A\frac{\varepsilon}{\delta^2} \leq 0,\\[1mm]
		\left|v^r\left(-A\frac{\varepsilon}{\delta^2}, y\right) - v^*\right| &\leq C \delta_1 e^{-\frac{2}{\kappa}\left|y + s_2A\frac{\varepsilon}{\delta^2}\right|}, && \text{for } y + s_2A\frac{\varepsilon}{\delta^2} \geq 0.
	\end{align*}
	Therefore,
	\begin{align*}
		&\int_{-\infty}^{-s_2A\frac{\varepsilon}{\delta^2}} \left|v^r\left(-A\frac{\varepsilon}{\delta^2}, y\right) - v_-\right|^2 \,dy + \int_{-s_2A\frac{\varepsilon}{\delta^2}}^{+\infty} \left|v^r\left(-A\frac{\varepsilon}{\delta^2}, y\right) - v^*\right|^2 \,dy \\
		&\leq C \delta_1^2 \int_{\bbr} e^{-\frac{C}{\kappa}\left|y + s_2A\frac{\varepsilon}{\delta^2}\right|} \,dy \leq C \delta_1^2 \kappa = C \delta_1^2 \varepsilon.
	\end{align*}	
	It follows from Lemma \ref{lemma-shock} that
	\begin{align*}
		\left|v^s\left(y + s_2A\frac{\varepsilon}{\delta^2}\right) - v^*\right| &\leq C \delta_2 e^{-\frac{C\delta_2}{\varepsilon}\left|y + s_2A\frac{\varepsilon}{\delta^2}\right|}, && \text{for } y + s_2A\frac{\varepsilon}{\delta^2} \leq 0, \\[1mm]
		\left|v^s\left(y + s_2A\frac{\varepsilon}{\delta^2}\right) - v_+\right| &\leq C \delta_2 e^{-\frac{C\delta_2}{\varepsilon}\left|y + s_2A\frac{\varepsilon}{\delta^2}\right|}, && \text{for } y + s_2A\frac{\varepsilon}{\delta^2} \geq 0,
	\end{align*}
	which implies
	\begin{align*}
		&\int_{-\infty}^{-s_2A\frac{\varepsilon}{\delta^2}} \left|v^s\left(y + s_2A\frac{\varepsilon}{\delta^2}\right) - v^*\right|^2 \,dy + \int_{- s_2A\frac{\varepsilon}{\delta^2}}^{+\infty} \left|v^s\left(y + s_2A\frac{\varepsilon}{\delta^2}\right) - v_+\right|^2 \,dy \\ 
		&\leq C \delta_2^2 \int_{\bbr} e^{-\frac{C\delta_2}{\varepsilon}\left|y + s_2A\frac{\varepsilon}{\delta^2}\right|} \,dy \leq C \delta_2 \varepsilon.
	\end{align*}
	
	Similarly, for the incoming shocks evaluated at $\tau = -A\frac{\varepsilon}{\delta^2}$, we can obtain
	\begin{align*}
		\left|v^s_l\left(y + s_lA\frac{\varepsilon}{\delta^2}\right) - v_-\right| &\leq C \delta_l e^{-\frac{C\delta_l}{\varepsilon}\left|y + s_lA\frac{\varepsilon}{\delta^2}\right|}, && \text{for } y + s_lA\frac{\varepsilon}{\delta^2} \leq 0, \\[1mm]
		\left|v^s_l\left(y + s_lA\frac{\varepsilon}{\delta^2}\right) - v_*\right| &\leq C \delta_l e^{-\frac{C\delta_l}{\varepsilon}\left|y + s_lA\frac{\varepsilon}{\delta^2}\right|}, && \text{for } y + s_lA\frac{\varepsilon}{\delta^2} \geq 0.
	\end{align*}
	Note that $s_l > s_2$, which implies $-s_l < -s_2$. Splitting the integrals properly yields
	\begin{align*}
		&\int_{-\infty}^{- s_2A\frac{\varepsilon}{\delta^2}} \left|v^s_l\left(y + s_lA\frac{\varepsilon}{\delta^2}\right) - v_-\right|^2 \,dy + \int_{- s_2A\frac{\varepsilon}{\delta^2}}^{+\infty} \left|v^s_l\left(y + s_lA\frac{\varepsilon}{\delta^2}\right) - v_*\right|^2 \,dy \\
		&= \int_{-\infty}^{- s_lA\frac{\varepsilon}{\delta^2}} \left|v^s_l\left(y + s_lA\frac{\varepsilon}{\delta^2}\right) - v_-\right|^2 \,dy + \int_{- s_lA\frac{\varepsilon}{\delta^2}}^{+\infty} \left|v^s_l\left(y + s_lA\frac{\varepsilon}{\delta^2}\right) - v_*\right|^2 \,dy \\
		&\quad + \int_{-s_lA\frac{\varepsilon}{\delta^2}}^{- s_2A\frac{\varepsilon}{\delta^2}} \left|v^s_l\left(y + s_lA\frac{\varepsilon}{\delta^2}\right) - v_-\right|^2 \,dy - \int_{-s_lA\frac{\varepsilon}{\delta^2}}^{- s_2A\frac{\varepsilon}{\delta^2}} \left|v^s_l\left(y + s_lA\frac{\varepsilon}{\delta^2}\right) - v_*\right|^2 \,dy \\
		&\leq C \delta_l^2 \int_{\bbr} e^{-\frac{C\delta_l}{\varepsilon}\left|y + s_lA\frac{\varepsilon}{\delta^2}\right|} \,dy + C\delta_l^2|s_l - s_2|A \frac{\varepsilon}{\delta^2} \\
		&\leq C (\delta_l + \delta_r) \varepsilon.
	\end{align*}
	By an analogous procedure for the right shock, we have
	\begin{equation*}
		\int_{-\infty}^{- s_2A\frac{\varepsilon}{\delta^2}} \left|v^s_r\left(y + s_rA\frac{\varepsilon}{\delta^2}\right) - v_*\right|^2 \,dy + \int_{- s_2A\frac{\varepsilon}{\delta^2}}^{+\infty} \left|v^s_r\left(y + s_rA\frac{\varepsilon}{\delta^2}\right) - v_+\right|^2 \,dy \leq C (\delta_l + \delta_r) \varepsilon.
	\end{equation*}
	
	Finally, a direct computation for the derivative estimates yields
	\begin{equation*}
		\left\| \left((\tilde{v} - \bar{v})_y, (\tilde{u} - \bar{u})_y\right)\left(- A\frac{\varepsilon}{\delta^2}\right) \right\|^2 \leq C \delta_1^2\varepsilon^{-1} + C(\delta_l + \delta_r)^3\varepsilon^{-1}.
	\end{equation*}
\end{proof}

Therefore, by \eqref{hit before} and \eqref{hit}, we have the estimates of the initial data at the approximate collision time $-A\frac{\varepsilon}{\delta^2}$, i.e.
\begin{equation}
	\begin{aligned} \label{hit after}
		&\| (\tilde{\phi}_0, \tilde{\psi}_0) \|^2 \leq C (\delta_l + \delta_r)\varepsilon, \\
		&\| (\tilde{\phi}_{0y}, \tilde{\psi}_{0y}) \|^2 \leq C(\delta_l + \delta_r)\varepsilon^{-1}.
	\end{aligned}
\end{equation}

After the approximate collision time $-A\frac{\varepsilon}{\delta^2}$, we seek the solution to \eqref{per after}-\eqref{per afterI} in the functional space defined as follows:
\begin{equation*}
	\tilde{\Pi}(I) = \left\lbrace (\tilde{\phi},\tilde{\psi}) \in C(I, H^1), \tilde{\psi}_y \in L^2(I, H^1) \right\rbrace,
\end{equation*}
where $I \subset \bbr$ is any interval.

\subsection{Proof of global existence and uniform estimates after collision}
After the approximate collision time $-A\frac{\varepsilon}{\delta^2}$, we can obtain the following global existence and uniform estimates.

\begin{theorem} \label{global estimate after}
	{\rm (Global existence and uniform estimates after collision).}
	There exist positive constants $\tilde{\delta}_0$ and $\tilde{\varepsilon}_0$ which are independent of $\varepsilon$, such that if the wave strength $\delta_l \sim \delta_r \sim \delta_2 \sim \delta \leq \tilde{\delta}_0$, and the viscosity coefficient $\varepsilon \leq \tilde{\varepsilon}_0$, then the perturbation problem \eqref{per after}-\eqref{per afterI} admits a unique global solution $(\tilde{\phi}, \tilde{\psi})(\tau, y) \in \tilde{\Pi}\left( [-A\frac{\varepsilon}{\delta^2}, +\infty)\right)$ satisfying
	\begin{enumerate}
		\item There exists a positive constant $C$ such that
		\begin{align}
			\begin{aligned} \label{global estimate 0after}
				&\|(\tilde{\phi},\tilde{\psi})(\tau)\|^2 + \delta_2 \int_{- A\frac{\varepsilon}{\delta^2}}^{\tau} |\dot{X}(\tau)|^2 d\tau + \int_{- A\frac{\varepsilon}{\delta^2}}^{\tau}(G_1(\tau) + G^R(\tau) + G^S(\tau) + D(\tau)) d\tau \\
				&\leq C \|(\tilde{\phi}_0,\tilde{\psi}_0)\|^2 + C \delta_1^{\frac{1}{2}}\varepsilon,
			\end{aligned}
		\end{align}
		where
		\begin{align}
			\begin{aligned} \label{global estimate Gafter}
				&G_1 = \int_{\bbr} a^{-X}_y (\tilde{\psi} + s_2\tilde{\phi})^2 dy, \\			
				&G^R = \int_{\bbr} u^r_y \tilde{\phi}^2 dy, \\
				&G^S = \int_{\bbr} (v^s)^{-X}_y \tilde{\psi}^2 dy, \\
				&D = \varepsilon\int_{\bbr} |\tilde{\psi}_y|^2 dy.
			\end{aligned}
		\end{align}
		\item There exists a constant $C >0$, such that 
		\begin{align}
			\begin{aligned} \label{global estimate 1after}
				\| (\tilde{\phi}_y, \tilde{\psi}_y)(\tau) \|^2 + \int_{- A\frac{\varepsilon}{\delta^2}}^{\tau} (\varepsilon^{-1}D_1 + D_2) d\tau  \leq C (\varepsilon^{-2}\| (\tilde{\phi}_0, \tilde{\psi}_0) \|^2 + \| (\tilde{\phi}_{0y}, \tilde{\psi}_{0y}) \|^2) + C\delta_1^\frac{1}{2}\varepsilon^{-1},
			\end{aligned}
		\end{align}
		where $D_1 = \|\tilde{\phi}_y(\tau)\|^2$, and $D_2 = \varepsilon\|\tilde{\psi}_{yy}(\tau)\|^2$.
		\item Moreover, it holds that
		\begin{equation} \label{L infty after}
			\|(\tilde{\phi},\tilde{\psi})(\tau)\|_{L^\infty} \leq C \|(\tilde{\phi},\tilde{\psi})(\tau)\|^{\frac{1}{2}} \|(\tilde{\phi}_y,\tilde{\psi}_y)(\tau)\|^{\frac{1}{2}} \leq C (\delta_l + \delta_r)^{\frac{1}{2}}.
		\end{equation}
	\end{enumerate}
	Here, the positive constant $C$ is independent of $\varepsilon, \delta_l, \delta_r$ and $\tau$, but depends on $v_\pm, u_\pm$ and $v_*, u_*$.
\end{theorem}

The proof of local existence and uniqueness for the system \eqref{per after}-\eqref{per afterI} is similar to the proof in \cite{KVW}, hence we omit it for brevity.
Therefore, to prove Theorem \ref{global estimate after}, it suffices to close the following {\it a priori} assumptions:
\begin{equation} \label{priori assump after}
	\|(\tilde{\phi}, \tilde{\psi})(\tau)\|^2 \leq \delta^{\frac{2}{3}} \varepsilon, \quad \|(\tilde{\phi}_y, \tilde{\psi}_y)(\tau)\|^2 \leq \delta^{\frac{2}{3}} \varepsilon^{-1},
\end{equation}
for $\tau \in [- A\frac{\varepsilon}{\delta^2}, \tilde{\tau}]$ with $\tilde{\tau} \in [-A\frac{\varepsilon}{\delta^2}, +\infty)$, where $[- A\frac{\varepsilon}{\delta^2}, \tilde{\tau}]$ is the time interval in which the solution is assumed to exist. 

\begin{proposition}[Uniform {\it a priori} estimates after collision] \label{priori estimate after}
	Suppose that the Cauchy problem \eqref{per after}-\eqref{per afterI} has a solution $(\tilde{\phi}, \tilde{\psi})(\tau, y) \in \tilde{\Pi}([- A\frac{\varepsilon}{\delta^2}, \tilde{\tau}])$ for some $\tilde{\tau} \in [- A\frac{\varepsilon}{\delta^2}, +\infty)$ and satisfies the {\it a priori} assumption \eqref{priori assump after}.
	Then there exist positive constants $\tilde{\delta}_0$ and $\tilde{\varepsilon}_0$ which are independent of $\varepsilon$, such that if $\delta_l \sim \delta_r \sim \delta_2 \sim \delta \leq \tilde{\delta}_0, ~ \varepsilon \leq \tilde{\varepsilon}_0$,
	then the inequalities \eqref{global estimate 0after}-\eqref{L infty after} hold.
\end{proposition}

The proof of Proposition \ref{priori estimate after} will be given in Section 6.

\subsection{Wave interaction estimates after collision} 
We now give the wave interaction estimates after the approximate collision time $-A\frac{\varepsilon}{\delta^2}$, which is useful for the estimates of the error terms in Section 6.

\begin{lemma} \label{rare-shock interact}
	Let $X(\tau)$ be the shift function defined in \eqref{X}. Under the same hypotheses as in Proposition \ref{priori estimate after}, the following estimates hold for any $\tau \geq -A\frac{\varepsilon}{\delta^2}$:
	\begin{align*}
		\|(v^s)^{-X}_y (v^r-v^*)\|_{L^1} &\leq C \delta_1 \delta_2 e^{-\frac{C\delta_2}{\varepsilon}\left(\tau + A\frac{\varepsilon}{\delta^2}\right)} + C \delta_1 \delta_2^2 \varepsilon^{-1} \kappa e^{-\frac{C}{\kappa}\left(\tau + A\frac{\varepsilon}{\delta^2}\right)}, \\
		\|(v^s)^{-X}_{y} v^r_{y}\|_{L^1} &\leq C \delta_1 \delta_2 \kappa^{-1} e^{-\frac{C\delta_2}{\varepsilon}\left(\tau + A\frac{\varepsilon}{\delta^2}\right)} + C \delta_1 \delta_2^2 \varepsilon^{-1} e^{-\frac{C}{\kappa}\left(\tau + A\frac{\varepsilon}{\delta^2}\right)}, \\
		\|v^r_{y} ((v^s)^{-X}-v^*)\|_{L^1} &\leq C \delta_1 \varepsilon \kappa^{-1} e^{-\frac{C\delta_2}{\varepsilon}\left(\tau + A\frac{\varepsilon}{\delta^2}\right)} + C \delta_1 \delta_2 e^{-\frac{C}{\kappa}\left(\tau + A\frac{\varepsilon}{\delta^2}\right)}, \\
		\|(v^s)^{-X}_{y} (v^r-v^*)\|^2 &\leq C \delta_1^2 \delta_2^3 \varepsilon^{-1} e^{-\frac{C\delta_2}{\varepsilon}\left(\tau + A\frac{\varepsilon}{\delta^2}\right)} + C \delta_1^2 \delta_2^4 \varepsilon^{-2} \kappa e^{-\frac{C}{\kappa}\left(\tau + A\frac{\varepsilon}{\delta^2}\right)}, \\
		\|(v^s)^{-X}_{y} v^r_{y}\|^2 &\leq C \delta_1^2 \delta_2^3 \kappa^{-2} \varepsilon^{-1} e^{-\frac{C\delta_2}{\varepsilon}\left(\tau + A\frac{\varepsilon}{\delta^2}\right)} + C \delta_1^2 \delta_2^4 \varepsilon^{-2} \kappa^{-1} e^{-\frac{C}{\kappa}\left(\tau + A\frac{\varepsilon}{\delta^2}\right)}, \\
		\|v^r_{y} ((v^s)^{-X}-v^*)\|^2 &\leq C \delta_1^2 \delta_2^2 \kappa^{-1} e^{-\frac{C\delta_2}{\varepsilon}\left(\tau + A\frac{\varepsilon}{\delta^2}\right)} + C \delta_1^2 \delta_2^2 \kappa^{-1} e^{-\frac{C}{\kappa}\left(\tau + A\frac{\varepsilon}{\delta^2}\right)}.
	\end{align*}	
\end{lemma}

\begin{proof}
	First, combining \eqref{X} with the \textit{a priori} assumption \eqref{priori assump after}, we have
	\begin{equation*}
		|\dot{X}(\tau)| \leq \frac{C}{\delta_2} \int_{\bbr} \big|(v^s)^{-X}_y\tilde{\psi}\big| \,dy \leq \frac{C}{\delta_2} \|(v^s)^{-X}_y\| \|\tilde{\psi}\| \leq C \delta_2^{1/2}\varepsilon^{-1/2} \|\tilde{\psi}\| \leq C \delta_2^{1/2}
	\end{equation*}
	for $- A\frac{\varepsilon}{\delta^2} \leq \tau \leq \tilde{\tau}$. Since $X\left(- A\frac{\varepsilon}{\delta^2}\right)=0$, taking $\delta_2$ suitably small yields
	\begin{equation*}
		|X(\tau)| \leq C \delta_2^{1/2} \left(\tau + A\frac{\varepsilon}{\delta^2}\right) \leq \frac{s_2}{4}\left(\tau + A\frac{\varepsilon}{\delta^2}\right), \qquad \text{for } - A\frac{\varepsilon}{\delta^2} \leq \tau \leq \tilde{\tau}.
	\end{equation*}
	
	To obtain the desired estimates, we divide the spatial domain into two regions. On the one hand, for $y - s_2\tau < -\frac{s_2}{2}\left(\tau + A\frac{\varepsilon}{\delta^2}\right)$, it follows that
	\begin{equation*}
		y - s_2\tau - X(\tau) < -\frac{s_2}{2}\left(\tau + A\frac{\varepsilon}{\delta^2}\right) + \frac{s_2}{4}\left(\tau + A\frac{\varepsilon}{\delta^2}\right) = - \frac{s_2}{4}\left(\tau + A\frac{\varepsilon}{\delta^2}\right),
	\end{equation*}
	which implies $| y - s_2\tau - X(\tau) | > \frac{s_2}{4}\left(\tau + A\frac{\varepsilon}{\delta^2}\right)$. Thus, Lemma \ref{lemma-shock} guarantees that
	\begin{align*}
		\big| v^s(y - s_2\tau - X(\tau)) -v^* \big| &\leq C\delta_2 e^{-\frac{C\delta_2}{\varepsilon} |y - s_2\tau - X(\tau)|} \\
		&\leq C\delta_2 e^{-\frac{C\delta_2}{2\varepsilon} |y - s_2\tau - X(\tau)|} e^{-\frac{C\delta_2s_2}{8\varepsilon} \left(\tau + A\frac{\varepsilon}{\delta^2}\right)},
	\end{align*}
	and similarly,
	\begin{align*}
		\big|v^s_y(y - s_2\tau - X(\tau))\big| &\leq C\frac{\delta_2^2}{\varepsilon} e^{-\frac{C\delta_2}{\varepsilon} |y - s_2\tau - X(\tau)|} \\
		&\leq C\frac{\delta_2^2}{\varepsilon} e^{-\frac{C\delta_2}{2\varepsilon} |y - s_2\tau - X(\tau)|} e^{-\frac{C\delta_2s_2}{8\varepsilon} \left(\tau + A\frac{\varepsilon}{\delta^2}\right)}.
	\end{align*}
	
	On the other hand, for $y - s_2\tau \geq -\frac{s_2}{2}\left(\tau + A\frac{\varepsilon}{\delta^2}\right)$, we observe that
	\begin{equation*}
		y - s_2\tau - X(\tau) \geq -\frac{s_2}{2}\left(\tau + A\frac{\varepsilon}{\delta^2}\right) + \frac{s_2}{4}\left(\tau + A\frac{\varepsilon}{\delta^2}\right) = - \frac{s_2}{4}\left(\tau + A\frac{\varepsilon}{\delta^2}\right).
	\end{equation*}
	In this region, Lemma \ref{lemma-rare} implies
	\begin{align*}
		|v^r(\tau, y) - v^*| &\leq C \delta_1 e^{-\frac{2}{\kappa}\left(\left|y + s_2A\frac{\varepsilon}{\delta^2}\right| + |\Lambda_1(v^*)| \left(\tau + A\frac{\varepsilon}{\delta^2}\right)\right)}, \\
		|v^r_y(\tau, y)| &\leq C \delta_1 \kappa^{-1} e^{-\frac{2}{\kappa}\left(\left|y + s_2A\frac{\varepsilon}{\delta^2}\right| + |\Lambda_1(v^*)| \left(\tau + A\frac{\varepsilon}{\delta^2}\right)\right)},
	\end{align*}
	where we note that $|\Lambda_1(v^*)| >0$ is $O(1)$-constant, since $\frac{v_+}{2} \leq v^* \leq v_+$.
	
	Furthermore, recall the uniform bounds provided by Lemmas \ref{lemma-shock} and \ref{lemma-rare}:
	\begin{align*}
		& |v^r(\tau, y) - v^*| \leq C\delta_1, \qquad |v^r_y(\tau, y)| \leq C\delta_1\kappa^{-1}, \qquad \|v^r_y(\tau, y)\|_{L^1} \leq C\delta_1,\\
		& |v^s(y - s_2\tau - X(\tau)) - v^* | \leq C\delta_2, \qquad |v^s_y(y - s_2\tau - X(\tau))| \leq C\delta_2^2\varepsilon^{-1},
	\end{align*}
	Coupling these uniform bounds with the region-specific exponential decays derived above, we deduce that
	\begin{align*}
		\big|(v^s)^{-X}_y\big| \big|v^r - v^* \big| \leq \begin{cases}
			\displaystyle C\delta_1\delta_2^2\varepsilon^{-1} e^{-\frac{C\delta_2}{2\varepsilon} |y - s_2\tau - X(\tau)|} e^{-\frac{C\delta_2s_2}{8\varepsilon} \left(\tau + A\frac{\varepsilon}{\delta^2}\right)}, &\text{for } y - s_2\tau < -\frac{s_2}{2}\left(\tau + A\frac{\varepsilon}{\delta^2}\right), \\[3mm]
			\displaystyle C\delta_1\delta_2^2\varepsilon^{-1} e^{-\frac{2}{\kappa}\left(\left|y + s_2A\frac{\varepsilon}{\delta^2}\right| + |\Lambda_1(v^*)| \left(\tau + A\frac{\varepsilon}{\delta^2}\right)\right)}, &\text{for } y - s_2\tau \geq -\frac{s_2}{2}\left(\tau + A\frac{\varepsilon}{\delta^2}\right).
		\end{cases}
	\end{align*}
	Consequently, integrating over $\bbr$ and utilizing the smallness of $\delta_2$, we obtain the first estimate in the Lemma:
	\begin{align*}
		\int_\bbr \big|(v^s)^{-X}_y\big| \big|v^r - v^* \big| \,dy &\leq C\delta_1\delta_2^2\varepsilon^{-1} e^{-\frac{C\delta_2s_2}{8\varepsilon} \left(\tau + A\frac{\varepsilon}{\delta^2}\right)} \int_\bbr e^{-\frac{C\delta_2}{2\varepsilon} |y - s_2\tau - X(\tau)|} \,dy  \\
		&\quad + C\delta_1\delta_2^2\varepsilon^{-1} \int_\bbr e^{-\frac{2}{\kappa}\left(\left|y + s_2A\frac{\varepsilon}{\delta^2}\right| + |\Lambda_1(v^*)| \left(\tau + A\frac{\varepsilon}{\delta^2}\right)\right)} \,dy \\
		&\leq C \delta_1 \delta_2 e^{-\frac{C\delta_2}{\varepsilon}\left(\tau + A\frac{\varepsilon}{\delta^2}\right)} + C \delta_1 \delta_2^2 \varepsilon^{-1} \kappa e^{-\frac{C}{\kappa}\left(\tau + A\frac{\varepsilon}{\delta^2}\right)}.
	\end{align*}	
	The remaining estimates involve analogous domain-splitting arguments and algebraic manipulations; thus, we omit the details for brevity.
\end{proof}

\subsection{Construction of weight function}
In the process of estimating relative entropy in Section 6, we need to introduce the following weight function $a(y-s_2\tau)$:
\begin{equation}\label{weight}
	a(y-s_2\tau) = 1+ \frac{\lambda}{\delta_2}(v^s(y-s_2\tau) - v^*),
\end{equation}
where
$$\delta_2 \ll \lambda \leq C\sqrt{\delta_2}.$$
Note that
\[ 1<a(y-s_2\tau)<1+\lambda,
\]
and \[ a_y(y-s_2\tau) = \frac{\lambda}{\delta_2}v^s_y(y-s_2\tau) > 0.
\]

\bigskip

\section{Proof of the Main Theorem}

In this section, we prove the main Theorem \ref{theorem1} of this paper based on the global existence and uniform estimates established in Theorem \ref{global estimate before} and Theorem \ref{global estimate after}, both before and after the collision time. In the proof, we need to decompose the full time interval into three distinct subintervals. On the two outer intervals $[ 0, t_0 - A\frac{\varepsilon}{\delta^2}]$ and $[t_0, +\infty)$, the constructed approximate wave pattern is well matched to the exact entropy solution. By contrast, such a matching property no longer holds on the intermediate transition interval $ [ t_0 - A\frac{\varepsilon}{\delta^2}, t_0]$. Nevertheless, this transition interval has a width of order $O(\varepsilon)$, and through a delicate analysis, we show that the error between the approximate composite wave (consisting of viscous shock and smooth rarefaction waves) and the exact entropy solution \eqref{shock-shock-1} tends to zero with the desired convergence rate as $\varepsilon\to 0$. The structure of the entropy solutions for each part can be clearly seen in the following Figure \ref{fig:rare-shock}.


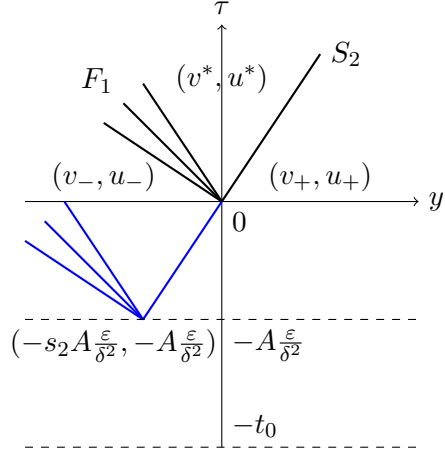
\begin{figure}[h]
	\begin{tikzpicture}[scale = 1.3]
		\draw[->] (-2,0) -- (2,0) node[right] {$y$};
		\draw[->] (0,-2.5) -- (0,1.8) node[above] {$\tau$};
		
		\coordinate (O) at (0,0);
		\coordinate (X0T0) at (-0.8,-1.2);
		
		\draw[thick] (O) -- +(1,1.5) node[right] {$S_2$};
		\draw[thick] (O) -- +(-1,1) node[above left] {$F_1$};
		\draw[thick] (O) -- +(-0.8,1.2);
		\draw[thick] (O) -- +(-1.2,0.8);
		\draw[thick,color=blue] (X0T0) -- +(0.8,1.2);
		\draw[thick,color=blue] (X0T0) -- +(-1,1);
		\draw[thick,color=blue] (X0T0) -- +(-0.8,1.2);
		\draw[thick,color=blue] (X0T0) -- +(-1.2,0.8);
		
		\node[above] at (0,1) {$(v^*,u^*)$};		
		\node[below right] at (0,0) {$0$};		
		\node[above] at (-1.2,0) {$(v_-,u_-)$};
		\node[above] at (1,0) {$(v_+,u_+)$};
		\node[below right] at (0,-1.2) {$-A\frac{\varepsilon}{\delta^2}$};
		\node[below] at (-1.1,-1.2) {$(-s_2A\frac{\varepsilon}{\delta^2}, -A\frac{\varepsilon}{\delta^2})$};
		\node[above right] at (0,-2.5) {$-t_0$};
			
		\draw[dashed] (-2,-1.2) -- (2,-1.2);
		\draw[dashed] (-2,-2.5) -- (2,-2.5);	
	\end{tikzpicture}
	\centering
	\caption{Approximation wave interactions}
	\label{fig:rare-shock}
\end{figure}

Before the collision time $\tau =0$, according to the approximate collision time $-A\frac{\varepsilon}{\delta^2}$, we partition the time interval $[-t_0, 0]$ into two sub-intervals, $[-t_0, -A\frac{\varepsilon}{\delta^2}]$ and $[-A\frac{\varepsilon}{\delta^2}, 0]$.

For the first sub-interval $\tau \in \left[-t_0, -A\frac{\varepsilon}{\delta^2}\right]$, by Theorem \ref{global estimate before} and Lemma \ref{lemma-shock}, we deduce
\begin{align}
	\begin{aligned} \label{proof thm-1}
		\| v(\tau, y) - V(\tau, y) \|^2 
		&= \left\| v(\tau, y) - \left( v^S_l(y - s_l \tau) + v^S_r(y - s_r \tau) - v_* \right) \right\|^2 \\
		&\leq \| v(\tau, y) - \bar{v} \|^2 + \| v^s_l(y - s_l \tau) - v^S_l(y - s_l \tau) \|^2 \\
		&\quad + \| v^s_r(y - s_r \tau) - v^S_r(y - s_r \tau) \|^2 \\
		&\leq \| \phi \|^2 + C \delta_l^2 \int_{\bbr} e^{-\frac{C\delta_l}{\varepsilon}|y|} \,dy + C \delta_r^2 \int_{\bbr} e^{-\frac{C\delta_r}{\varepsilon}|y|} \,dy \\
		&\leq C (\delta_l + \delta_r)\varepsilon.
	\end{aligned}	
\end{align}

On the second sub-interval $\tau \in \left[-A\frac{\varepsilon}{\delta^2}, 0\right]$, it follows from Theorem \ref{global estimate after} that
\begin{align}
	\begin{aligned} \label{proof thm-2}
		\| v(\tau, y) - V(\tau, y) \|^2 &\leq \| v(\tau, y) - \tilde{v}(\tau, y) \|^2 + \| \tilde{v}(\tau, y) - V(\tau, y) \|^2 \\
		&\leq C\delta\varepsilon + \| \tilde{v}(\tau, y) - V(\tau, y) \|^2.
	\end{aligned}	
\end{align}
Therefore, to bound $\| v(\tau, y) - V(\tau, y) \|^2$ in \eqref{proof thm-2}, it suffices to estimate $\| \tilde{v}(\tau, y) - V(\tau, y) \|^2$. Observe that
\begin{align}
	\begin{aligned} \label{proof thm-3}
		&\tilde{v}(\tau, y) - V(\tau, y) \\
		&= \left( v^r(\tau, y) + v^s(y - s_2 \tau - X(\tau)) - v^* \right) - \left( v^S_l(y - s_l \tau) + v^S_r(y - s_r \tau) - v_* \right) \\
		&= \begin{cases} 
			\begin{aligned}
				&(v^r(\tau, y) - v_-) + (v^s(y - s_2 \tau - X(\tau)) - v^*) \\
				&- (v^S_l(y - s_l \tau) - v_-) - (v^S_r(y - s_r \tau) - v_*), 
			\end{aligned} & \text{for } y \leq s_2\tau, \\[5mm]
			\begin{aligned}
				&(v^r(\tau, y) - v^*) + (v^s(y - s_2 \tau - X(\tau)) - v_+) \\
				&- (v^S_l(y - s_l \tau) - v_*) - (v^S_r(y - s_r \tau) - v_+), 
			\end{aligned} & \text{for } y \geq s_2\tau,
		\end{cases}
	\end{aligned}	
\end{align}
we proceed to estimate the terms in \eqref{proof thm-3} individually. 

First, for the terms involving $v^r(\tau, y)$, we have
\begin{align*}
	&\int_{-\infty}^{s_2\tau} (v^r(\tau, y) - v_-)^2 \,dy + \int_{s_2\tau}^{+\infty}(v^r(\tau, y) - v^*)^2 \,dy \\
	&= \int_{-\infty}^{\Lambda_1(v_-)\left(\tau + A\frac{\varepsilon}{\delta^2}\right) -s_2A\frac{\varepsilon}{\delta^2}} (v^r(\tau, y) - v_-)^2 \,dy \\
	&\quad + \int_{\Lambda_1(v_-)\left(\tau + A\frac{\varepsilon}{\delta^2}\right) -s_2A\frac{\varepsilon}{\delta^2}}^{s_2\tau} (v^r(\tau, y) - v_-)^2 \,dy \\
	&\quad + \int_{\Lambda_1(v^*)\left(\tau + A\frac{\varepsilon}{\delta^2}\right) -s_2A\frac{\varepsilon}{\delta^2}}^{+\infty}(v^r(\tau, y) - v^*)^2 \,dy \\
	&\quad + \int_{s_2\tau}^{\Lambda_1(v^*)\left(\tau + A\frac{\varepsilon}{\delta^2}\right) -s_2A\frac{\varepsilon}{\delta^2}}(v^r(\tau, y) - v^*)^2 \,dy \\
	&\leq C\delta_1^2\varepsilon + C \delta_1^2A\frac{\varepsilon}{\delta^2} \\
	&\leq C\delta^4 \varepsilon.
\end{align*}
Furthermore, utilizing the bound
\begin{align*}
	|\dot{X}(\tau)| \leq \frac{C}{\delta_2} \int_{\bbr} \big|(v^s)^{-X}_y\tilde{\psi}\big| \,dy \leq \frac{C}{\delta_2} \|(v^s)^{-X}_y\| \|\tilde{\psi}\| \leq C \delta_2^\frac{1}{2}\varepsilon^{-\frac{1}{2}} \|\tilde{\psi}\| \leq C \delta,
\end{align*}
along with the initial condition $X\left(-A\frac{\varepsilon}{\delta^2}\right)=0$, we deduce
$$ |X(\tau)| \leq |\dot{X}(\tau)| A\frac{\varepsilon}{\delta^2} \leq C \delta A\frac{\varepsilon}{\delta^2} \leq C \frac{\varepsilon}{\delta}, \qquad \text{for } \tau \in \left[-A\frac{\varepsilon}{\delta^2}, 0\right]. $$
Therefore, for the terms involving $v^s(y - s_2 \tau - X(\tau))$, we obtain
\begin{align*}
	&\int_{-\infty}^{s_2\tau} (v^s(y - s_2 \tau - X(\tau)) - v^*)^2 \,dy + \int_{s_2\tau}^{+\infty}(v^s(y - s_2 \tau - X(\tau)) - v_+)^2 \,dy \\
	&= \int_{-\infty}^{s_2\tau + X(\tau)} (v^s(y - s_2 \tau - X(\tau)) - v^*)^2 \,dy \\
	&\quad + \int_{s_2\tau + X(\tau)}^{s_2\tau} (v^s(y - s_2 \tau - X(\tau)) - v^*)^2 \,dy \\
	&\quad + \int_{s_2\tau + X(\tau)}^{+\infty}(v^s(y - s_2 \tau - X(\tau)) - v_+)^2 \,dy \\
	&\quad + \int_{s_2\tau}^{s_2\tau + X(\tau)}(v^s(y - s_2 \tau - X(\tau)) - v_+)^2 \,dy \\
	&\leq C\delta_2\varepsilon + C \delta_2^2 |X(\tau)| \\
	&\leq C\delta \varepsilon.
\end{align*}	
The remaining terms can be estimated similarly. Combining these bounds yields
\begin{equation} \label{proof thm-4}
	\| \tilde{v}(\tau, y) - V(\tau, y) \|^2 \leq C\delta \varepsilon, \qquad \text{for } \tau \in \left[-A\frac{\varepsilon}{\delta^2}, 0\right].
\end{equation}
Substituting \eqref{proof thm-4} into \eqref{proof thm-2} thus implies
\begin{equation} \label{proof thm-5}
	\| v(\tau, y) - V(\tau, y) \|^2 \leq C\delta \varepsilon, \qquad \text{for } \tau \in \left[-A\frac{\varepsilon}{\delta^2}, 0\right].
\end{equation}

For the post-collision regime $\tau \in [0, +\infty)$, we have
\begin{align}
	\begin{aligned} \label{proof thm-6}
		\| v(\tau, y) - V(\tau, y) \|^2 
		&\leq \| v(\tau, y) - \tilde{v}(\tau, y) \|^2 + \| \tilde{v}(\tau, y) - V(\tau, y) \|^2 \\
		&\leq \|\tilde{\phi}\|^2 + \| \tilde{v}(\tau, y) - V(\tau, y) \|^2 \\
		&\leq C\delta\varepsilon + \| v^s(y - s_2 \tau - X(\tau)) - v^S(y - s_2 \tau) \|^2 \\
		&\quad + \left\| v^r(\tau, y) - v^R\left(\frac{y}{\tau}\right) \right\|^2.
	\end{aligned}
\end{align}
We estimate the terms on the right-hand side of \eqref{proof thm-6} individually. First, Lemma \ref{lemma-shock} gives
\begin{align}
	\begin{aligned} \label{proof thm-7}
		&\| v^s(y - s_2 \tau - X(\tau)) - v^S(y - s_2 \tau) \|^2 \\
		&\leq \| v^s(y - s_2 \tau - X(\tau)) - v^S(y - s_2 \tau -X(\tau)) \|^2  \\
		&\quad + \| v^S(y - s_2 \tau - X(\tau)) - v^S(y - s_2 \tau) \|^2 \\
		&\leq C \delta_2 \varepsilon + C \delta_2^2 |X(\tau)| \\
		&\leq C \delta \varepsilon + C \delta^2\sqrt{\tau}\varepsilon^{\frac{1}{2}},
	\end{aligned}	
\end{align}
where, to bound $|X(\tau)|$, we utilized the estimate from the weighted relative entropy bound:
$$ \delta_2 \int_{0}^{\tau}|\dot{X}(\tau)|^2 \,d\tau \leq  C \delta\varepsilon, $$
which ensures that
\begin{align*}
	|X(\tau)| &\leq |X(0)| + \left| \int_{0}^{\tau} \dot{X}(\tau) \,d\tau \right| \\
	&\leq |X(0)| + \left( \int_{0}^{\tau} 1 \,d\tau \right)^\frac{1}{2} \left( \int_{0}^{\tau}|\dot{X}(\tau)|^2 \,d\tau \right)^\frac{1}{2} \\
	&\leq C \varepsilon\delta^{-1} + C \sqrt{\tau}\varepsilon^\frac{1}{2}.
\end{align*}
Additionally, Lemma \ref{lemma-rare} implies
\begin{align}
	\begin{aligned} \label{proof thm-8}
		\left\| v^r(\tau, y) - v^R\left( \frac{y}{\tau} \right) \right\|^2 
		&\leq C\left\| v^r(0, y) - v^R(0, y) \right\|^2 \\
		&\leq C\left\| v^r(0, y) - v^R\left( \frac{y+s_2A\frac{\varepsilon}{\delta^2}}{A\frac{\varepsilon}{\delta^2}}\right) \right\|^2 \\
		&\quad + C\left\| v^R\left(  \frac{y+s_2A\frac{\varepsilon}{\delta^2}}{A\frac{\varepsilon}{\delta^2}}\right)  - v^R(0, y) \right\|^2 \\
		&\leq C \delta_1^2\varepsilon + C\delta_1^2A\frac{\varepsilon}{\delta^2} 
		\leq C \delta^4\varepsilon.
	\end{aligned}
\end{align}
Combining \eqref{proof thm-7} and \eqref{proof thm-8} thus yields
\begin{equation} \label{proof thm-9}
	\| v(\tau, y) - V(\tau, y) \|^2 \leq C \delta \varepsilon + C \delta^2 \sqrt{\tau}\varepsilon^\frac{1}{2}, \qquad \text{for } \tau \in [0, +\infty).
\end{equation}

Finally, piecing together \eqref{proof thm-1}, \eqref{proof thm-5}, and \eqref{proof thm-9} establishes the $L^2$ estimate required for the main Theorem \ref{theorem1}. The corresponding $L^p$ ($p>2$) estimate follows readily by incorporating the bounds \eqref{L infty before} and \eqref{L infty after}. As the details are entirely analogous, we omit them. The proof of the main Theorem \ref{theorem1} is now complete.

\bigskip

%
%

\section{Proof of Uniform {\it a priori} Estimates in Proposition \ref{priori estimate before}}
\setcounter{equation}{0}

This section is devoted to the proof of the uniform {\it a priori} estimates before the approximate collision time $\left[-t_0, -A\frac{\varepsilon}{\delta^2}\right]$ in Proposition \ref{priori estimate before}. The argument is based on the anti-derivative method, combined with the relative entropy method.

First note that, from the {\it a priori} assumption in \eqref{priori assump before}, it is straightforward to see  $\frac{v_+}{4} \leq v \leq 4 v_+$ if $\delta$ is suitably small. Before proving Proposition \ref{priori estimate before}, we present some useful explicit expressions on the relative quantities associated to the pressure $p(v) = v^{-\gamma}$, and the potential energy $Q(v) := \frac{v^{1-\gamma}}{\gamma - 1}$ (i.e., $Q'(v) = -p(v)$) based on Taylor expansions, which are significant for the {\it a priori} estimates, especially for relative entropy estimates. The relative pressure $p(v|\bar{v})$ and relative potential energy $Q(v|\bar{v})$ are defined as
$$p(v|\bar{v}) := p(v) -p(\bar{v}) - p'(\bar{v})(v - \bar{v}), \quad Q(v|\bar{v}) := Q(v) -Q(\bar{v}) - Q'(\bar{v})(v - \bar{v}).$$

We now give the proof of Proposition \ref{priori estimate before} for the uniform $H^2$ {\it a priori} estimates for the anti-derivative variables. 
To begin with, Subsection 5.1 is devoted to proving the lower-order estimates, including specifically  the $L^2$ estimates of the anti-derivative variables stated in Lemma \ref{lemma anti}, their derivative estimates in Lemma \ref{lemma anti1}, and the relative entropy estimates of the original variables in Lemma \ref{lemma 0before}. Subsequently, the first-order derivative estimates for the original variables are established in Lemmas \ref{lemma phi 1before} and \ref{lemma psi 1before} within Subsection 5.2. A crucial point is that, to close the {\it a priori} estimates, the solution must be sought on the time interval $\left[-t_0, -A\frac{\varepsilon}{\delta^2}\right]$ with a suitably large constant $A$,  rather than directly on the entire interval $[-t_0, 0]$. In other words, it is necessary to introduce a carefully constructed approximate collision time $\tau = -A\frac{\varepsilon}{\delta^2}$ for the collision time $\tau =0$.

\subsection{Lower order estimates}

We start with the $L^2$-estimates for the anti-derivative variables.

\begin{lemma} \label{lemma anti}
	Under the hypotheses of Proposition \ref{priori estimate before}, it holds that for $\tau \in \left[-t_0, -A\frac{\varepsilon}{\delta^2}\right]$,
	\begin{align}
		\begin{aligned} \label{estimate anti}
			&\| (\Phi, \Psi)(\tau) \|^2 + \int_{-t_0}^{\tau}\int_{\bbr} |\bar{u}_y|\Psi^2 dyd\tau + \varepsilon \int_{-t_0}^{\tau} \| \Psi_y \|^2 d\tau \\
			&\leq  Ce^{-cA}\delta^{-1}\varepsilon^3 + C\left(\delta^2 + e^{-c\sqrt{A}}\right)\varepsilon \int_{-t_0}^{\tau}\|\Phi_y\|^2 d\tau + C e^{-c\sqrt{A}}\varepsilon^3 \int_{-t_0}^{\tau} \|\psi_y\|^2 d\tau.
		\end{aligned}
	\end{align}
\end{lemma}

\begin{proof}
	Multiplying \eqref{anti-per line}$_1$ by $\Phi$, \eqref{anti-per line}$_2$ by $-\frac{\Psi}{p'(\bar{v})}$, it holds that
	$${ \left( \frac{1} {2} \Phi^{2} \right) }_{\tau} - \Phi\Psi_y = 0 ,$$
	and
	\begin{align*}
		&\left( -\frac{1}{ 2p^{\prime} (\bar{v}) } \Psi^2 \right)_{\tau} - \frac{p^{\prime\prime} (\bar{v})}{2 |p^{\prime} (\bar{v})|^2} \bar{v}_{\tau} \Psi^2 - \Phi_y \Psi + \frac{\varepsilon\Psi_y^2}{\bar{v} | p^{\prime} (\bar{v}) |} \\
	    &= - \left( \frac{\varepsilon\Psi_y \Psi}{\bar{v} p^{\prime} (\bar{v})} \right)_y + \left( \frac{\varepsilon}{\bar{v} p^{\prime} (\bar{v})} \right)_y \Psi_y \Psi - \frac{\Psi}{p^{\prime}(\bar{v})} ( Q_1 + Q_2 - F_1 - F_2 ).
	\end{align*}
	Adding the above two equations, we get
	\begin{align*}
		&\left(\frac{1} {2} \Phi^{2} -\frac{1}{ 2p^{\prime} (\bar{v}) } \Psi^2 \right)_{\tau} - \frac{p^{\prime\prime} (\bar{v})}{2 |p^{\prime} (\bar{v})|^2} \bar{v}_{\tau} \Psi^2 + \frac{\varepsilon\Psi_y^2}{\bar{v} | p^{\prime} (\bar{v}) |} \\
		&= \left( \Phi\Psi - \frac{\varepsilon\Psi_y \Psi}{\bar{v} p^{\prime} (\bar{v})} \right)_y + \left( \frac{\varepsilon}{\bar{v} p^{\prime} (\bar{v})} \right)_y \Psi_y \Psi - \frac{\Psi}{p^{\prime}(\bar{v})} ( Q_1 + Q_2 - F_1 - F_2 ).
	\end{align*}
	Then taking the integration over  $[-t_0, \tau] \times \bbr$ for $\tau \in \left[-t_0, -A\frac{\varepsilon}{\delta^2}\right]$, we have
	\begin{align}
		\begin{aligned} \label{need estimate anti}
			&\| (\Phi, \Psi)(\tau) \|^2 + \int_{-t_0}^{\tau}\int_{\bbr}  |\bar{u}_y| \Psi^2 dyd\tau + \varepsilon\int_{-t_0}^{\tau} \|\Psi_y\|^2 d\tau \\
			&= C\left| \int_{-t_0}^{\tau}\int_{\bbr} \left( \frac{\varepsilon}{\bar{v} p^{\prime} (\bar{v})} \right)_y \Psi_y \Psi dyd\tau \right| + C \left| \int_{-t_0}^{\tau}\int_{\bbr} \frac{\Psi}{p^{\prime}(\bar{v})} ( Q_1 + Q_2 - F_1 - F_2 ) dyd\tau \right|.
		\end{aligned}	    
    \end{align}
   
Now we estimate the terms on the right-hand side of \eqref{need estimate anti}.
First of all, by Young's inequality, H\"{o}lder's inequality and Lemma \ref{lemma-shock}, one has
\begin{align*}
	&C\left| \int_{-t_0}^{\tau}\int_{\bbr} \left( \frac{\varepsilon}{\bar{v} p^{\prime} (\bar{v})} \right)_y \Psi_y \Psi dyd\tau \right| \\
	&\leq C \varepsilon \int_{-t_0}^{\tau}\int_{\bbr} |\bar{v}_y \Psi\Psi_y| dyd\tau \\
	&\leq \frac{1}{8} \varepsilon\int_{-t_0}^{\tau} \|\Psi_y\|^2 d\tau + C\varepsilon \int_{-t_0}^{\tau}\int_{\bbr} |\bar{v}_y|^2 \Psi^2 dyd\tau \\
	&\leq \frac{1}{8} \varepsilon\int_{-t_0}^{\tau} \|\Psi_y\|^2 d\tau + C\varepsilon\|\bar{v}_y\|_{L^\infty} \int_{-t_0}^{\tau}\int_{\bbr} |\bar{u}_y|  \Psi^2 dyd\tau \\
	&\leq \frac{1}{8} \varepsilon\int_{-t_0}^{\tau} \|\Psi_y\|^2 d\tau + C\left(\delta_l^2 + \delta_r^2\right) \int_{-t_0}^{\tau}\int_{\bbr} |\bar{u}_y|  \Psi^2 dyd\tau.
\end{align*}
By Cauchy's inequality, Young's inequality, H\"{o}lder's inequality, Lemma \ref{lemma-shock} and {\it a priori} assumption \eqref{priori assump before}, we have
\begin{align*}
	&C \left| \int_{-t_0}^{\tau}\int_{\bbr} \frac{\Psi}{p^{\prime}(\bar{v})} Q_1 dyd\tau \right|\\
	&= C \varepsilon \left| \int_{-t_0}^{\tau}\int_{\bbr} \frac{\Psi}{p^{\prime}(\bar{v})} \left(\frac{1}{v} - \frac{1}{\bar{v}}\right)\left( \Psi_{yy} + \bar{u}_y \right) dyd\tau \right| \\
	&\leq C \varepsilon\int_{-t_0}^{\tau}\int_{\bbr} |\Psi\Phi_y\Psi_{yy}| dyd\tau + C \varepsilon\int_{-t_0}^{\tau}\int_{\bbr} |\bar{u}_y \Psi\Phi_y| dyd\tau \\
	&\leq C\|\Psi\|_{L^\infty} \int_{-t_0}^{\tau} \|\Phi_y\|^2 d\tau + C \varepsilon^2 \|\Psi\|_{L^\infty} \int_{-t_0}^{\tau} \|\Psi_{yy}\|^2 d\tau + \frac{1}{8}\int_{-t_0}^{\tau} \int_{\bbr} |\bar{u}_y| \Psi^2 dyd\tau \\
	&\quad + C\varepsilon^2 \|\bar{u}_y\|_{L^\infty} \int_{-t_0}^{\tau}\|\Phi_y\|^2 d\tau \\
	&\leq Ce^{-c\sqrt{A}}\varepsilon \int_{-t_0}^{\tau} \|\Phi_y\|^2 d\tau + Ce^{-c\sqrt{A}} \varepsilon^3 \int_{-t_0}^{\tau} \|\psi_y\|^2 d\tau + \frac{1}{8}\int_{-t_0}^{\tau} \int_{\bbr} |\bar{u}_y| \Psi^2 dyd\tau \\
	&\quad + C\left(\delta_l^2 + \delta_r^2\right)\varepsilon \int_{-t_0}^{\tau}\|\Phi_y\|^2 d\tau.
\end{align*}
It follows from H\"{o}lder's inequality and {\it a priori} assumption \eqref{priori assump before} that
\begin{align*}
	&C\left| \int_{-t_0}^{\tau}\int_{\bbr} \frac{\Psi}{p^{\prime}(\bar{v})} Q_2 dyd\tau \right|
	= C \left| \int_{-t_0}^{\tau}\int_{\bbr} \frac{\Psi}{p^{\prime}(\bar{v})} (p(v) -p(\bar{v}) - p'(\bar{v})\phi) dyd\tau \right| \\
	&\leq C \int_{-t_0}^{\tau}\int_{\bbr} |\Psi\Phi_y^2| dyd\tau
	\leq C \|\Psi\|_{L^\infty} \int_{-t_0}^{\tau}\|\Phi_y\|^2 d\tau \leq C e^{-c\sqrt{A}}\varepsilon \int_{-t_0}^{\tau}\|\Phi_y\|^2 d\tau.
\end{align*}
By Young's inequality, H\"{o}lder's inequality, Sobolev's inequality and Lemma \ref{lem:shock-interact}, we obtain
\begin{align}
	\begin{aligned} \label{estimate error before}
		&C \left| \int_{-t_0}^{\tau}\int_{\bbr} \frac{\Psi}{p^{\prime}(\bar{v})} F_1 dyd\tau \right| \\
		&\leq C \int_{-t_0}^{\tau}\|\Psi\|_{L^\infty}\|F_1\|_{L^1} d\tau \\
		&\leq C \int_{-t_0}^{\tau} \|\Psi\|^{\frac{1}{2}}\|\Psi_y\|^{\frac{1}{2}}\|F_1\|_{L^1} d\tau \\
		&\leq \frac{1}{8} \varepsilon \int_{-t_0}^{\tau}\|\Psi_y\|^2 d\tau + C \varepsilon^{-\frac{1}{3}} \int_{-t_0}^{\tau} \|\Psi\|^{\frac{2}{3}} \|F_1\|_{L^1}^{\frac{4}{3}} d\tau \\
		&\leq \frac{1}{8} \varepsilon \int_{-t_0}^{\tau}\|\Psi_y\|^2 d\tau + C \varepsilon^{-\frac{1}{3}} \sup_{-t_0 \leq \tau \leq -A\frac{\varepsilon}{\delta^2}} \|\Psi\|^{\frac{2}{3}} \int_{-t_0}^{\tau} \|F_1\|_{L^1}^{\frac{4}{3}} d\tau \\
		&\leq \frac{1}{8} \varepsilon \int_{-t_0}^{\tau}\|\Psi_y\|^2 d\tau + \frac{1}{8}\sup_{-t_0 \leq \tau \leq -A\frac{\varepsilon}{\delta^2}} \|\Psi\|^2 + C \varepsilon^{-\frac{1}{2}} \left( \int_{-t_0}^{\tau} \|F_1\|_{L^1}^{\frac{4}{3}} d\tau \right)^{\frac{3}{2}} \\
		&\leq \frac{1}{8} \varepsilon \int_{-t_0}^{\tau}\|\Psi_y\|^2 d\tau + \frac{1}{8}\sup_{-t_0 \leq \tau \leq -A\frac{\varepsilon}{\delta^2}} \|\Psi\|^2 + C \varepsilon^3 \delta_l^2\delta_r^2\left( \frac{1}{\delta_l^3} e^{ \frac{c \delta_l^2\tau}{\varepsilon}} + \frac{1}{\delta_r^3} e^{ \frac{c \delta_r^2\tau}{\varepsilon}} \right) \\
		&\leq \frac{1}{8} \varepsilon \int_{-t_0}^{\tau}\|\Psi_y\|^2 d\tau + \frac{1}{8}\sup_{-t_0 \leq \tau \leq -A\frac{\varepsilon}{\delta^2}} \|\Psi\|^2 + C e^{-cA} \delta\varepsilon^3,
	\end{aligned}
\end{align}
and
\begin{align*}
	&C \left| \int_{-t_0}^{\tau}\int_{\bbr} \frac{\Psi}{p^{\prime}(\bar{v})} F_2 dyd\tau \right| \\
	&\leq \frac{1}{8} \varepsilon \int_{-t_0}^{\tau}\|\Psi_y\|^2 d\tau + \frac{1}{8}\sup_{-t_0 \leq \tau \leq -A\frac{\varepsilon}{\delta^2}} \|\Psi\|^2 + C \varepsilon^{-\frac{1}{2}} \left( \int_{-t_0}^{\tau} \|F_2\|_{L^1}^{\frac{4}{3}} d\tau \right)^{\frac{3}{2}} \\
	&\leq \frac{1}{8} \varepsilon \int_{-t_0}^{\tau}\|\Psi_y\|^2 d\tau + \frac{1}{8}\sup_{-t_0 \leq \tau \leq -A\frac{\varepsilon}{\delta^2}} \|\Psi\|^2 + C  \varepsilon^3 \left( \frac{\delta_r^2}{\delta_l^3} e^{ \frac{c \delta_l^2\tau}{\varepsilon}} + \frac{\delta_l^2}{\delta_r^3} e^{ \frac{c \delta_r^2\tau}{\varepsilon}} \right) \\
	&\leq \frac{1}{8} \varepsilon \int_{-t_0}^{\tau}\|\Psi_y\|^2 d\tau + \frac{1}{8}\sup_{-t_0 \leq \tau \leq -A\frac{\varepsilon}{\delta^2}} \|\Psi\|^2 + C e^{-cA}\delta^{-1}\varepsilon^3.
\end{align*}
Note that by Lemma \ref{lemma-shock} and Lemma \ref{lem:shock-interact}, we have
\begin{align*}
	&\|F_1\|_{L^1} \leq C \delta_l\delta_r\varepsilon \left( e^{ \frac{c \delta_l^2\tau}{\varepsilon}} +  e^{ \frac{c \delta_r^2\tau}{\varepsilon}} \right), \quad -t_0\leq\tau\leq 0, \\
	&\|F_2\|_{L^1} \leq C \delta_r\varepsilon e^{ \frac{c \delta_l^2\tau}{\varepsilon}} + C\delta_l\varepsilon e^{ \frac{c \delta_r^2\tau}{\varepsilon}}, \quad -t_0\leq\tau\leq 0.
\end{align*}

Combining the above estimates and using the smallness of $\delta_l \sim \delta_r \sim \delta$, we get the estimate \eqref{estimate anti}.
\end{proof}

We now address the derivative estimate for the anti-derivative variables.

\begin{lemma} \label{lemma anti1}
	Under the hypotheses of Proposition \ref{priori estimate before}, it holds that for $\tau \in \left[-t_0, -A\frac{\varepsilon}{\delta^2}\right]$,
	\begin{align}
		\begin{aligned} \label{estimate anti1}
			&\varepsilon \|\Phi_y(\tau)\|^2 + \int_{-t_0}^{\tau} \|\Phi_y\|^2 d\tau \\
			&\leq C e^{-cA} \delta \varepsilon^2 + C\varepsilon^{-1}\|\Psi(\tau)\|^2 + C\int_{-t_0}^{\tau}\|\Psi_y(\tau)\|^2 d\tau + C e^{-c\sqrt{A}} \delta\varepsilon^2 \int_{-t_0}^{\tau} \|\psi_y\|^2 d\tau.
		\end{aligned}
	\end{align}
\end{lemma}
\begin{proof}
	From \eqref{anti-per line}, we have
	\begin{equation} \label{Phiy}
		\frac{\varepsilon}{\bar{v}}\Phi_{y\tau} - p'(\bar{v})\Phi_y - \Psi_{\tau} = -Q_1 - Q_2 + F_1 + F_2.
	\end{equation}
	Multiplying \eqref{Phiy} by $\Phi_y$ implies
	\begin{equation} \label{Phiy2}
		\left(\frac{\varepsilon}{2\bar{v}}\Phi_y^2 - \Psi\Phi_y\right)_{\tau} + |p'(\bar{v})|\Phi_{y}^{2} = \left(\frac{\varepsilon}{2\bar{v}}\right)_{\tau}\Phi_{y}^{2} + \Psi_{y}^{2} + (Q_1 + Q_2 - F_1 - F_2)\Phi_{y}-(\Phi_{\tau}\Psi)_{y}.
	\end{equation}
	Integrating \eqref{Phiy2} over $[-t_0,\tau]\times\bbr$ for $\tau \in \left[-t_0, -A\frac{\varepsilon}{\delta^2}\right]$ and using the Young's inequality, Lemma \ref{lemma-shock} and the smallness of $\delta_l \sim \delta_r \sim \delta$, we obtain
	\begin{align*}
		&\varepsilon \|\Phi_y(\tau)\|^2 + \int_{-t_0}^{\tau} \|\Phi_y\|^2 d\tau \\
		&\leq C\varepsilon^{-1}\|\Psi(\tau)\|^2 + C\int_{-t_0}^{\tau}\|\Psi_y(\tau)\|^2 d\tau + C \left| \int_{-t_0}^{\tau}\int_{\bbr} (Q_1 + Q_2 - F_1 - F_2)\Phi_{y} dyd\tau \right|.
	\end{align*}
Now we estimate the rest terms as follows. By H\"{o}lder's inequality, Young's inequality, Lemma \ref{lemma-shock} and {\it a priori} assumption \eqref{priori assump before}, we have
\begin{align*}
	&C \left| \int_{-t_0}^{\tau}\int_{\bbr} Q_1\Phi_y dyd\tau \right|
	= C \varepsilon \left| \int_{-t_0}^{\tau}\int_{\bbr} \left(\frac{1}{v} - \frac{1}{\bar{v}}\right)\left( \Psi_{yy} + \bar{u}_y \right) \Phi_y dyd\tau \right| \\
	&\leq C \varepsilon\int_{-t_0}^{\tau}\int_{\bbr} |\Phi_y^2\Psi_{yy}| dyd\tau + C \varepsilon\int_{-t_0}^{\tau}\int_{\bbr} |\bar{u}_y \Phi_y^2| dyd\tau \\
	&\leq C\varepsilon \|\Phi_y\|_{L^\infty} \int_{-t_0}^{\tau} \|\Phi_y\| \|\Psi_{yy}\| d\tau + C\left(\delta_l^2 + \delta_r^2\right) \int_{-t_0}^{\tau}\|\Phi_y\|^2 d\tau \\
	&\leq \frac{1}{8} \int_{-t_0}^{\tau} \|\Phi_y\|^2 d\tau + C \varepsilon^2 \|\phi\|_{L^\infty}^2 \int_{-t_0}^{\tau} \|\Psi_{yy}\|^2 d\tau + C\left(\delta_l^2 + \delta_r^2\right) \int_{-t_0}^{\tau}\|\Phi_y\|^2 d\tau \\
	&\leq \frac{1}{4} \int_{-t_0}^{\tau} \|\Phi_y\|^2 d\tau + C e^{-c\sqrt{A}} \delta\varepsilon^2 \int_{-t_0}^{\tau} \|\psi_y\|^2 d\tau.
\end{align*}
It follows from H\"{o}lder's inequality and {\it a priori} assumption \eqref{priori assump before} that
\begin{align*}
	&C\left| \int_{-t_0}^{\tau}\int_{\bbr} Q_2 \Phi_y dyd\tau \right|
	= C \left| \int_{-t_0}^{\tau}\int_{\bbr} (p(v) -p(\bar{v}) - p'(\bar{v})\phi) \Phi_y dyd\tau \right| \\
	&\leq C \int_{-t_0}^{\tau}\int_{\bbr} |\Phi_y^3| dyd\tau
	\leq C \|\phi\|_{L^\infty} \int_{-t_0}^{\tau}\|\Phi_y\|^2 d\tau \leq C e^{-c\sqrt{A}} \delta^{\frac{1}{2}} \int_{-t_0}^{\tau}\|\Phi_y\|^2 d\tau.
\end{align*}
By Young's inequality, Lemma \ref{lemma-shock} and Lemma \ref{lem:shock-interact}, we get
\begin{align*}
	C \left| \int_{-t_0}^{\tau}\int_{\bbr} F_1 \Phi_y dyd\tau \right| 
	&\leq \frac{1}{8} \int_{-t_0}^{\tau} \|\Phi_y\|^2 d\tau + C \int_{-t_0}^{\tau}\int_{\bbr} |F_1|^2 dyd\tau \\
	&\leq \frac{1}{8} \int_{-t_0}^{\tau} \|\Phi_y\|^2 d\tau + C (\delta_l + \delta_r) \delta_r^2 \varepsilon^2 e^{\frac{c \delta_l^2}{\varepsilon} \tau} + C (\delta_l + \delta_r) \delta_l^2 \varepsilon^2 e^{\frac{c \delta_r^2}{\varepsilon} \tau} \\
	&\leq \frac{1}{8} \int_{-t_0}^{\tau} \|\Phi_y\|^2 d\tau + C e^{-cA} \delta^3 \varepsilon^2,
\end{align*}
and
\begin{align*}
	C \left| \int_{-t_0}^{\tau}\int_{\bbr} F_2 \Phi_y dyd\tau \right| 
	&\leq \frac{1}{8} \int_{-t_0}^{\tau} \|\Phi_y\|^2 d\tau + C \int_{-t_0}^{\tau}\int_{\bbr} |F_2|^2 dyd\tau \\
	&\leq \frac{1}{8} \int_{-t_0}^{\tau} \|\Phi_y\|^2 d\tau + C \frac{\delta_r^2}{\delta_l} \varepsilon^2 e^{\frac{c \delta_l^2}{\varepsilon} \tau} + C \frac{\delta_l^2}{\delta_r} \varepsilon^2 e^{\frac{c \delta_r^2}{\varepsilon} \tau} \\
	&\leq \frac{1}{8} \int_{-t_0}^{\tau} \|\Phi_y\|^2 d\tau + C e^{-cA} \delta \varepsilon^2,
\end{align*}
where we have used the fact that
\begin{align*}
	&\int_{\bbr} |F_1|^2 dy \leq C (\delta_l + \delta_r)\delta_l^2 \delta_r^2 \varepsilon \left( e^{\frac{c \delta_l^2}{\varepsilon} \tau} + e^{\frac{c \delta_r^2}{\varepsilon} \tau} \right), \quad -t_0\leq\tau\leq 0, \\
	&\int_{\bbr} |F_2|^2 dy \leq C \delta_l\delta_r^2\varepsilon e^{ \frac{c \delta_l^2\tau}{\varepsilon}} + C\delta_l^2\delta_r\varepsilon e^{ \frac{c \delta_r^2\tau}{\varepsilon}}, \quad -t_0\leq\tau\leq 0.
\end{align*}
Combining the above estimates and using the smallness of $\delta_l \sim \delta_r \sim \delta$, we get the estimate \eqref{estimate anti1}.
\end{proof}

We now turn to the $L^2$-relative entropy estimates for the original variables.
\begin{lemma} \label{lemma 0before}
	Under the hypotheses of Proposition \ref{priori estimate before}, it holds that for $\tau \in \left[-t_0, -A\frac{\varepsilon}{\delta^2}\right]$,
	\begin{equation} \label{estimate 0before}
		\|(\phi, \psi)(\tau)\|^2 + \varepsilon \int_{-t_0}^{\tau} \|\psi_y\|^2 d\tau \leq C e^{-cA}\delta\varepsilon + C \delta^2\varepsilon^{-1} \int_{-t_0}^{\tau} \|\Phi_y\|^2 d\tau.
	\end{equation}
\end{lemma}
\begin{proof}
	Multiplying \eqref{per before}$_1$ by $-(p(v) - p(\bar{v}))$, it holds that
	\begin{equation} \label{phi2 before}
		(Q(v|\bar{v}))_\tau + \bar{v}_{\tau} p(v|\bar{v}) + \psi_y (p(v) -p(\bar{v})) =0.
	\end{equation}	
	Then multiplying \eqref{per before}$_2$ by $\psi$ leads to
	\begin{equation} \label{psi2 before}
		\left( \frac{\psi^2}{2} \right)_\tau + \psi(p(v)-p(\bar{v}))_y + \frac{\varepsilon}{v} |\psi_y|^2
		= \varepsilon \left( \frac{\psi\psi_y}{v} \right)_y + Q_{3y}\psi - F_{1y}\psi - F_{2y}\psi,
	\end{equation}
	where $Q_3 = \varepsilon \bar{u}_y\left( \frac{1}{v} - \frac{1}{\bar{v}}\right) $.
	Now we add \eqref{phi2 before}-\eqref{phi2 before} together, then take the integration over $[-t_0, \tau]\times\bbr$ for $\tau \in \left[-t_0, -A\frac{\varepsilon}{\delta^2}\right]$ to obtain
	\begin{align}
		\begin{aligned} \label{need estimate 0before}
			&\|(\phi, \psi)(\tau)\|^2 + \varepsilon \int_{-t_0}^{\tau} \|\psi_y\|^2 d\tau \\
			&\leq C \left| \int_{-t_0}^{\tau}\int_{\bbr} \bar{v}_{\tau} p(v|\bar{v}) dyd\tau \right| + C \left| \int_{-t_0}^{\tau}\int_{\bbr} (Q_{3y} - F_{1y} - F_{2y})\psi dyd\tau \right|.
		\end{aligned}
	\end{align}
	Wherein, by Taylor's formula, it holds that $Q(v|\bar{v}) \sim \phi^2, p(v|\bar{v}) \sim \phi^2$. Now we estimate the terms on the right-hand side of \eqref{need estimate 0before}.
	First of all, by H\"{o}lder's inequality and Lemma \ref{lemma-shock}, one has
	\begin{align*}
		C \left| \int_{-t_0}^{\tau}\int_{\bbr} \bar{v}_{\tau} p(v|\bar{v}) dyd\tau \right|
		\leq C \int_{-t_0}^{\tau}\int_{\bbr} \bar{u}_y \phi^2 dyd\tau
		\leq C \delta^2\varepsilon^{-1} \int_{-t_0}^{\tau} \|\Phi_y\|^2 d\tau.
	\end{align*}
	By Young's inequality, H\"{o}lder's inequality and Lemma \ref{lemma-shock}, we get
	\begin{align*}
		&C \left| \int_{-t_0}^{\tau}\int_{\bbr} Q_{3y}\psi dyd\tau \right|
		=C \left| \int_{-t_0}^{\tau}\int_{\bbr} Q_3\psi_y dyd\tau \right| \\
		&\leq \frac{\varepsilon}{8} \int_{-t_0}^{\tau} \|\psi_y\|^2 d\tau + C \varepsilon\int_{-t_0}^{\tau}\int_{\bbr} |\bar{u}_y \phi|^2 dyd\tau \\
		&\leq \frac{\varepsilon}{8} \int_{-t_0}^{\tau} \|\psi_y\|^2 d\tau + C \delta^4\varepsilon^{-1} \int_{-t_0}^{\tau} \|\Phi_y\|^2 d\tau.
	\end{align*}
	Similar to \eqref{estimate error before}, we can get
	\begin{align*}
		&C \left| \int_{-t_0}^{\tau}\int_{\bbr} F_{1y}\psi dyd\tau \right| \\
		&\leq \frac{1}{8} \varepsilon \int_{-t_0}^{\tau}\|\psi_y\|^2 d\tau + \frac{1}{8}\sup_{-t_0 \leq \tau \leq -A\frac{\varepsilon}{\delta^2}} \|\psi\|^2 + C \varepsilon^{-\frac{1}{2}} \left( \int_{-t_0}^{\tau} \|F_{1y}\|_{L^1}^{\frac{4}{3}} d\tau \right)^{\frac{3}{2}} \\
		&\leq \frac{1}{8} \varepsilon \int_{-t_0}^{\tau}\|\psi_y\|^2 d\tau + \frac{1}{8}\sup_{-t_0 \leq \tau \leq -A\frac{\varepsilon}{\delta^2}} \|\psi\|^2 + C (\delta_l + \delta_r)^2 \varepsilon \left( \frac{\delta_r^2}{\delta_l} e^{ \frac{c \delta_l^2\tau}{\varepsilon}} + \frac{\delta_l^2}{\delta_r} e^{ \frac{c \delta_r^2\tau}{\varepsilon}} \right) \\
		&\leq \frac{1}{8} \varepsilon \int_{-t_0}^{\tau}\|\psi_y\|^2 d\tau + \frac{1}{8}\sup_{-t_0 \leq \tau \leq -A\frac{\varepsilon}{\delta^2}} \|\psi\|^2 + C e^{-cA}\delta^3\varepsilon,
	\end{align*}
	and
	\begin{align*}
		&C \left| \int_{-t_0}^{\tau}\int_{\bbr} F_{2y}\psi dyd\tau \right| \\
		&\leq \frac{1}{8} \varepsilon \int_{-t_0}^{\tau}\|\psi_y\|^2 d\tau + \frac{1}{8}\sup_{-t_0 \leq \tau \leq -A\frac{\varepsilon}{\delta^2}} \|\psi\|^2 + C \varepsilon^{-\frac{1}{2}} \left( \int_{-t_0}^{\tau} \|F_{2y}\|_{L^1}^{\frac{4}{3}} d\tau \right)^{\frac{3}{2}} \\
		&\leq \frac{1}{8} \varepsilon \int_{-t_0}^{\tau}\|\psi_y\|^2 d\tau + \frac{1}{8}\sup_{-t_0 \leq \tau \leq -A\frac{\varepsilon}{\delta^2}} \|\psi\|^2 + C \varepsilon \left( \frac{\delta_r^2}{\delta_l} e^{ \frac{c \delta_l^2\tau}{\varepsilon}} + \frac{\delta_l^2}{\delta_r} e^{ \frac{c \delta_r^2\tau}{\varepsilon}} \right) \\
		&\leq \frac{1}{8} \varepsilon \int_{-t_0}^{\tau}\|\psi_y\|^2 d\tau + \frac{1}{8}\sup_{-t_0 \leq \tau \leq -A\frac{\varepsilon}{\delta^2}} \|\psi\|^2 + C e^{-cA}\delta\varepsilon,
	\end{align*}
	where we have used
	\begin{align*}
		&\|F_{1y}\|_{L^1} \leq C \delta_l\delta_r(\delta_l + \delta_r) \left( e^{ \frac{c \delta_l^2\tau}{\varepsilon}} +  e^{ \frac{c \delta_r^2\tau}{\varepsilon}} \right), \quad -t_0\leq\tau\leq 0,\\
		&\|F_{2y}\|_{L^1} \leq C \delta_l\delta_r \left( e^{ \frac{c \delta_l^2\tau}{\varepsilon}} +  e^{ \frac{c \delta_r^2\tau}{\varepsilon}} \right), \quad -t_0\leq\tau\leq 0.
	\end{align*}
	
	The combination of the above estimates finishes the proof of Lemma \ref{lemma 0before}.	
\end{proof}

\subsection{First-order derivative estimates}

In this subsection, we derive the first-order derivative estimates for the original variables $(\phi, \psi)$. We commence with the density variable.
\begin{lemma} \label{lemma phi 1before}
Under the assumption of Proposition \ref{priori estimate before}, it holds that for $\tau \in \left[-t_0, -A\frac{\varepsilon}{\delta^2}\right]$,
  \begin{equation} \label{estimate phi 1before}
  	\begin{aligned}
  		&\| \phi_y(\tau) \|^2 + \varepsilon^{-1}\int_{-t_0}^{\tau} \|\phi_y(\tau)\|^2 d\tau \\
  		&\leq Ce^{-cA}\delta^3\varepsilon^{-1} + C\varepsilon^{-2} \|\psi\|^2 + C\varepsilon^{-1} \int_{-t_0}^{\tau}\|\psi_y\|^2 d\tau + C\delta^4\varepsilon^{-3} \int_{-t_0}^{\tau}\left( \|\Psi_y\|^2 + \|\Phi_y\|^2 \right) d\tau.
  	\end{aligned}
  \end{equation}
\end{lemma}
\begin{proof}
	Multiplying the equation $\eqref{per before}_2$ by $\frac{\phi_y}{v}$ and integrating the resulting equation over $\bbr$ lead to
	\begin{align*}
	\begin{aligned}
	&\frac{\varepsilon}{2} \frac{d}{d\tau} \int_\bbr \left(\frac{\phi_y}{v}\right)^2 dy \\
	&= \int_\bbr \frac{\phi_y}{v}(p(v) -p(\bar{v}))_y dy + \int_\bbr \frac{\phi_y}{v}\psi_\tau dy - \varepsilon \int_\bbr \frac{\phi_y}{v} \left( \left(\frac{\bar{v}_y}{v}\right)_\tau - \left(\frac{u^s_{ly}}{v^s_l} + \frac{u^s_{ry}}{v^s_r}\right)_y \right) dy \\
	&\quad + \int_\bbr F_{2y}\frac{\phi_y}{v} dy,
	\end{aligned}
	\end{align*}
where we have used the fact that
$$\varepsilon\left(\frac{u_y}{v}\right)_y = \varepsilon\left(\frac{v_\tau}{v}\right)_y = \varepsilon(\ln v)_{\tau y} = \varepsilon\left(\frac{v_y}{v}\right)_\tau = \varepsilon\left(\frac{\phi_y}{v}\right)_\tau + \varepsilon\left(\frac{\bar{v}_y}{v}\right)_\tau.$$
Now we estimate the right-hand side terms respectively.   
First, by Young's inequality, H\"{o}lder's inequality, Taylor's formula, and Lemma \ref{lemma-shock}, it holds that
  \begin{align*}
  	&\int_\bbr \frac{\phi_y}{v}(p(v) -p(\bar{v}))_y dy 
  	= \int_\bbr \frac{p'(v)}{v} (\phi_y)^2 dy + \int_\bbr \frac{\phi_y}{v}(p'(v) -p'(\bar{v}))\bar{v}_y dy \\
  	&\leq -C_1 \|\phi_y\|^2 + C\int_\bbr \bar{v}_y|\phi\phi_y| dy 
  	\leq -\frac{7}{8}C_1 \|\phi_y\|^2 + C\int_\bbr |\bar{v}_y\phi|^2 dy \\
  	&\leq -\frac{7}{8}C_1 \|\phi_y\|^2 + C\delta^4\varepsilon^{-2}\|\Phi_y\|^2.
  \end{align*}
  By \eqref{per before}, Young's inequality and Lemma \ref{lemma-shock}, one has
  \begin{align*}
  	\int_\bbr \frac{\phi_y}{v}\psi_\tau dy 
  	&= \frac{d}{d\tau}\int_\bbr \frac{\phi_y}{v}\psi dy - \int_\bbr \frac{\phi_{y\tau}}{v} \psi dy + \int_\bbr \frac{\phi_yv_\tau}{v^2} \psi dy \\
  	&= \frac{d}{d\tau}\int_\bbr \frac{\phi_y}{v}\psi dy - \int_\bbr \frac{\psi_{yy}}{v} \psi dy + \int_\bbr \frac{\phi_y}{v^2}\psi_y \psi dy + \int_\bbr \frac{\phi_y}{v^2}\bar{u}_y \psi dy \\
  	&= \frac{d}{d\tau}\int_\bbr \frac{\phi_y}{v}\psi dy + \int_\bbr \frac{\psi_{y}^2}{v} dy - \int_\bbr \frac{\bar{v}_y}{v^2}\psi_y \psi dy + \int_\bbr \frac{\phi_y}{v^2}\bar{u}_y \psi dy \\
  	&\leq \frac{d}{d\tau}\int_\bbr \frac{\phi_y}{v}\psi dy + C \|\psi_y\|^2 + \frac{1}{8}C_1 \|\phi_y\|^2 + C \|\bar{v}_y\psi\|^2 \\
  	&\leq \frac{d}{d\tau}\int_\bbr \frac{\phi_y}{v}\psi dy + C \|\psi_y\|^2 + \frac{1}{8}C_1 \|\phi_y\|^2 + C\delta^4\varepsilon^{-2} \|\Psi_y\|^2.
  \end{align*}
By Young's inequality, H\"{o}lder's inequality, Lemma \ref{lemma-shock} and Lemma \ref{lem:shock-interact}, one has
\begin{align*}
	&- \varepsilon \int_\bbr \frac{\phi_y}{v} \left( \left(\frac{\bar{v}_y}{v}\right)_\tau - \left(\frac{u^s_{ly}}{v^s_l} + \frac{u^s_{ry}}{v^s_r}\right)_y \right) dy \\
	&= - \varepsilon\int_\bbr \frac{\phi_y}{v}\frac{\bar{v}_{y\tau}}{v} dy + \varepsilon\int_\bbr \frac{\phi_y}{v}\frac{\bar{v}_yv_\tau}{v^2} dy + \varepsilon\int_\bbr \frac{\phi_y}{v}\frac{u^s_{lyy}}{v^s_l} dy -\varepsilon\int_\bbr \frac{\phi_y}{v}\frac{u^s_{ly}v^s_{ly}}{(v^s_l)^2} dy \\
	&\quad + \varepsilon\int_\bbr \frac{\phi_y}{v}\frac{u^s_{ryy}}{v^s_r} dy -\varepsilon\int_\bbr \frac{\phi_y}{v}\frac{u^s_{ry}v^s_{ry}}{(v^s_r)^2} dy \\
	&= \varepsilon\int_\bbr \frac{\phi_y}{v}u^s_{lyy}\left(\frac{1}{v^s_l} - \frac{1}{v}\right) dy + \varepsilon\int_\bbr \frac{\phi_y}{v}u^s_{ryy}\left(\frac{1}{v^s_r} - \frac{1}{v}\right) dy  + \varepsilon\int_\bbr \frac{\bar{v}_y}{v^3}\phi_y\psi_y dy \\
	&\quad + \varepsilon \int_\bbr \frac{\phi_y}{v^3}(v^s_{ly}u^s_{ry} + u^s_{ly}v^s_{ry})dy + \varepsilon\int_\bbr \frac{\phi_y}{v}v^s_{ly}u^s_{ly}\left(\frac{1}{v^2} - \frac{1}{(v^s_l)^2}\right) dy \\
	&\quad + \varepsilon\int_\bbr \frac{\phi_y}{v}v^s_{ry}u^s_{ry}\left(\frac{1}{v^2} - \frac{1}{(v^s_r)^2}\right) dy \\
	&\leq \frac{1}{8}C_1 \|\phi_y\|^2 + C\varepsilon^2\|(u^s_{lyy} + v^s_{ly}u^s_{ly})\phi\|^2 + C\varepsilon^2\|(u^s_{ryy} + v^s_{ry}u^s_{ry})\phi\|^2 + C\varepsilon^2\|\bar{v}_y\|_{L^\infty}^2\|\psi_y\|^2 \\
	&\quad + C\varepsilon^2\|(u^s_{lyy} + v^s_{ly}u^s_{ly})(v^s_r - v_*)\|^2 + C\varepsilon^2\|(u^s_{ryy} + v^s_{ry}u^s_{ry})(v^s_l - v_*)\|^2 + C \varepsilon^2\|v^s_{ly}v^s_{ry}\|^2  \\
	&\leq \frac{1}{8}C_1 \|\phi_y\|^2 + C\delta^6\varepsilon^{-2} \|\Phi_y\|^2 + C\delta^4 \|\psi_y\|^2 + C \frac{(\delta_l^3 + \delta_r^3)\delta_l^2\delta_r^2}{\varepsilon} \left( e^{ \frac{c \delta_l^2\tau}{\varepsilon}} + e^{ \frac{c \delta_r^2\tau}{\varepsilon}} \right) \\
	&\quad + C\frac{\delta_l^4\delta_r^3}{\varepsilon} e^{\frac{c \delta_l^2\tau}{\varepsilon}} + C\frac{\delta_l^3\delta_r^4}{\varepsilon} e^{ \frac{c \delta_r^2\tau}{\varepsilon}}.
\end{align*}
Similarly, it holds that
\begin{align*}
	\int_\bbr F_{2y}\frac{\phi_y}{v} dy
	\leq \frac{1}{8}C_1 \|\phi_y\|^2 + C \|F_{2y}\|^2 
	\leq \frac{1}{8}C_1 \|\phi_y\|^2 + C \frac{\delta_l^2\delta_r^2(\delta_l + \delta_r)}{\varepsilon} \left( e^{ \frac{c \delta_l^2\tau}{\varepsilon}} + e^{ \frac{c \delta_r^2\tau}{\varepsilon}} \right).
\end{align*}

Combining the above estimates, integrating over $[-t_0, \tau]$ for $\tau \in \left[-t_0, -A\frac{\varepsilon}{\delta^2}\right]$, by Young's inequality, we can obtain the estimate \eqref{estimate phi 1before}.
\end{proof}

Now we give the first-order derivative estimate of the velocity.
\begin{lemma} \label{lemma psi 1before}
	Under the assumption of Proposition \ref{priori estimate before}, it holds that for $\tau \in \left[-t_0, -A\frac{\varepsilon}{\delta^2}\right]$,
	\begin{equation} \label{estimate psi 1before}
		\begin{aligned}
			&\| \psi_y(\tau) \|^2 + \varepsilon \int_{-t_0}^{\tau}\|\psi_{yy}\|^2 d\tau \leq Ce^{-cA}\delta^3\varepsilon^{-1} + C\varepsilon^{-1} \int_{-t_0}^{\tau}\|\phi_y\|^2 d\tau \\
			&\quad + C\Big(\delta^4\varepsilon^{-1} + \varepsilon \sup_{-t_0 \leq \tau \leq -A\frac{\varepsilon}{\delta^2}}\|\phi_y\|^4\Big) \int_{-t_0}^{\tau} \|\psi_y\|^2 d\tau + C\delta^4\varepsilon^{-3} \int_{-t_0}^{\tau} \|\Phi_y\|^2 d\tau.
		\end{aligned}		
	\end{equation}
\end{lemma}
\begin{proof}
	Multiplying the equation $\eqref{per before}_2$ by $-\psi_{yy}$ and integrating the resulting equation over $\bbr$ lead to
	\begin{align*}
		\begin{aligned}
			&\frac{1}{2} \frac{d}{d\tau} \int_\bbr \psi_y^2 dy + \varepsilon\int_\bbr \frac{\psi_{yy}^2}{v} dy \\
			&= \int_\bbr (p(v) -p(\bar{v}))_y\psi_{yy} dy + \varepsilon\int_\bbr \frac{\phi_y}{v^2}\psi_y\psi_{yy} dy + \varepsilon\int_\bbr \frac{\bar{v}_y}{v^2}\psi_y\psi_{yy} dy - \int_\bbr (Q_{3y} - F_{1y} - F_{2y}) \psi_{yy} dy,
		\end{aligned}
	\end{align*}
	where we have used the fact that
	$$\varepsilon\left(\frac{u_y}{v} - \frac{\bar{u}_y}{\bar{v}}\right)_y (-\psi_{yy}) = -\varepsilon\frac{\psi_{yy}^2}{v} + \varepsilon\frac{\phi_y}{v^2}\psi_y\psi_{yy} + \varepsilon \frac{\bar{v}_y}{v^2}\psi_y\psi_{yy} - \varepsilon \left(\bar{u}_y\left(\frac{1}{v} - \frac{1}{\bar{v}}\right)\right)_y \psi_{yy}.$$
We now estimate the right-hand side terms respectively. 
By Young's inequality and Lemma \ref{lemma-shock}, we can infer the following estimates
\begin{align*}
	&\int_\bbr (p(v) -p(\bar{v}))_y\psi_{yy} dy \leq \frac{1}{8}\varepsilon \left\|\frac{\psi_{yy}}{\sqrt{v}}\right\|^2 + C \varepsilon^{-1}\|\phi_y\|^2 + C\varepsilon^{-1} \int_\bbr |\bar{v}_y\phi|^2 dy \\
	&\leq \frac{1}{8}\varepsilon \left\|\frac{\psi_{yy}}{\sqrt{v}}\right\|^2 + C \varepsilon^{-1}\|\phi_y\|^2 + C\delta^4\varepsilon^{-3} \|\Phi_y\|^2.
\end{align*}
By H\"{o}lder's inequality, Sobolev's inequality and Young's inequality, one has
\begin{align*}
	\varepsilon\int_\bbr \frac{\phi_y}{v^2}\psi_y\psi_{yy} dy \leq C \varepsilon \|\psi_{yy}\| \|\psi_y\|_{L^\infty}\|\phi_y\| 
	\leq C \varepsilon \|\psi_{yy}\|^{\frac{3}{2}} \|\psi_y\|^{\frac{1}{2}} \|\phi_y\|
	\leq \frac{1}{8}\varepsilon \left\|\frac{\psi_{yy}}{\sqrt{v}}\right\|^2 + C\varepsilon \|\phi_y\|^4 \|\psi_y\|^2.
\end{align*}
By Young's inequality, H\"{o}lder's inequality and Lemma \ref{lemma-shock}, one get
\begin{align*}
	\varepsilon \int_\bbr \frac{\bar{v}_y}{v^2}\psi_y\psi_{yy} dy
	\leq \frac{1}{8}\varepsilon \left\|\frac{\psi_{yy}}{\sqrt{v}}\right\|^2 + \varepsilon \int_\bbr |\bar{v}_y \psi_y|^2 dy \leq \frac{1}{8}\varepsilon \left\|\frac{\psi_{yy}}{\sqrt{v}}\right\|^2 + C\delta^4\varepsilon^{-1} \|\psi_y\|^2,
\end{align*}
and
\begin{align*}
	&- \int_{\bbr} Q_{3y}\psi_{yy} dy = - \varepsilon \int_\bbr \left(\bar{u}_y\left(\frac{1}{v} - \frac{1}{\bar{v}}\right)\right)_y \psi_{yy} dy \\
	&\leq \frac{1}{8}\varepsilon \left\|\frac{\psi_{yy}}{\sqrt{v}}\right\|^2 + C \varepsilon \int_\bbr \bar{u}_{yy}^2\phi^2 dy + C \varepsilon \int_\bbr \bar{u}_y^2\phi_y^2 dy + C \varepsilon \int_\bbr |\bar{u}_y\bar{v}_y\phi|^2 dy \\
	&\leq \frac{1}{8}\varepsilon \left\|\frac{\psi_{yy}}{\sqrt{v}}\right\|^2 + C \delta^6\varepsilon^{-3} \|\Phi_y\|^2 + C \delta^4\varepsilon^{-1} \|\phi_y\|^2.
\end{align*}
By Young's inequality, Lemma \ref{lemma-shock} and Lemma \ref{lem:shock-interact}, we obtain
\begin{align*}
	\int_\bbr F_{1y}\psi_{yy} dy
	\leq \frac{1}{8}\varepsilon \left\|\frac{\psi_{yy}}{\sqrt{v}}\right\|^2 + C \varepsilon^{-1} \|F_{1y}\|^2
	\leq \frac{1}{8}\varepsilon \left\|\frac{\psi_{yy}}{\sqrt{v}}\right\|^2 + C \frac{(\delta_l^3 + \delta_r^3)\delta_l^2\delta_r^2}{\varepsilon^2} \left( e^{ \frac{c \delta_l^2\tau}{\varepsilon}} + e^{ \frac{c \delta_r^2\tau}{\varepsilon}} \right),
\end{align*}
and
\begin{align*}
	\int_\bbr F_{2y}\psi_{yy} dy 
	\leq \frac{1}{8}\varepsilon \left\|\frac{\psi_{yy}}{\sqrt{v}}\right\|^2 + C \varepsilon^{-1} \|F_{2y}\|^2
	\leq \frac{1}{8}\varepsilon \left\|\frac{\psi_{yy}}{\sqrt{v}}\right\|^2 + C \frac{\delta_l^2\delta_r^2(\delta_l + \delta_r)}{\varepsilon^2} \left(e^{ \frac{c \delta_l^2\tau}{\varepsilon}} + e^{ \frac{c \delta_r^2\tau}{\varepsilon}}\right).
\end{align*}

Combining the above estimates, integrating over $[-t_0, \tau]$ for $\tau \in \left[-t_0, -A\frac{\varepsilon}{\delta^2}\right]$, and using the smallness of $\delta_l \sim \delta_r \sim \delta$, we can obtain the estimate \eqref{estimate psi 1before}.
\end{proof}

\bigskip

%
%

\section{Proof of Uniform {\it a priori} Estimates in Proposition \ref{priori estimate after}}
\setcounter{equation}{0}

This section is devoted to the proof of the uniform {\it a priori} estimates after the approximate collision time $\left[-A\frac{\varepsilon}{\delta^2}, +\infty\right)$ in Proposition \ref{priori estimate after}. The argument is based on the relative entropy framework and the $a$-contraction method with time-dependent shifts.

Firstly, from the {\it a priori} assumption in \eqref{priori assump after}, it is straightforward to see that $\frac{v_+}{4} \leq v \leq 4 v_+$ if $\delta$ is suitably small. Similarly, before proving Proposition \ref{priori estimate after}, we present some useful explicit expressions on the relative quantities associated to the pressure $p(v) = v^{-\gamma}$, and the potential energy $Q(v) := \frac{v^{1-\gamma}}{\gamma - 1}$ (i.e., $Q'(v) = -p(v)$) based on Taylor expansions, which are significant for the {\it a priori} estimates, especially for relative entropy estimates. The relative pressure $p(v|\tilde{v})$ and relative potential energy $Q(v|\tilde{v})$ can be defined as  
$$p(v|\tilde{v}) := p(v) -p(\tilde{v}) - p'(\tilde{v})(v - \tilde{v}), \quad Q(v|\tilde{v}) := Q(v) -Q(\tilde{v}) - Q'(\tilde{v})(v - \tilde{v}).$$

To lay the groundwork for the {\it a priori} estimates, we first present several useful lemmas.

\begin{lemma} \label{p}
	For shock wave strength $\delta_2=v_+-v^*$ suitably small, 
	we have
	$$\left|\frac{p(v^s)- p(v_+)}{v^s - v_+} - \frac{p(v^s)- p(v^*)}{v^s - v^*} - \frac{\delta_2p''(v^*)}{2}\right| \leq C\delta_2^2.$$
\end{lemma}
\begin{proof}
	For $p(v^s)$, using Taylor expansions at $v_+$ and $v^*$
	$$p(v^s) = p(v_+) + p'(v_+)(v^s - v_+) + \frac{p''(v_+)}{2!}(v^s - v_+)^2 + \frac{p'''(\zeta_+)}{3!}(v^s - v_+)^3,$$
	$$p(v^s) = p(v^*) + p'(v^*)(v^s - v^*) + \frac{p''(v^*)}{2!}(v^s - v^*)^2 + \frac{p'''(\zeta_m)}{3!}(v^s - v^*)^3,$$
	one has
	$$\left|\frac{p(v^s) - p(v_+)}{(v^s - v_+)} - p'(v_+) - \frac{p''(v_+)}{2}(v^s - v_+)\right| \leq C (v^s - v_+)^2 \leq C\delta_2^2,$$
	$$\left|\frac{p(v^s) - p(v^*)}{(v^s - v^*)} - p'(v^*) - \frac{p''(v^*)}{2}(v^s - v^*)\right| \leq C (v^s - v^*)^2 \leq C\delta_2^2,$$
	where $\zeta_+$ and $\zeta_m$ are between $v$ and $\bar{v}$.
	Therefore, we have
	\begin{align*}
		&\left|\frac{p(v^s) - p(v_+)}{(v^s - v_+)} - \frac{p(v^s) - p(v^*)}{(v^s - v^*)} - \frac{\delta_2p''(v^*)}{2}\right| \\
		&\leq \left|\frac{p(v^s) - p(v_+)}{(v^s - v_+)} - p'(v_+) - \frac{p''(v_+)}{2}(v^s - v_+)\right| + \left|\frac{p(v^s) - p(v^*)}{(v^s - v^*)} - p'(v^*) - \frac{p''(v^*)}{2}(v^s - v^*)\right| \\
		&\quad + \left|p'(v_+) - p'(v^*) - \delta_2p''(v^*)\right| + \left|\frac{p''(v_+)}{2}(v^s - v_+) - \frac{p''(v^*)}{2}(v^s - v^*) + \frac{\delta_2p''(v^*)}{2}\right| \leq C\delta_2^2,
	\end{align*}
	where we have used
	$$\left|p'(v_+) - p'(v^*) - \delta_2p''(v^*)\right| \leq C\delta_2^2,$$
	and
	\begin{align*}
		&\left|\frac{p''(v_+)}{2}(v^s - v_+) - \frac{p''(v^*)}{2}(v^s - v^*) + \frac{\delta_2p''(v^*)}{2}\right| \\
		&= \left|\frac{p''(v_+)}{2}(v^s - v_+) - \frac{p''(v^*)}{2}(v^s - v_+)\right| \leq C\delta_2^2.
	\end{align*}
	This completes the proof.
\end{proof}

The $a$-contraction method with time-dependent shifts depends on the following Poincar\'e-type inequality (see Lemma 2.9 in \cite{KV1} or Lemma 1.1 in \cite{KVW}).
\begin{lemma}\label{poincare}
	For any $f:[0, 1] \to \bbr$ satisfying $\int_0^1 \eta(1 - \eta)|f'|^2 d\eta < \infty$, 
	\begin{equation*}
		\int_0^1 \Big|f-\int_0^1 f d\eta \Big|^2 d\eta \leq \frac{1}{2}\int_0^1  \eta(1 - \eta)|f'|^2 d\eta.
	\end{equation*}	
\end{lemma}

We now give the proof of Proposition \ref{priori estimate after} for the uniform $H^1$ {\it a priori} estimates.
First, in subsection 6.1, we establish the lower-order estimates, which are derived from the relative entropy method and the $a$-contraction method with time-dependent shifts, as presented in Lemmas \ref{lemma 0after}-\ref{lemma leading}. In subsection 6.2, the first-order derivative estimates are proved in Lemmas \ref{lemma phi1 after} and \ref{lemma psi1 after}.

\subsection{Relative entropy estimates}

We start with the lower-order $L^2$-relative entropy estimates.

\begin{lemma} \label{lemma 0after}
	Under the hypotheses of Proposition \ref{priori estimate after}, it holds that for $\tau \in \left[- A\frac{\varepsilon}{\delta^2}, \tau_1\right]$,
	\begin{align}
		\begin{aligned} \label{estimate 0after}
			&\|(\tilde{\phi},\tilde{\psi})(\tau)\|^2 + \delta_2 \int_{- A\frac{\varepsilon}{\delta^2}}^{\tau} |\dot{X}(\tau)|^2 d\tau + \int_{- A\frac{\varepsilon}{\delta^2}}^{\tau}(G_1(\tau) + G^R(\tau) + G^S(\tau) + D(\tau)) d\tau \\
			&\leq C \|(\tilde{\phi}_0,\tilde{\psi}_0)\|^2 + C \Big(\delta_1^2\delta_2\varepsilon +  \delta_1^2\delta_2^4\varepsilon^{-2}\kappa^3 +  \delta_1\varepsilon^2\kappa^{-1} + \delta_1^2\delta_2^2\kappa +   \delta_1^2\delta_2\varepsilon^3\kappa^{-2}  +  \delta_1^{\frac{1}{2}}\varepsilon^\frac{3}{2}\kappa^{-\frac{1}{2}}\Big) \\
			&\leq C \|(\tilde{\phi}_0,\tilde{\psi}_0)\|^2 + C \delta_1^{\frac{1}{2}}\varepsilon,
		\end{aligned}
	\end{align}
	where
	\begin{align*}
		&G_1 = \int_{\bbr} a^{-X}_y (\tilde{\psi} + s_2\tilde{\phi})^2 dy, 		\\
		&G^R = \int_{\bbr} u^r_y \tilde{\phi}^2 dy, \\
		&G^S = \int_{\bbr} (v^s)^{-X}_y \tilde{\psi}^2 dy, \\
		&D = \varepsilon\int_{\bbr} |\tilde{\psi}_y|^2 dy.
	\end{align*}
	It is worth noting here that we take $\kappa=\varepsilon$.
\end{lemma}

Set $a^{-X} :=a(y - s_2\tau -X(\tau))$. By virtue of the perturbation equation \eqref{per after}, we obtain the following lemma.
\begin{lemma} \label{lemma 0Pafter}
	Let $a$ be the weighted function defined by \eqref{weight}, it holds
	\begin{align}
		\begin{aligned} \label{estimate 0Pafter}
			&\frac{d}{d\tau}\int_{\bbr} a^{-X}\left(Q(v|\tilde{v}) + \frac{1}{2}\tilde{\psi}^2\right)dy + \mathbf{G}_1(\tau) + \mathbf{G}_2(\tau) + \mathbf{G}^R(\tau) + \mathbf{D}(\tau) \\
			&= \dot{X}(\tau)Y(\tau) + \sum_{i=1}^{11}\mathbf{B}_i(\tau),
		\end{aligned}
	\end{align}
	where
	\begin{align}
		\begin{aligned} \label{Y}
			Y(\tau) = \int_{\bbr}\left(- a^{-X}(v^s)_y^{-X}p'(\tilde{v})\tilde{\phi} + a^{-X} (u^s)_y^{-X}\tilde{\psi} - a^{-X}_y\left(Q(v|\tilde{v}) + \frac{1}{2}\tilde{\psi}^2\right)\right)dy,
		\end{aligned}
	\end{align}
	\begin{align*}
		&\mathbf{G}_1 = - \frac{1}{2s_2}\int_{\bbr}a^{-X}_y p'(\tilde{v})(\tilde{\psi} + s_2\tilde{\phi})^2 dy, \\
		&\mathbf{G}_2 = s_2 \int_{\bbr} a^{-X}_y \frac{\tilde{\psi}^2}{2}dy, \\
		&\mathbf{G}^R = \int_{\bbr} a^{-X} u^r_y p(v|\tilde{v})dy, \\
		&\mathbf{D} = \varepsilon\int_{\bbr}\frac{a^{-X}}{v} |\tilde{\psi}_y|^2dy,	
	\end{align*}
	and the bad terms are given by
	\begin{align*}
		&\mathbf{B}_1 = - \frac{1}{2s_2}\int_{\bbr}a^{-X}_y p'(\tilde{v})\tilde{\psi}^2 dy, \\
		&\mathbf{B}_2 = \int_{\bbr}a^{-X}_y p(v|\tilde{v})\tilde{\psi} dy, \\
		&\mathbf{B}_3 = - s_2 \int_{\bbr}a^{-X}_y \left( Q(v|\tilde{v}) + \frac{p'(\tilde{v})}{2}\tilde{\phi}^2 \right) dy, \\
		&\mathbf{B}_4 = \frac{1}{2s_2}\int_{\bbr}a^{-X} p''(\tilde{v})(v^s)^{-X}_y\tilde{\psi}^2dy, \\
		&\mathbf{B}_5 = -\frac{1}{s_2}\int_{\bbr}a^{-X} p''(\tilde{v})(v^s)^{-X}_y\tilde{\psi}(\tilde{\psi} + s_2\tilde{\phi})dy, \\
		&\mathbf{B}_6 = \frac{1}{2s_2}\int_{\bbr}a^{-X} p''(\tilde{v})(v^s)^{-X}_y(\tilde{\psi} + s_2\tilde{\phi})^2dy, \\
		&\mathbf{B}_7 = s_2 \int_{\bbr} a^{-X} (v^s)^{-X}_y \left( p(v|\tilde{v}) - \frac{p''(\tilde{v})}{2}\tilde{\phi}^2 \right) dy, \\
		&\mathbf{B}_8 = - \varepsilon\int_{\bbr}a^{-X}_y \tilde{\psi} \frac{\tilde{\psi}_y}{v}dy, \\
		&\mathbf{B}_9 = - \varepsilon\int_{\bbr}a^{-X}_y \tilde{u}_y\tilde{\psi} \left(\frac{1}{v} - \frac{1}{\tilde{v}}\right)dy, \\
		&\mathbf{B}_{10} = - \varepsilon\int_{\bbr} a^{-X} \tilde{u}_y\tilde{\psi}_y \left(\frac{1}{v} - \frac{1}{\tilde{v}}\right)dy, \\
		&\mathbf{B}_{11} = - \int_{\bbr} (F_{3y} + F_{4y})a^{-X}\tilde{\psi} dy.
	\end{align*}
\end{lemma}

\begin{proof}
	Multiplying \eqref{per after}$_1$ by $-a^{-X}(p(v) - p(\tilde{v}))$, it holds that
	\begin{align}
		\begin{aligned} \label{phi2 after}
			&(a^{-X}Q(v|\tilde{v}))_\tau + s_2 a^{-X}_y Q(v|\tilde{v}) + \dot{X}(\tau)a^{-X}_y Q(v|\tilde{v}) + \dot{X}(\tau)a^{-X} (v^s)_y^{-X}p'(\tilde{v})\tilde{\phi} \\
			&+ a^{-X} \tilde{u}_y p(v|\tilde{v}) + (a^{-X} (p(v) -p(\tilde{v}))\tilde{\psi})_y - a^{-X}_y (p(v) -p(\tilde{v}))\tilde{\psi} - a^{-X} (p(v) -p(\tilde{v}))_y\tilde{\psi} \\
			& =0.
		\end{aligned}
	\end{align}	
	Then multiplying \eqref{per after}$_2$ by $a^{-X}\tilde{\psi}$ leads to
	\begin{align}
		\begin{aligned} \label{psi2 after}
			&\left(\frac{a^{-X}}{2}\tilde{\psi}^2\right)_\tau + s_2 a^{-X}_y \frac{\tilde{\psi}^2}{2} + \dot{X}(\tau)a^{-X}_y \frac{\tilde{\psi}^2}{2} + a^{-X}\tilde{\psi}(p(v)-p(\tilde{v}))_y - \dot{X}(\tau)a^{-X} (u^s)_y^{-X}\tilde{\psi} \\
			& = \varepsilon\left(a^{-X} \tilde{\psi} \left(\frac{u_y}{v} -\frac{\tilde{u}_y}{\tilde{v}}\right)\right)_y - \varepsilon a^{-X}_y \tilde{\psi} \frac{\tilde{\psi}_y}{v} - \varepsilon\frac{a^{-X}}{v} |\tilde{\psi}_y|^2 - \varepsilon a^{-X}_y \tilde{u}_y\tilde{\psi} \left(\frac{1}{v} - \frac{1}{\tilde{v}}\right) \\
			&\quad - \varepsilon a^{-X} \tilde{u}_y\tilde{\psi}_y \left(\frac{1}{v} - \frac{1}{\tilde{v}}\right) - (F_{3y} + F_{4y})a^{-X}\tilde{\psi}.
		\end{aligned}
	\end{align}	
	Now we add \eqref{phi2 after}-\eqref{psi2 after} together, then take the integration over  $\bbr$ to obtain
	\begin{align}
		\begin{aligned} \label{phi2psi2 after}
			&\frac{d}{d\tau} \int_{\bbr} a^{-X} \left(Q(v|\tilde{v}) + \frac{1}{2}\tilde{\psi}^2\right)dy + s_2 \int_{\bbr} a^{-X}_y Q(v|\tilde{v})dy + s_2 \int_{\bbr} a^{-X}_y \frac{\tilde{\psi}^2}{2}dy \\
			&\quad + \int_{\bbr} a^{-X} u^r_y p(v|\tilde{v})dy + \varepsilon \int_{\bbr}\frac{a^{-X}}{v} |\tilde{\psi}_y|^2dy\\
			& = \dot{X}(\tau)\int_{\bbr}\left(- a^{-X}(v^s)_y^{-X}p'(\tilde{v})\tilde{\phi} + a^{-X} (u^s)_y^{-X}\tilde{\psi} - a^{-X}_y\left(Q(v|\tilde{v}) + \frac{1}{2}\tilde{\psi}^2\right)\right)dy \\
			&\quad - \int_{\bbr} a^{-X} (u^s)^{-X}_y p(v|\tilde{v})dy + \int_{\bbr}a^{-X}_y (p(v) -p(\tilde{v}))\tilde{\psi} dy - \varepsilon \int_{\bbr}a^{-X}_y \tilde{\psi} \frac{\tilde{\psi}_y}{v}dy \\
			&\quad - \varepsilon\int_{\bbr}a^{-X}_y \tilde{u}_y\tilde{\psi} \left(\frac{1}{v} - \frac{1}{\tilde{v}}\right)dy - \varepsilon\int_{\bbr} a^{-X} \tilde{u}_y\tilde{\psi}_y \left(\frac{1}{v} - \frac{1}{\tilde{v}}\right)dy - \int_{\bbr} (F_{3y} + F_{4y})a^{-X}\tilde{\psi} dy \\
			& = \dot{X}(\tau)Y(\tau) - \int_{\bbr} a^{-X} (u^s)^{-X}_y p(v|\tilde{v})dy + \int_{\bbr}a^{-X}_y (p(v) -p(\tilde{v}))\tilde{\psi} dy + \sum_{i=8}^{11} \mathbf{B}_i.
		\end{aligned}
	\end{align}	
	By observation, the pressure terms can be reorganized as follows:
	\begin{align}
		\begin{aligned} \label{phi2psi2 after-1}
			&\int_{\bbr}a^{-X}_y (p(v) -p(\tilde{v}))\tilde{\psi} dy  - s_2 \int_{\bbr} a^{-X}_y Q(v|\tilde{v})dy \\
			&= \int_{\bbr}a^{-X}_y p'(\tilde{v})\tilde{\phi}\tilde{\psi} dy + \int_{\bbr}a^{-X}_y (p(v) -p(\tilde{v}) - p'(\tilde{v})\tilde{\phi})\tilde{\psi} dy + s_2 \int_{\bbr}a^{-X}_y \frac{p'(\tilde{v})}{2}\tilde{\phi}^2 dy \\
			&\quad - s_2 \int_{\bbr}a^{-X}_y \left( Q(v|\tilde{v}) + \frac{p'(\tilde{v})}{2}\tilde{\phi}^2 \right) dy \\
			&= \frac{1}{2s_2}\int_{\bbr}a^{-X}_y p'(\tilde{v})(\tilde{\psi} + s_2\tilde{\phi})^2 dy - \frac{1}{2s_2}\int_{\bbr}a^{-X}_y p'(\tilde{v})\tilde{\psi}^2 dy + \int_{\bbr}a^{-X}_y p(v|\tilde{v})\tilde{\psi} dy \\
			&\quad - s_2 \int_{\bbr}a^{-X}_y \left( Q(v|\tilde{v}) + \frac{p'(\tilde{v})}{2}\tilde{\phi}^2 \right) dy \\
			&= - \mathbf{G}_1 + \mathbf{B}_1 + \mathbf{B}_2 + \mathbf{B}_3,
		\end{aligned}
	\end{align}
	and
	\begin{align}
		\begin{aligned} \label{phi2psi2 after-2}
			&- \int_{\bbr} a^{-X} (u^s)^{-X}_y p(v|\tilde{v}) dy
			= s_2 \int_{\bbr} a^{-X} (v^s)^{-X}_y p(v|\tilde{v})dy \\
			&= s_2 \int_{\bbr} a^{-X} (v^s)^{-X}_y \frac{p''(\tilde{v})}{2}\tilde{\phi}^2 dy + s_2 \int_{\bbr} a^{-X} (v^s)^{-X}_y \left( p(v|\tilde{v}) - \frac{p''(\tilde{v})}{2}\tilde{\phi}^2 \right) dy \\
			&= \frac{1}{2s_2}\int_{\bbr}a^{-X} p''(\tilde{v})(v^s)^{-X}_y\tilde{\psi}^2dy - \frac{1}{s_2}\int_{\bbr}a^{-X} p''(\tilde{v})(v^s)^{-X}_y\tilde{\psi}(\tilde{\psi} + s_2\tilde{\phi})dy \\
			&\quad + \frac{1}{2s_2}\int_{\bbr}a^{-X} p''(\tilde{v})(v^s)^{-X}_y(\tilde{\psi} + s_2\tilde{\phi})^2dy + s_2 \int_{\bbr} a^{-X} (v^s)^{-X}_y \left( p(v|\tilde{v}) - \frac{p''(\tilde{v})}{2}\tilde{\phi}^2 \right) dy \\
			&= \mathbf{B}_4 + \mathbf{B}_5 + \mathbf{B}_6 + \mathbf{B}_7.
		\end{aligned}
	\end{align}
	The combination of \eqref{phi2psi2 after}, \eqref{phi2psi2 after-1} and \eqref{phi2psi2 after-2} finishes the proof of Lemma \ref{lemma 0Pafter}.	
\end{proof}

To establish the $a$-contraction property, we introduce the following decomposition for $Y(\tau)$ from Lemma \ref{lemma 0Pafter}:
$$Y(\tau) = \sum_{i=1}^{5} Y_i(\tau),$$
where
\begin{align*}
	&Y_1(\tau) = \int_{\bbr} a^{-X}(u^s)^{-X}_y\tilde{\psi} dy, \\
	&Y_2(\tau) = \int_{\bbr} \frac{a^{-X}}{s_2}(v^s)^{-X}_y p'((v^s)^{-X})\tilde{\psi} dy, \\
	&Y_3(\tau) = \int_{\bbr} \frac{a^{-X}}{s_2}(v^s)^{-X}_y(p'(\tilde{v}) - p'((v^s)^{-X}))\tilde{\psi} dy, \\
	&Y_4(\tau) = - \int_{\bbr} \frac{a^{-X}}{s_2}(v^s)^{-X}_y p'(\tilde{v})(\tilde{\psi} + s_2\tilde{\phi}) dy, \\
	&Y_5(\tau) = - \int_{\bbr} a^{-X}_y \left(Q(v|\tilde{v}) + \frac{1}{2}\tilde{\psi}^2\right)dy.
\end{align*}
Observing that
\begin{equation} \label{XY}
	\dot{X}(\tau) = -\frac{M}{\delta_2}(Y_1(\tau) + Y_2(\tau)),
\end{equation}
we can therefore deduce
\begin{equation*}
	\dot{X}(\tau)Y(\tau) = -\frac{\delta_2}{M}|\dot{X}(\tau)|^2 + \dot{X}(\tau)\sum_{i=3}^{5} Y_i(\tau).
\end{equation*}

\subsubsection{Leading order estimates}

\begin{lemma} \label{lemma leading}
	There exists $C>0$ such that
	\begin{align}
		\begin{aligned} \label{estimate leading}
			&-\frac{\delta_2}{2M}|\dot{X}(\tau)|^2 + \mathbf{B}_1 + \mathbf{B}_3 - \mathbf{G}_2 - \frac{3}{4} \mathbf{D} \\
			&\leq - C_1 \int_{\bbr} (v^s)^{-X}_y\tilde{\psi}^2 dy + C \int_{\bbr} a^{-X}_y|v^r - v^*|\tilde{\psi}^2 dy.
		\end{aligned}
	\end{align}	
\end{lemma}

\begin{proof}
	First, we rewrite the above terms in terms of the new variables $\eta$:
	\begin{equation} \label{variable eta}
		\eta := \frac{v^s( y - s_2 \tau ) - v^*}{\delta_2}.
	\end{equation}
	By the change of variable $y \in \bbr \mapsto \eta \in (0, 1)$, it is easy to see from \eqref{weight} that
	\begin{equation} \label{variable etaP}
		\frac{d\eta}{dy} = \frac{v^s_y}{\delta_2} >0, \quad a = 1 + \lambda \eta, \quad a_y(y-s_2\tau) = \lambda \frac{d\eta}{dy} > 0, \quad |a - 1| \leq \lambda.
	\end{equation}
	In terms of the new variables, we will apply the Poincar\'e-type inequality to each perturbation
	$$
	\omega := \tilde{\psi}(\tau, y + X(\tau)) \circ \eta^{-1}.
	$$
	Consider the $O(1)$-constant:
	$$ s^* = \sqrt{-p'(v^*)}, \quad \alpha^* = \frac{p''(v^*)}{2s^*}. $$
	It is easy to check that
	\begin{equation} \label{speed}
		|s_2 - s^*| \leq C\delta_2,
	\end{equation}
	which consequently implies
	\begin{equation} \label{speed1}
		\begin{aligned}
			&|(s^*)^2 - |p'(v^s)| | \leq C\delta_2, \qquad\quad \left|\frac{1}{(s^*)^2} - \frac{1}{|p'(v^s)|} \right|\leq C\delta_2, \\
			&|(s^*)^2 - |p'(\tilde{v})| | \leq C(\delta_1 + \delta_2), \quad \left|\frac{1}{(s^*)^2} - \frac{1}{|p'(\tilde{v})|} \right|\leq C(\delta_1 + \delta_2).
		\end{aligned}	
	\end{equation}
\textbf{Step 1: Estimate on} $\frac{\delta_2}{2M}|\dot{X}(\tau)|^2$.
	In light of \eqref{XY}, our approach to estimating term $|\dot{X}(\tau)|^2$ involves first estimating $Y_1, Y_2$.
	
	By the change of variable $y \to y + X(\tau)$ and apply the change of variables for $\eta, \omega$ to observe that, we have
	$$ Y_1(\tau) = \int_{\bbr} a^{-X}(u^s)_y^{-X}\tilde{\psi} dy = -s_2 \int_{\bbr} a^{-X}(v^s)_y^{-X}\tilde{\psi} dy = -s_2\delta_2 \int_{0}^{1} a \omega d\eta. $$
	Noting that \eqref{variable etaP} and \eqref{speed}, we have
	\begin{equation} \label{Y1}
		\left|Y_1(\tau) + s^*\delta_2 \int_{0}^{1} \omega d\eta\right| \leq C\delta_2(\lambda+\delta_2) \int_{0}^{1} |\omega| d\eta.
	\end{equation}
	For
	$$ Y_2(\tau) = \int_{\bbr} \frac{a^{-X}}{s_2}(v^s)_y^{-X}p'((v^s)^{-X})\tilde{\psi} dy = \delta_2 \int_{0}^1 \frac{a}{s_2}p'(v^s) \omega d\eta, $$
	by \eqref{variable etaP}, \eqref{speed} and \eqref{speed1}, it holds
	\begin{equation} \label{Y2}
		\left|Y_2(\tau) + s^*\delta_2 \int_{0}^{1} \omega d\eta\right| \leq C\delta_2(\lambda+\delta_2) \int_{0}^{1} |\omega| d\eta.
	\end{equation}
	Therefore, by \eqref{XY}, \eqref{Y1} and \eqref{Y2}, we have
	$$ \left|\dot{X}(\tau) - 2M s^* \int_{0}^{1} \omega d\eta\right| \leq \left|\sum_{i=1}^{2} \frac{M}{\delta_2}\left(Y_i(\tau) + s^*\delta_2 \int_{0}^{1} \omega d\eta\right)\right| \leq C(\lambda+\delta_2) \int_{0}^{1} |\omega| d\eta, $$
	which implies
	$$ \left(\left|2M s^* \int_{0}^{1} \omega d\eta\right| - |\dot{X}(\tau)|\right)^2 \leq C(\lambda+\delta_2)^2 \int_{0}^{1} \omega^2 d\eta. $$
	Noting that the algebraic inequality $\frac{p^2}{2} - q^2 \leq (p-q)^2$ for all $p,q\geq 0$, we have
	$$ 2M^2 (s^*)^2 \left(\int_{0}^{1} \omega d\eta\right)^2 - |\dot{X}(\tau)|^2 \leq C(\lambda+\delta_2)^2 \int_{0}^{1} \omega^2 d\eta. $$
	Therefore,
	\begin{equation} \label{X2}
		- \frac{\delta_2}{2M}|\dot{X}(\tau)|^2 \leq -M (s^*)^2\delta_2 \left(\int_{0}^{1} \omega d\eta\right)^2 + C(\lambda+\delta_2)^2 \delta_2\int_{0}^{1} \omega^2 d\eta.
	\end{equation}
%
\textbf{Step 2: Change of variable for} $\mathbf{B}_1, \mathbf{B}_4$ and $\mathbf{G}_2$.
	By \eqref{variable etaP}, \eqref{speed} and \eqref{speed1}, we have
	\begin{equation} \label{B1}
		\begin{aligned}
			\mathbf{B}_1 - \mathbf{G}_2 &= - \frac{1}{2s_2}\int_{\bbr}a^{-X}_y p'(\tilde{v})\tilde{\psi}^2 dy - \frac{s_2}{2}\int_{\bbr}a^{-X}_y \tilde{\psi}^2 dy \\
			&= \int_{\bbr} \left( - \frac{p'(\tilde{v})}{2s_2} - \frac{s_2}{2} \right) a^{-X}_y \tilde{\psi}^2 dy \\
			&\leq C(\delta_1 + \delta_2) \int_{\bbr} a^{-X}_y \tilde{\psi}^2 dy \\
			&\leq C(\delta_1 + \delta_2)\lambda \int_{\bbr} \omega^2 d\eta,
		\end{aligned}
	\end{equation}
	and
	\begin{equation} \label{B4}
		\begin{aligned}
			\mathbf{B}_4 &= \frac{1}{2s_2}\int_{\bbr}a^{-X} p''(\tilde{v})(v^s)^{-X}_y\tilde{\psi}^2dy \\
			&= \frac{\delta_2}{2s_2}\int_0^1 (1+\lambda\eta) p''(\tilde{v})(\tau, y+X(\tau)) \omega^2 d\eta \\
			&\leq \delta_2 (1+\lambda)\int_0^1 \left(\frac{p''(v^*)}{2s^*} - \frac{p''(v^*)}{2s^*} + \frac{p''(\tilde{v})(\tau, y+X(\tau))}{2s_2}\right) \omega^2 d\eta \\
			&\leq \alpha^*\delta_2(1 + C(\delta_1 + \delta_2 + \lambda))\int_0^1 \omega^2 d\eta.
		\end{aligned}
	\end{equation}	
 \textbf{Step 3: Change of variable for} $\mathbf{D}$.
	First, it follows from \eqref{variable etaP} that
	$$ \mathbf{D} = \varepsilon \int_{\bbr} \frac{a^{-X}}{v} |\tilde{\psi}_y|^2 dy \geq \varepsilon \int_{\bbr} \frac{1}{v} |\tilde{\psi}_y|^2 dy = \varepsilon \int_0^1 |\omega_\eta|^2 \frac{1}{v} \frac{d\eta}{dy} d\eta. $$
	By \eqref{shock smooth2}, we have
	$$ \varepsilon\left( \frac{u^s_y}{v^s}\right)_y = s_2^2 v^s_y + p(v^s)_y, $$
	then integrate it over $(-\infty, y)$ to obtain
	$$ \varepsilon\frac{u^s_y}{v^s} = s_2^2(v^s - v^*) + p(v^s) - p(v^*), $$
	which together with $s_2^2 = - \frac{p(v_+) - p(v^*)}{v_+ - v^*}$ yields
	\begin{align*}
		&\varepsilon\delta_2 \frac{1}{v^s} \frac{d\eta}{dy} 
		= \varepsilon\frac{v^s_y}{v^s} = -\varepsilon \frac{u^s_y}{s_2v^s} \\
		&= -\frac{1}{s_2}(s_2^2(v^s - v^*) + p(v^s) - p(v^*)) \\
		&= -\frac{1}{s_2(v_+ - v^*)}((p(v^*)- p(v_+))(v^s - v^*) + (v_+ - v^*)(p(v^s) - p(v^*))) \\
		&= -\frac{1}{s_2(v_+ - v^*)}((p(v^s)- p(v_+))(v^s - v^*) + (v^s - v^*)(p(v^*)- p(v^s)) \\
		&\quad + (v^s - v^*)(p(v^s)- p(v^*)) + (v_+ - v^s)(p(v^s) - p(v^*))) \\
		&= -\frac{1}{s_2(v_+ - v^*)}((p(v^s)- p(v_+))(v^s - v^*) + (v_+ - v^s)(p(v^s) - p(v^*))).
	\end{align*}
	Since $\eta := \frac{v^s - v^*}{\delta_2}$ and $1 - \eta := \frac{v_+ - v^s}{\delta_2}$,
	$$
	\varepsilon\frac{1}{\eta(1-\eta)} \frac{1}{v^s} \frac{d\eta}{dy} \\
	= \frac{1}{s_2}\left(\frac{p(v^s)- p(v_+)}{v^s - v_+} - \frac{p(v^s)- p(v^*)}{v^s - v^*}\right).
	$$
	By the inequality in Lemma \ref{p}
	$$\left|\frac{p(v^s)- p(v_+)}{v^s - v_+} - \frac{p(v^s)- p(v^*)}{v^s - v^*} - \frac{\delta_2p''(v^*)}{2}\right| \leq C\delta_2^2,$$
	we have
	$$\left|\varepsilon\frac{1}{\eta(1-\eta)} \frac{1}{v^s} \frac{d\eta}{dy} - \frac{\delta_2p''(v^*)}{2s^*}\right| \leq C\delta_2^2.$$
	Thus, combining this with the {\it a priori} assumption \eqref{priori assump after}, we obtain
	\begin{align}
		\begin{aligned}	\label{D}
			\mathbf{D} &\geq \varepsilon\int_0^1 \eta(1-\eta) |\omega_\eta|^2 \frac{v^s}{v} \frac{1}{\eta(1-\eta)} \frac{1}{v^s} \frac{d\eta}{dy} d\eta \\
			&\geq (1-C(\delta_1 + \delta_2 + \delta^\frac{1}{3}))(\alpha^*\delta_2 - C\delta_2^2) \int_0^1 \eta (1-\eta) |\omega_\eta|^2 d\eta \\
			&\geq \alpha^*\delta_2(1-C(\delta_1 + \delta_2 + \delta^\frac{1}{3})) \int_0^1 \eta (1-\eta) |\omega_\eta|^2 d\eta.	
		\end{aligned}
	\end{align}
%
\textbf{Step 4: Conclusion}.
	By \eqref{X2}, \eqref{B1}, \eqref{B4} and \eqref{D} with the smallness of $\lambda, \delta_1, \delta_2, \delta$, one has
	\begin{align*}	
		&-\frac{\delta_2}{2M}|\dot{X}(\tau)|^2 + \mathbf{B}_1 + \mathbf{B}_4 - \mathbf{G}_2 -\frac{3}{4}\mathbf{D} \\
		&\leq -M (s^*)^2\delta_2 \left(\int_{0}^{1} \omega d\eta\right)^2 + \alpha^*\delta_2(1 + C(\delta_1 + \delta_2 + \lambda))\int_0^1 \omega^2 d\eta \\
		&\quad -\frac{3}{4}\alpha^*\delta_2(1-C(\delta_1 + \delta_2 + \delta^\frac{1}{3})) \int_0^1 \eta (1-\eta) |\omega_\eta|^2 d\eta + \frac{\lambda}{2s_2}\int_0^1 (p'(v^s) - p'(\tilde{v}))\omega^2 d\eta \\
		&\leq \alpha^*\delta_2\left(\frac{9}{8}\int_0^1 \omega^2 d\eta - \frac{5}{8}\alpha^*\delta_2 \int_0^1 \eta (1-\eta) |\omega_\eta|^2 d\eta\right) -M (s^*)^2\delta_2 \left(\int_{0}^{1} \omega d\eta\right)^2  \\
		&\quad  + \frac{\lambda}{2s_2}\int_0^1 (p'(v^s) - p'(\tilde{v}))\omega^2 d\eta.
	\end{align*}
	Using the Poincar\'e type inequality in Lemma \ref{poincare},
	$$ \int_0^1 \eta(1-\eta) |\omega_\eta|^2 d\eta \geq 2\int_{0}^{1}\left( \omega - \int_{0}^{1}\omega d\eta \right)^2 d\eta = 2\int_{0}^{1}\omega^2 d\eta - 2\left( \int_{0}^{1}\omega d\eta \right)^2, $$
	then it holds that
	\begin{align*}
		&-\frac{\delta_2}{2M}|\dot{X}(\tau)|^2 + \mathbf{B}_1 + \mathbf{B}_4 - \mathbf{G}_2 -\frac{3}{4}\mathbf{D} \\
		&\leq -\frac{1}{8}\alpha^*\delta_2\int_0^1 \omega^2 d\eta + \left(\frac{5}{4}\alpha^* - M (s^*)^2\right)\delta_2 \left(\int_{0}^{1} \omega d\eta\right)^2 + \frac{\lambda}{2s_2}\int_0^1 (p'(v^s) - p'(\tilde{v})) \omega^2 d\eta.
	\end{align*}
	Taking $M=\frac{5\alpha^*}{4(s^*)^2}$, we have
	\begin{align*}
		&-\frac{\delta_2}{2M}|\dot{X}(\tau)|^2 + \mathbf{B}_1 + \mathbf{B}_4 - \mathbf{G}_2 -\frac{3}{4}\mathbf{D} \\
		&\leq -\frac{1}{8}\alpha^*\delta_2\int_0^1 \omega^2 d\eta + \frac{\lambda}{2s_2}\int_0^1 (p'(v^s) - p'(\tilde{v})) \omega^2 d\eta.
	\end{align*}
	Returning to the original variable yields the desired estimate \eqref{estimate leading}, so the proof of Lemma \ref{lemma leading} is completed.
\end{proof}

\noindent
{\bf Proof of Lemma \ref{lemma 0after}.}
Now let us proceed to prove Lemma \ref{lemma 0after}. First, it follows from \eqref{estimate 0Pafter} and \eqref{XY} that
\begin{align*}
	&\frac{d}{d\tau}\int_{\bbr} a^{-X}\left( Q(v|\tilde{v}) + \frac{1}{2}\tilde{\psi}^2 \right) dy + \mathbf{G}_1(\tau) + \mathbf{G}^R(\tau) + \frac{1}{4}\mathbf{D}(\tau) \\
	&\leq -\frac{\delta_2}{2M}|\dot{X}(\tau)|^2 + \mathbf{B}_1 + \mathbf{B}_4 - \mathbf{G}_2 -\frac{3}{4}\mathbf{D} \\
	&\quad -\frac{\delta_2}{2M}|\dot{X}(\tau)|^2 + \dot{X}(\tau)\sum_{i=3}^{5} Y_i(\tau) + \mathbf{B}_2(\tau) + \mathbf{B}_3(\tau) +  \sum_{i=5}^{11}\mathbf{B}_i(\tau).
\end{align*}
By \eqref{estimate leading} and Young's inequality, we have
\begin{align}
	\begin{aligned} \label{need estimate 0after}
		&\frac{d}{d\tau}\int_{\bbr} a^{-X}\left( Q(v|\tilde{v}) + \frac{1}{2}\tilde{\psi}^2 \right) dy + \frac{\delta_2}{4M}|\dot{X}(\tau)|^2 + \mathbf{G}_1(\tau) + \mathbf{G}^R(\tau) + \mathbf{G}^S(\tau) + \frac{1}{4}\mathbf{D}(\tau) \\
		&\leq C \int_{\bbr} a^{-X}_y|v^r - v^*|\tilde{\psi}^2 dy +  C\frac{1}{\delta_2}\sum_{i=3}^{5} Y_i^2(\tau) + \mathbf{B}_2(\tau) + \mathbf{B}_3(\tau) +  \sum_{i=5}^{11}\mathbf{B}_i(\tau),
	\end{aligned}
\end{align}
where $\mathbf{G}^S(\tau) = C_2 \int_{\bbr} (v^s)^{-X}_y\tilde{\psi}^2 dy$.

Now we estimate the terms on the right-hand side of \eqref{need estimate 0after}.
First of all, by \eqref{weight}, \eqref{priori assump after}, Lemma \ref{rare-shock interact}, H\"{o}lder's inequality, interpolation inequality and Young's inequality, one has
\begin{align*}
	&C \int_{\bbr} a^{-X}_y|v^r - v^*|\tilde{\psi}^2 dy \\
	&\leq C\frac{\lambda}{\delta_2} \int_{\bbr} (v^s)^{-X}_y|v^r - v^*|\tilde{\psi}^2 dy \\
	&\leq C \frac{\lambda}{\delta_2} \|\tilde{\psi}\|_{L^4}^2 \|(v^s)^{-X}_y(v^r - v^*)\| \\
	&\leq C \frac{\lambda}{\delta_2} \|\tilde{\psi}\|^{\frac{3}{2}} \|\tilde{\psi}_y\|^{\frac{1}{2}} \|(v^s)^{-X}_y(v^r - v^*)\| \\
	&\leq \frac{1}{40}\mathbf{D} + C \left( \frac{\lambda}{\delta_2} \right)^{\frac{4}{3}} \varepsilon^{-\frac{1}{3}} \|\tilde{\psi}\|^2 \|(v^s)^{-X}_y(v^r - v^*)\|^{\frac{4}{3}} \\
	&\leq \frac{1}{40}\mathbf{D} + C \left( \delta_1^\frac{4}{3} \delta_2^\frac{2}{3} \lambda^\frac{4}{3} \varepsilon^{-1} e^{-\frac{C\delta_2}{\varepsilon}(\tau + A\frac{\varepsilon}{\delta^2})} +  \delta_1^\frac{4}{3} \delta_2^\frac{4}{3} \lambda^\frac{4}{3} \varepsilon^{-\frac{5}{3}} \kappa^{\frac{2}{3}} e^{-\frac{C}{\kappa}(\tau + A\frac{\varepsilon}{\delta^2})} \right) \|\tilde{\psi}\|^2.
\end{align*}
%
\textbf{Step 1: Estimates on the terms $Y_i$}.
From H\"{o}lder's inequality and Lemma \ref{lemma-shock}, it holds that
\begin{align*}
	\frac{C}{\delta_2}|Y_3(\tau)|^2 & =\frac{C}{\delta_2}  \left|\int_{\bbr} \frac{a^{-X}}{s_2}(v^s)_y^{-X}(p'(\tilde{v}) - p'((v^s)^{-X}))\tilde{\psi} dy\right|^2 \\
	&\leq C\frac{\delta_1^2}{\delta_2}\|(v^s)_y^{-X}\|_{L^1} \mathbf{G}^S \leq C\delta_1^2 \mathbf{G}^S \leq \frac{1}{40}\mathbf{G}^S.
\end{align*}
By \eqref{weight}, H\"{o}lder's inequality and Lemma \ref{lemma-shock}, it holds that
\begin{align*}
	\frac{C}{\delta_2}|Y_4(\tau)|^2& \leq C \frac{C}{\delta_2} \left| \int_{\bbr} (v^s)_y^{-X}|\tilde{\psi} + s_2\tilde{\phi}| dy \right|^2 \\
	&\leq C\frac{\delta_2}{\lambda^2} \left|\int_{\bbr} a^{-X}_y|\tilde{\psi} + s_2\tilde{\phi}| dy \right|^2 \\
	&\leq C\frac{\delta_2}{\lambda^2} \|a^{-X}_y\|_{L^1} \mathbf{G}_1 \leq C \frac{\delta_2}{\lambda} \mathbf{G}_1 \leq \frac{1}{40}\mathbf{G}_1.
\end{align*}
By \eqref{weight}, Taylor's formula, H\"{o}lder's inequality, Lemma \ref{lemma-shock} and {\it a priori} assumption \eqref{priori assump after}, we have
\begin{align*}
	\frac{C}{\delta_2}|Y_5(\tau)|^2& = \frac{C}{\delta_2} \left|\int_{\bbr}a^{-X}_y\left(Q(v|\tilde{v}) + \frac{1}{2}\tilde{\psi}^2\right)dy\right|^2 \\
	&\leq \frac{C}{\delta_2} \left|\int_{\bbr}a^{-X}_y(\tilde{\phi}^2 + \tilde{\psi}^2)dy \right|^2 \\
	&\leq \frac{C}{\delta_2} \left| \int_{\bbr}a^{-X}_y((\tilde{\psi} + s_2\tilde{\phi})^2 + \tilde{\psi}^2)dy \right|^2 \\
	&\leq \frac{C}{\delta_2} \|a^{-X}_y\|_{L^\infty} (\|\tilde{\psi}\|^2 + \|\tilde{\phi}\|^2) \mathbf{G}_1 + C\frac{\lambda^2}{\delta_2^3} \|(v^s)_y^{-X}\|_{L^\infty} \|\tilde{\psi}\|^2 \mathbf{G}^S \\
	&\leq C \frac{\lambda}{\varepsilon} (\|\tilde{\psi}\|^2 + \|\tilde{\phi}\|^2) \mathbf{G}_1 + C\frac{\lambda^2}{\delta_2\varepsilon} \|\tilde{\psi}\|^2 \mathbf{G}^S \leq \frac{1}{40}\mathbf{G}_1 + \frac{1}{40}\mathbf{G}^S.
\end{align*}
%
\textbf{Step 2: Estimates on the terms $\mathbf{B}_i \ (i=2,3,5,\cdots11)$}.
By \eqref{weight}, Taylor's formula, Cauchy's inequality, H\"{o}lder's inequality, interpolation inequality, Lemma \ref{lemma-shock} and {\it a priori} assumption \eqref{priori assump after}, we have
\begin{align*}
	\mathbf{B}_2(\tau) + \mathbf{B}_3(\tau) &= \int_{\bbr}a^{-X}_y p(v|\tilde{v})\tilde{\psi} dy - s_2 \int_{\bbr}a^{-X}_y \left( Q(v|\tilde{v}) + \frac{p'(\tilde{v})}{2}\tilde{\phi}^2 \right) dy \\
	&\leq C \int_{\bbr}a^{-X}_y(|\tilde{\phi}|^3 + \tilde{\phi}^2|\tilde{\psi}|)dy \\
	&\leq C \int_{\bbr}a^{-X}_y(|\tilde{\phi}|(\tilde{\psi} + s_2\tilde{\phi})^2 + |\tilde{\phi}|\tilde{\psi}^2)dy \\
	&\leq C \int_{\bbr}a^{-X}_y(|\tilde{\phi}|(\tilde{\psi} + s_2\tilde{\phi})^2 + |\tilde{\psi}|^3 + |\tilde{\psi}+s_2\tilde{\phi}|\tilde{\psi}^2)dy \\
	&\leq C \|\tilde{\phi}\|_{L^\infty}\mathbf{G}_1 + C\frac{\lambda}{\delta_2}\|\tilde{\psi}\|_{L^\infty}^2 \int_{\bbr} (v^s)^{-X}_y |\tilde{\psi}| dy + C\|\tilde{\psi}\|_{L^\infty}^2\int_{\bbr} a^{-X}_y |\tilde{\psi}+s_2\tilde{\phi}|dy \\
	&\leq C \|\tilde{\phi}\|_{L^\infty} \mathbf{G}_1 + C\frac{\lambda}{\delta_2}\|\tilde{\psi}\|\|\tilde{\psi}_y\| \|(v^s)_y^{-X}\|_{L^1}^{\frac{1}{2}} \sqrt{\mathbf{G}^S} + C \|\tilde{\psi}\|\|\tilde{\psi}_y\|\|a^{-X}_y\|_{L^1}^{\frac{1}{2}} \sqrt{\mathbf{G}_1} \\
	&\leq C \|\tilde{\phi}\|_{L^\infty} \mathbf{G}_1 + C\frac{\lambda}{\sqrt{\delta_2}}\|\tilde{\psi}\| \varepsilon^{-\frac{1}{2}} \sqrt{\mathbf{D}} \sqrt{\mathbf{G}^S} + C \sqrt{\lambda}\|\tilde{\psi}\| \varepsilon^{-\frac{1}{2}} \sqrt{\mathbf{D}}\sqrt{\mathbf{G}_1} \\
	&\leq C \|\tilde{\phi}\|_{L^\infty} \mathbf{G}_1 +  C \frac{\lambda\varepsilon^{-\frac{1}{2}}}{\sqrt{\delta_2}}\|\tilde{\psi}\|  (\mathbf{D} + \mathbf{G}^S) + C \sqrt{\lambda\varepsilon^{-1}} \|\tilde{\psi}\| (\mathbf{D} + \mathbf{G}_1) \\
	&\leq \frac{1}{40} \mathbf{G}_1 + \frac{1}{40} \mathbf{D} + \frac{1}{40} \mathbf{G}^S.
\end{align*}
From Young's inequality and \eqref{weight}, it follows that
\begin{align*}	
	\mathbf{B}_5(\tau)&\leq C \int_{\bbr} (v^s)^{-X}_y |\tilde{\psi}(\tilde{\psi} + s_2\tilde{\phi})| dy \\
	&\leq \frac{1}{40} \mathbf{G}^S + C  \int_{\bbr} (v^s)^{-X}_y (\tilde{\psi} + s_2\tilde{\phi})^2 dy\\
	&\leq \frac{1}{40} \mathbf{G}^S + C\frac{\delta_2}{\lambda} \int_{\bbr} a^{-X}_y (\tilde{\psi}+s_2\tilde{\phi})^2 dy \\
	&\leq \frac{1}{40} \mathbf{G}^S + \frac{1}{40} \mathbf{G}_1,
\end{align*}
and
\begin{align*}
	\mathbf{B}_6(\tau) \leq C  \int_{\bbr} (v^s)^{-X}_y (\tilde{\psi} + s_2\tilde{\phi})^2 dy \leq C\frac{\delta_2}{\lambda} \int_{\bbr} a^{-X}_y (\tilde{\psi}+s_2\tilde{\phi})^2 dy \leq \frac{1}{40} \mathbf{G}_1.
\end{align*}
By \eqref{weight}, Cauchy's inequality, H\"{o}lder's inequality and {\it a priori} assumption \eqref{priori assump after}, one has
\begin{align*}
	\mathbf{B}_7(\tau) &\leq C \int_{\bbr} (v^s)^{-X}_y |\tilde{\phi}|^3 dy \\
	&\leq C \int_{\bbr} (v^s)^{-X}_y |\tilde{\phi}|(|\tilde{\psi}+s_2\tilde{\phi}|^2 + |\tilde{\psi}|^2) dy \\
	&\leq C\frac{\delta_2}{\lambda} \|\tilde{\phi}\|_{L^\infty} \mathbf{G}_1 + C \|\tilde{\phi}\|_{L^\infty} \mathbf{G}^S \\
	&\leq \frac{1}{40} \mathbf{G}_1 + \frac{1}{40} \mathbf{G}^S.
\end{align*}
From Young's inequality, \eqref{weight} and Lemma \ref{lemma-shock}, it follows that
\begin{align*}
	\mathbf{B}_8(\tau) &\leq C \varepsilon \int_{\bbr} a^{-X}_y |\tilde{\psi}\tilde{\psi}_y| dy \\
	&\leq \frac{1}{40} \mathbf{D} + C \varepsilon \frac{\lambda^2}{\delta_2^2} \int_{\bbr} ((v^s)^{-X}_y)^2 \tilde{\psi}^2 dy \\
	&\leq \frac{1}{40} \mathbf{D} + C \lambda^2 \int_{\bbr} (v^s)^{-X}_y \tilde{\psi}^2 dy \\
	&\leq \frac{1}{40} \mathbf{D} + \frac{1}{40} \mathbf{G}^S.
\end{align*}
By \eqref{weight}, Cauchy's inequality, H\"{o}lder's inequality, Lemma \ref{lemma-rare} and Lemma \ref{lemma-shock}, one has
\begin{align*}
	\mathbf{B}_{9}(\tau) &\leq C \varepsilon\int_{\bbr} a^{-X}_y (v^s)^{-X}_y |\tilde{\psi}\tilde{\phi}|dy + C \varepsilon\int_{\bbr} a^{-X}_y u^r_y |\tilde{\psi}\tilde{\phi}|dy \\
	&\leq C \varepsilon\int_{\bbr} a^{-X}_y (v^s)^{-X}_y \tilde{\psi}^2 dy + C \varepsilon\int_{\bbr} a^{-X}_y (v^s)^{-X}_y (\tilde{\psi} + s_2\tilde{\phi})^2 dy \\
	&\quad + C \varepsilon\int_{\bbr} (a^{-X}_y)^2 \tilde{\psi}^2 dy + C \varepsilon\int_{\bbr} (u^r_y)^2 \tilde{\phi}^2 dy \\
	&\leq C \varepsilon\|a^{-X}_y\|_{L^\infty}\mathbf{G}^S + C \varepsilon\|(v^s)_y^{-X}\|_{L^\infty}\mathbf{G}_1 + C \varepsilon\frac{\lambda^2}{\delta_2^2}\|(v^s)_y^{-X}\|_{L^\infty}\mathbf{G}^S + C \varepsilon\|u^r_y\|_{L^\infty}\mathbf{G}^R \\
	&\leq C \lambda\delta_2\mathbf{G}^S + C \delta_2^2 \mathbf{G}_1 + C\lambda^2 \mathbf{G}^S + C\varepsilon\kappa^{-1}\delta_1\mathbf{G}^R \\
	&\leq \frac{1}{40} \mathbf{G}^S + \frac{1}{40} \mathbf{G}_1 + \frac{1}{40} \mathbf{G}^R,
\end{align*}
and
\begin{align*}
	\mathbf{B}_{10}(\tau) &\leq C \varepsilon\int_{\bbr} (v^s)^{-X}_y |\tilde{\psi}_y\tilde{\phi}|dy + C \varepsilon\int_{\bbr} u^r_y |\tilde{\psi}_y\tilde{\phi}|dy \\
	&\leq \frac{1}{40} \mathbf{D} + C \varepsilon\int_{\bbr} ((v^s)^{-X}_y)^2 (\tilde{\psi} + s_2\tilde{\phi})^2 dy + C \varepsilon\int_{\bbr} ((v^s)^{-X}_y)^2 \tilde{\psi}^2 dy + C \varepsilon\int_{\bbr} (u^r_y)^2 \tilde{\phi}^2 dy \\
	&\leq \frac{1}{40} \mathbf{D} + C\varepsilon \frac{\delta_2}{\lambda} \|(v^s)_y^{-X}\|_{L^\infty}\mathbf{G}_1 + C \varepsilon \|(v^s)_y^{-X}\|_{L^\infty}\mathbf{G}^S + C \varepsilon \|u^r_y\|_{L^\infty}\mathbf{G}^R \\
	&\leq \frac{1}{40} \mathbf{D} + C \frac{\delta_2^3}{\lambda} \mathbf{G}_1 + C \delta_2^2\mathbf{G}^S + C \varepsilon\kappa^{-1}\delta_1 \mathbf{G}^R \\
	&\leq \frac{1}{40} \mathbf{D} + \frac{1}{40} \mathbf{G}_1 + \frac{1}{40} \mathbf{G}^S+ \frac{1}{40} \mathbf{G}^R.
\end{align*}
By H\"{o}lder's inequality, interpolation inequality and Young's inequality, it holds that
\begin{align*}
	\mathbf{B}_{11}(\tau) &\leq C \int_{\bbr} |(F_{3y} + F_{4y})\tilde{\psi}|dy \\
	&\leq C \|\tilde{\psi}\|_{L^\infty}(\|F_{3y}\|_{L^1} + \|F_{4y}\|_{L^1}) \\
	&\leq C \|\tilde{\psi}\|^{\frac{1}{2}}\|\tilde{\psi}_y\|^{\frac{1}{2}}(\|F_{3y}\|_{L^1} + \|F_{4y}\|_{L^1}) \\
	&\leq \frac{1}{40} \mathbf{D} + C \varepsilon^{-\frac{1}{3}}\|\tilde{\psi}\|^{\frac{2}{3}} \left( \|F_{3y}\|_{L^1}^{\frac{4}{3}} + \|F_{4y}\|_{L^1}^{\frac{4}{3}} \right).
\end{align*}
Note that by Lemma \ref{lemma-rare} and Lemma \ref{rare-shock interact}, we have
\begin{align*}
	&\|F_{3y}\|_{L^1} \leq C \varepsilon \left( \|(v^s)^{-X}_y v^r_y\|_{L^1} + \|v^r_y\|^2 + \|u^r_{yy}\|_{L^1}\right) \\
	&\leq C \delta_1 \delta_2 \varepsilon\kappa^{-1} e^{-\frac{C\delta_2}{\varepsilon}\left(\tau + A\frac{\varepsilon}{\delta^2}\right)} + C \delta_1 \delta_2^2 e^{-\frac{C}{\kappa}\left(\tau + A\frac{\varepsilon}{\delta^2}\right)} + C\varepsilon \left(\|v^r_y\|^2 + \|u^r_{yy}\|_{L^1}\right),
\end{align*}
\begin{align*}
	&\|F_{4y}\|_{L^1} \leq C \left( \|(v^s)^{-X}_y(v^r - v^*)\|_{L^1} + \|v^r_y((v^s)^{-X} - v^*)\|_{L^1} \right) \\
	&\leq C (\delta_1 \varepsilon \kappa^{-1} + \delta_1 \delta_2) e^{-\frac{C\delta_2}{\varepsilon}\left(\tau + A\frac{\varepsilon}{\delta^2}\right)} + C (\delta_1 \delta_2 + \delta_1 \delta_2^2 \varepsilon^{-1} \kappa) e^{-\frac{C}{\kappa}\left(\tau + A\frac{\varepsilon}{\delta^2}\right)},
\end{align*}
and
\begin{align*}
	&\int_{-A\frac{\varepsilon}{\delta^2}}^{+\infty} \left(\|v^r_y\|^2 + \|u^r_{yy}\|_{L^1}\right)^\frac{4}{3} d\tau \\
	&\leq C \int_{-A\frac{\varepsilon}{\delta^2}}^{\delta_1^{-1}\kappa - A\frac{\varepsilon}{\delta^2}} (\delta_1\kappa^{-1})^\frac{4}{3} d\tau + C \int_{\delta_1^{-1}\kappa -A\frac{\varepsilon}{\delta^2}}^{+\infty} \left(\tau + A\frac{\varepsilon}{\delta^2}\right)^{-\frac{4}{3}} d\tau \leq C\delta_1^\frac{1}{3} \kappa^{-\frac{1}{3}}.
\end{align*}
%
%
%
\textbf{Step 3: Conclusion}.
Combining the above estimates and integrating over $\left[- A\frac{\varepsilon}{\delta^2}, \tau\right]$, by Young's inequality, we have
\begin{align*}
	&\int_{\bbr} \bigg(Q(v|\tilde{v}) + \frac{1}{2}\tilde{\psi}^2\bigg)(\tau)dy + \delta_2 \int_{- A\frac{\varepsilon}{\delta^2}}^{\tau}|\dot{X}(\tau)|^2 d\tau + \int_{- A\frac{\varepsilon}{\delta^2}}^{\tau}\left(\mathbf{G}_1(\tau) + \mathbf{G}^R(\tau) + \mathbf{G}^S(\tau) + \mathbf{D}(\tau)\right)d\tau \\
	&\leq C\int_{\bbr} \bigg(Q(v|\tilde{v}) + \frac{1}{2}\tilde{\psi}^2\bigg)\big(- A\frac{\varepsilon}{\delta^2}\big)dy + C \varepsilon^{-\frac{1}{3}} \sup_{\tau\in [- A\frac{\varepsilon}{\delta^2}, \tau_1]}\|\tilde{\psi}\|^{\frac{2}{3}} \int_{- A\frac{\varepsilon}{\delta^2}}^{\tau}\left( \|F_{3y}\|_{L^1}^{\frac{4}{3}} + \|F_{4y}\|_{L^1}^{\frac{4}{3}} \right) d\tau \\
	&\quad + C \int_{- A\frac{\varepsilon}{\delta^2}}^{\tau} \left( \delta_1^\frac{4}{3} \delta_2^\frac{2}{3} \lambda^\frac{4}{3} \varepsilon^{-1} e^{-\frac{C\delta_2}{\varepsilon}\left(\tau + A\frac{\varepsilon}{\delta^2}\right)} +  \delta_1^\frac{4}{3} \delta_2^\frac{4}{3} \lambda^\frac{4}{3} \varepsilon^{-\frac{5}{3}} \kappa^{\frac{2}{3}} e^{-\frac{C}{\kappa}\left(\tau + A\frac{\varepsilon}{\delta^2}\right)} \right)d\tau \sup_{\tau\in [- A\frac{\varepsilon}{\delta^2}, \tau_1]}\|\tilde{\psi}\|^2 \\
	&\leq C\int_{\bbr} \bigg(Q(v|\tilde{v}) + \frac{1}{2}\tilde{\psi}^2\bigg)\big(- A\frac{\varepsilon}{\delta^2}\big)dy + \frac{1}{4}\sup_{[- A\frac{\varepsilon}{\delta^2}, \tau_1]}\|\tilde{\psi}\|^2 \\
	&\quad + C \varepsilon^{-\frac{1}{2}}  \bigg( \int_{- A\frac{\varepsilon}{\delta^2}}^{\tau}\big( \|F_{3y}\|_{L^1}^{\frac{4}{3}} + \|F_{4y}\|_{L^1}^{\frac{4}{3}} \big) d\tau \bigg)^{\frac{3}{2}} \\
	&\leq C\int_{\bbr} \bigg(Q(v|\tilde{v}) + \frac{1}{2}\tilde{\psi}^2\bigg)\big(- A\frac{\varepsilon}{\delta^2}\big)dy + \frac{1}{4}\sup_{[- A\frac{\varepsilon}{\delta^2}, \tau_1]}\|\tilde{\psi}\|^2 + C \delta_1^2\delta_2^{-\frac{3}{2}}\varepsilon^3\kappa^{-2} + C \delta_1^2\delta_2^2\varepsilon^{-\frac{1}{2}}\kappa^{\frac{3}{2}}  \\
	&\quad + C \delta_1^2\delta_2^{\frac{1}{2}}\varepsilon + C \delta_1^2\delta_2^4\varepsilon^{-\frac{5}{2}}\kappa^{\frac{7}{2}} + C \delta_1^{\frac{1}{2}}\varepsilon^\frac{3}{2}\kappa^{-\frac{1}{2}},	
\end{align*}
and this implies the desired estimate \eqref{estimate 0after} in Lemma \ref{lemma 0after} together with
$$Q(v|\tilde{v}) \sim \tilde{\phi}^2,\quad \mathbf{G}_1 \sim G_1,\quad \mathbf{G}^R \sim G^R, \quad \mathbf{G}^S \sim G^S,\quad \mathbf{D} \sim D.$$

\subsection{First-order derivative estimates}

In this subsection, we derive the first-order derivative estimates.
\begin{lemma} \label{lemma phi1 after}
	Under the assumption of Proposition \ref{priori estimate after}, it holds that for $- A\frac{\varepsilon}{\delta^2} \leq \tau \leq \tau_1$,
	\begin{equation} \label{estimate phi1 after}
		\| \tilde{\phi}_y(\tau) \|^2 + \varepsilon^{-1}\int_{- A\frac{\varepsilon}{\delta^2}}^{\tau} D_1 d\tau  \leq C (\varepsilon^{-2}\| (\tilde{\phi}_0, \tilde{\psi}_0) \|^2 + \| \tilde{\phi}_{0y} \|^2) + C\delta_1^\frac{1}{2}\varepsilon^{-1},
	\end{equation}
	where $D_1 = \|\tilde{\phi}_y(\tau)\|^2$.
\end{lemma}
\begin{proof}
	Multiplying the equation $\eqref{per after}_2$ by $\frac{\tilde{\phi}_y}{v}$ and integrating the resulting equation over $\bbr$ lead to
	\begin{align*}
		&\frac{\varepsilon}{2} \frac{d}{d\tau} \int_\bbr \left(\frac{\tilde{\phi}_y}{v}\right)^2 dy \\
		&= \int_\bbr \frac{\tilde{\phi}_y}{v}(p(v) -p(\tilde{v}))_y dy + \int_\bbr \frac{\tilde{\phi}_y}{v}\tilde{\psi}_\tau dy - \dot{X}(\tau)\int_\bbr (u^s)_y^{-X}\frac{\tilde{\phi}_y}{v} dy - \varepsilon\int_\bbr \frac{\tilde{\phi}_y}{v}\left(\frac{\tilde{v}_y}{v}\right)_\tau dy \\
		&\quad + \varepsilon\int_\bbr \left(\frac{(u^s)^{-X}_y}{(v^s)^{-X}}\right)_y\frac{\tilde{\phi}_y}{v} dy + \int_\bbr F_{4y}\frac{\tilde{\phi}_y}{v} dy \\
		& := \sum_{i=1}^{6}I_i,
	\end{align*}
	where we have used the fact that
	$$\varepsilon\left(\frac{u_y}{v}\right)_y = \varepsilon\left(\frac{v_\tau}{v}\right)_y = \varepsilon(\ln v)_{\tau y} = \varepsilon\left(\frac{v_y}{v}\right)_\tau = \varepsilon\left(\frac{\tilde{\phi}_y}{v}\right)_\tau + \varepsilon\left(\frac{\tilde{v}_y}{v}\right)_\tau.$$
	
	We now estimate the terms $I_i (i=1, \cdots, 6)$ respectively.   
	First, by Young's inequality, H\"{o}lder's inequality, Lemma \ref{lemma-rare}, Lemma \ref{lemma-shock} and \eqref{weight}, it holds that
	\begin{align*}
		I_1 &= \int_\bbr \frac{\tilde{\phi}_y}{v}(p(v) -p(\tilde{v}))_y dy \\
		&= \int_\bbr \frac{p'(v)}{v} (\tilde{\phi}_y)^2 dy + \int_\bbr \frac{\tilde{\phi}_y}{v}(p'(v) -p'(\tilde{v}))\tilde{v}_y dy \\
		&\leq -C_2 \|\tilde{\phi}_y\|^2 + C\int_\bbr (v^s)^{-X}_y|\tilde{\phi}\tilde{\phi}_y| dy + C\int_\bbr u^r_y|\tilde{\phi}\tilde{\phi}_y| dy \\
		&\leq -C_2 D_1 + \frac{1}{8}C_2 D_1 + C\int_\bbr ((v^s)^{-X}_y)^2(\tilde{\psi} + s_2\tilde{\phi})^2 dy + C\int_\bbr ((v^s)^{-X}_y)^2\tilde{\psi}^2 dy + C\int_\bbr (u^r_y)^2\tilde{\phi}^2 dy \\
		&\leq -\frac{7}{8}C_2 D_1 + C\|(v^s)^{-X}_y\|_{L^\infty}\frac{\delta_2}{\lambda}\int_\bbr a^{-X}_y(\tilde{\psi} + s_2\tilde{\phi})^2 dy + C\|(v^s)^{-X}_y\|_{L^\infty}\int_\bbr (v^s)^{-X}_y\tilde{\psi}^2 dy \\
		&\quad + C\|u^r_y\|_{L^\infty}\int_\bbr u^r_y\tilde{\phi}^2 dy \\
		&\leq -\frac{7}{8}C_2 D_1 + C\frac{\delta_2^3}{\lambda}\varepsilon^{-1} G_1 + C\delta_2^2\varepsilon^{-1} G^S + C\delta_1 \kappa^{-1} G^R.
	\end{align*}
	By Young's inequality, H\"{o}lder's inequality, interpolation inequality, Lemma \ref{lemma-rare} and Lemma \ref{lemma-shock}, one has
	\begin{align*}
		I_2 &= \int_\bbr \frac{\tilde{\phi}_y}{v}\tilde{\psi}_\tau dy \\
		&= \frac{d}{d\tau}\int_\bbr \frac{\tilde{\phi}_y}{v}\tilde{\psi} dy - \int_\bbr \frac{\tilde{\phi}_{y\tau}}{v} \tilde{\psi} dy + \int_\bbr \frac{\tilde{\phi}_yv_\tau}{v^2} \tilde{\psi} dy \\
		&= \frac{d}{d\tau}\int_\bbr \frac{\tilde{\phi}_y}{v}\tilde{\psi} dy - \int_\bbr \frac{\tilde{\psi}_{yy}}{v} \tilde{\psi} dy - \int_\bbr \frac{\dot{X}(\tau)(v^s)^{-X}_{yy}}{v} \tilde{\psi} dy + \int_\bbr \frac{\tilde{\phi}_y}{v^2}\tilde{\psi}_y \tilde{\psi} dy + \int_\bbr \frac{\tilde{\phi}_y}{v^2}\tilde{u}_y \tilde{\psi} dy \\
		&= \frac{d}{d\tau}\int_\bbr \frac{\tilde{\phi}_y}{v}\tilde{\psi} dy + \int_\bbr \frac{\tilde{\psi}_{y}^2}{v} dy - \int_\bbr \frac{\dot{X}(\tau)(v^s)^{-X}_{yy}}{v} \tilde{\psi} dy - \int_\bbr \frac{\tilde{v}_y}{v^2}\tilde{\psi}_y \tilde{\psi} dy + \int_\bbr \frac{\tilde{\phi}_y}{v^2}\tilde{u}_y \tilde{\psi} dy \\
		&\leq \frac{d}{d\tau}\int_\bbr \frac{\tilde{\phi}_y}{v}\tilde{\psi} dy + C\varepsilon^{-1}D + \frac{1}{8}C_2 D_1 + C\frac{\delta_2}{\varepsilon}|\dot{X}(\tau)|\|(v^s)^{-X}_{y}\|_{L^1}^{\frac{1}{2}} \sqrt{G^S} + C \|(v^s)^{-X}_{y}\|_{L^\infty} G^S \\
		&\quad + C\|u^r_{y}\|_{L^4}^2 \|\tilde{\psi}\|_{L^4}^2 \\
		&\leq \frac{d}{d\tau}\int_\bbr \frac{\tilde{\phi}_y}{v}\tilde{\psi} dy + C\varepsilon^{-1}D + \frac{1}{8}C_2 D_1 + C \frac{\delta_2^2}{\varepsilon}|\dot{X}(\tau)|^2 + C \frac{\delta_2}{\varepsilon} G^S + C \|u^r_{y}\|_{L^4}^2 \|\tilde{\psi}\|^\frac{3}{2}\|\tilde{\psi}_y\|^\frac{1}{2} \\
		&\leq \frac{d}{d\tau}\int_\bbr \frac{\tilde{\phi}_y}{v}\tilde{\psi} dy + \frac{1}{8}C_2 D_1 + C\varepsilon^{-1}D + C \frac{\delta_2^2}{\varepsilon}|\dot{X}(\tau)|^2 + C \frac{\delta_2}{\varepsilon} G^S + C \|u^r_{y}\|_{L^4}^\frac{8}{3} \|\tilde{\psi}\|^2.
	\end{align*}
	It follows from Young's inequality and Lemma \ref{lemma-shock} that
	$$
	I_3 = -\dot{X}(\tau) \int_{\bbr} (u^s)^{-X}_y\frac{\tilde{\phi}_y}{v} dy
	\leq \frac{1}{8} C_2 D_1 + C|\dot{X}(\tau)|^2 \int_{\bbr} ((u^s)^{-X}_y)^2 dy
	\leq \frac{1}{8} C_2 D_1 + C\frac{\delta_2^3}{\varepsilon}|\dot{X}(\tau)|^2.
	$$
	By Young's inequality, Cauchy's inequality, Lemma \ref{lemma-rare}, Lemma \ref{lemma-shock} and Lemma \ref{rare-shock interact}, one has
	\begin{align*}
		I_4 + I_5 &= -\varepsilon \int_\bbr \frac{\tilde{\phi}_y}{v}\left(\frac{\tilde{v}_y}{v} \right)_\tau dy + \varepsilon\int_\bbr \frac{\tilde{\phi}_y}{v}\left(\frac{(u^s)^{-X}_y}{(v^s)^{-X}} \right)_y dy \\
		&= - \varepsilon\int_\bbr \frac{\tilde{\phi}_y}{v}\frac{\tilde{v}_{y\tau}}{v} dy + \varepsilon\int_\bbr \frac{\tilde{\phi}_y}{v}\frac{\tilde{v}_yv_\tau}{v^2} dy + \varepsilon\int_\bbr \frac{\tilde{\phi}_y}{v}\frac{(u^s)^{-X}_{yy}}{(v^s)^{-X}} dy - \varepsilon\int_\bbr \frac{\tilde{\phi}_y}{v}\frac{(u^s)^{-X}_y(v^s)^{-X}_y}{((v^s)^{-X})^2} dy \\
		&= \varepsilon\int_\bbr \frac{\tilde{\phi}_y}{v^2}\dot{X}(\tau)(v^s)^{-X}_{yy} dy  - \varepsilon \int_\bbr \frac{\tilde{\phi}_y}{v^2}u^r_{yy} dy + \varepsilon\int_\bbr \frac{\tilde{\phi}_y}{v}(u^s)^{-X}_{yy}\left(\frac{1}{(v^s)^{-X}} - \frac{1}{v}\right) dy \\
		&\quad + \varepsilon\int_\bbr \frac{\tilde{v}_y}{v^3}\tilde{\phi}_y\tilde{\psi}_y dy + \varepsilon\int_\bbr \frac{\tilde{\phi}_y}{v^3}(v^r_yu^r_y + v^r_y(u^s)^{-X}_y + (v^s)^{-X}_yu^r_y)dy \\
		&\quad + \varepsilon\int_\bbr \frac{\tilde{\phi}_y}{v}(u^s)^{-X}_y(v^s)^{-X}_y\left(\frac{1}{v^2} - \frac{1}{((v^s)^{-X})^2}\right) dy \\
		&\leq \frac{1}{8}C_2 D_1 + C\frac{\delta_2^5}{\varepsilon}|\dot{X}(\tau)|^2 + C\varepsilon^{-1}(\delta_2^4 + \delta_1^2\varepsilon^2\kappa^{-2})D + C \varepsilon^2 (\|u^r_{yy}\|^2 + \|u^r_y\|_{L^4}^4)  \\
		&\quad + C \varepsilon^2 \|(v^s)^{-X}_yu^r_y\|^2 + C \delta_2^2 \|(v^s)^{-X}_y(v^r - v^*)\|^2 + C\delta_2^2 \int_\bbr |(v^s)^{-X}_y|^2 \tilde{\phi}^2 dy \\
		&\leq \frac{1}{8}C_2 D_1 + C\frac{\delta_2^5}{\varepsilon}|\dot{X}(\tau)|^2 + C\varepsilon^{-1}(\delta_2^4 + \delta_1^2\varepsilon^2\kappa^{-2})D + C \varepsilon^2 (\|u^r_{yy}\|^2 + \|u^r_y\|_{L^4}^4)  \\
		&\quad + C \varepsilon^2 \|(v^s)^{-X}_yu^r_y\|^2 + C \delta_2^2 \|(v^s)^{-X}_y(v^r - v^*)\|^2  + C\frac{\delta_2^4}{\varepsilon} \int_\bbr (v^s)^{-X}_y((\tilde{\psi} + s_2\tilde{\phi})^2 + \tilde{\psi}^2) dy \\
		&\leq \frac{1}{8}C_2 D_1 + C\frac{\delta_2^5}{\varepsilon}|\dot{X}(\tau)|^2 + C\varepsilon^{-1}(\delta_2^4 + \delta_1^2\varepsilon^2\kappa^{-2})D + C \frac{\delta_2^4}{\varepsilon}G^S + \frac{\delta_2^5}{\lambda\varepsilon} G_1 \\
		&\quad + C (\|u^r_{yy}\|^2 + \|u^r_y\|_{L^4}^2) + C \delta_1^2 \delta_2^3 \kappa^{-2} \varepsilon e^{-\frac{C\delta_2}{\varepsilon}\left(\tau + A\frac{\varepsilon}{\delta^2}\right)} + C \delta_1^2 \delta_2^4 \kappa^{-1} e^{-\frac{C}{\kappa}\left(\tau + A\frac{\varepsilon}{\delta^2}\right)} \\
		&\quad + C \delta_1^2 \delta_2^5 \varepsilon^{-1} e^{-\frac{C\delta_2}{\varepsilon}\left(\tau + A\frac{\varepsilon}{\delta^2}\right)} + C \delta_1^2 \delta_2^6 \varepsilon^{-2} \kappa e^{-\frac{C}{\kappa}\left(\tau + A\frac{\varepsilon}{\delta^2}\right)}.
	\end{align*}
	Similarly, it holds that
	\begin{align*}
		I_6 &= \int_\bbr F_{4y} \frac{\tilde{\phi}_y}{v} dy
		\leq \frac{1}{8}C_2 D_1 + C (\|v^r_y((v^s)^{-X}-v^*)\|^2 + \|(v^s)^{-X}_y(v^r - v^*)\|^2) \\
		&\leq \frac{1}{8}C_2 D_1 +  C \delta_1^2 \delta_2^2 \kappa^{-1} e^{-\frac{C\delta_2}{\varepsilon}\left(\tau + A\frac{\varepsilon}{\delta^2}\right)} + C \delta_1^2 \delta_2^2 \kappa^{-1} e^{-\frac{C}{\kappa}\left(\tau + A\frac{\varepsilon}{\delta^2}\right)} + C \delta_1^2 \delta_2^3 \varepsilon^{-1} e^{-\frac{C\delta_2}{\varepsilon}\left(\tau + A\frac{\varepsilon}{\delta^2}\right)} \\
		&\quad + C \delta_1^2 \delta_2^4 \varepsilon^{-2} \kappa e^{-\frac{C}{\kappa}\left(\tau + A\frac{\varepsilon}{\delta^2}\right)}.
	\end{align*}
	
	Combining the above estimates, integrating over $\left[- A\frac{\varepsilon}{\delta^2}, \tau\right]$, and then combining with the $L^2$-estimate \eqref{estimate 0after}, we can obtain the estimate \eqref{estimate phi1 after}.
\end{proof}

Now we give the first-order derivative estimate of the velocity.
\begin{lemma} \label{lemma psi1 after}
	Under the assumption of Proposition \ref{priori estimate after}, it holds that for $- A\frac{\varepsilon}{\delta^2} \leq \tau \leq \tau_1$,
	\begin{equation} \label{estimate psi1 after}
		\| \tilde{\psi}_y(\tau) \|^2 + \int_{- A\frac{\varepsilon}{\delta^2}}^{\tau} D_2 d\tau  \leq C (\varepsilon^{-2}\| (\tilde{\phi}_0, \tilde{\psi}_0) \|^2 + \| (\tilde{\phi}_{0y}, \tilde{\psi}_{0y}) \|^2) + C\delta_1^\frac{1}{2}\varepsilon^{-1},
	\end{equation}
	where $D_2 = \varepsilon\|\tilde{\psi}_{yy}(\tau)\|^2$.
\end{lemma}
\begin{proof}
	Multiplying the equation $\eqref{per after}_2$ by $-\tilde{\psi}_{yy}$ and integrating the resulting equation over $\bbr$ lead to
	\begin{align*}
		&\frac{1}{2} \frac{d}{d\tau} \int_\bbr \tilde{\psi}_y^2 dy + \varepsilon\int_\bbr \frac{\tilde{\psi}_{yy}^2}{v} dy \\
		&= \int_\bbr (p(v) - p(\tilde{v}))_y\tilde{\psi}_{yy} dy - \dot{X}(\tau)\int_\bbr (u^s)_y^{-X}\tilde{\psi}_{yy} dy + \varepsilon\int_\bbr \frac{\tilde{\phi}_y}{v^2}\tilde{\psi}_y\tilde{\psi}_{yy} dy \\
		&\quad + \varepsilon\int_\bbr \frac{\tilde{v}_y}{v^2}\tilde{\psi}_y\tilde{\psi}_{yy} dy - \varepsilon\int_\bbr \left(\tilde{u}_y\left(\frac{1}{v} - \frac{1}{\tilde{v}}\right)\right)_y \tilde{\psi}_{yy} dy + \int_\bbr (F_{3y} + F_{4y})\tilde{\psi}_{yy} dy \\
		& := \sum_{i=1}^{6}J_i,
	\end{align*}
	where we have used the fact that
	$$\varepsilon\left(\frac{u_y}{v} - \frac{\tilde{u}_y}{\tilde{v}}\right)_y (-\tilde{\psi}_{yy}) = -\varepsilon\frac{\tilde{\psi}_{yy}^2}{v} + \varepsilon\frac{\tilde{\phi}_y}{v^2}\tilde{\psi}_y\tilde{\psi}_{yy} + \varepsilon\frac{\tilde{v}_y}{v^2}\tilde{\psi}_y\tilde{\psi}_{yy} - \varepsilon\left(\tilde{u}_y\left(\frac{1}{v} - \frac{1}{\tilde{v}}\right)\right)_y \tilde{\psi}_{yy}.$$
	
	We now estimate the terms $J_i (i=1, \cdots, 6)$ respectively. 	 
	By Young's inequality, Lemma \ref{lemma-rare} and Lemma \ref{lemma-shock}, we can infer that
	\begin{align*}
		J_1 &= \int_\bbr (p(v) -p(\tilde{v}))_y\tilde{\psi}_{yy} dy \leq \frac{1}{20}\mathbf{D}_2 + C \varepsilon^{-1} \|\tilde{\phi}_y\|^2 + C \varepsilon^{-1} \int_\bbr |\tilde{v}_y\tilde{\phi}|^2 dy \\
		&\leq \frac{1}{20}\mathbf{D}_2 + C \varepsilon^{-1} D_1 + C\frac{\delta_2^3}{\lambda\varepsilon^2} G_1 + C\frac{\delta_1}{\varepsilon\kappa} G^R + C\frac{\delta_2^2}{\varepsilon^2} G^S,
	\end{align*}
	and
	\begin{align*}
		J_2 = - \dot{X}(\tau)\int_\bbr (u^s)_y^{-X}\tilde{\psi}_{yy} dy \leq \frac{1}{20}\mathbf{D}_2 + C \varepsilon^{-1} |\dot{X}(\tau)|^2 \int_\bbr ((u^s)^{-X}_y)^2 dy \leq \frac{1}{20}\mathbf{D}_2 + C\frac{\delta_2^3}{\varepsilon^2} |\dot{X}(\tau)|^2.
	\end{align*}
	By Young's inequality, interpolation inequality and {\it a priori} assumption \eqref{priori assump after}, one has
	\begin{align*}
		J_3 &= \varepsilon \int_\bbr \frac{\tilde{\phi}_y}{v^2}\tilde{\psi}_y\tilde{\psi}_{yy} dy \leq C \varepsilon \|\tilde{\psi}_y\|_{L^\infty} \|\tilde{\phi}_y\| \|\tilde{\psi}_{yy}\| \leq C \varepsilon \|\tilde{\psi}_y\|^\frac{1}{2} \|\tilde{\phi}_y\| \|\tilde{\psi}_{yy}\|^\frac{3}{2} \\
		&\leq \frac{1}{20}\mathbf{D}_2 + C \|\tilde{\phi}_y\|^4 D \leq \frac{1}{20}\mathbf{D}_2 + C \delta^\frac{4}{3}\varepsilon^{-2} D.
	\end{align*}
	By Young's inequality, Lemma \ref{lemma-rare} and Lemma \ref{lemma-shock},
	\begin{align*}
		J_4 = \varepsilon \int_\bbr \frac{\tilde{v}_y}{v^2}\tilde{\psi}_y\tilde{\psi}_{yy} dy \leq \frac{1}{20}\mathbf{D}_2 + C \varepsilon \|\tilde{v}_y\|_{L^\infty}^2 \|\tilde{\psi}_y\|^2 \leq \frac{1}{20}\mathbf{D}_2 + C\left(\frac{\delta_2^4}{\varepsilon^2} + \frac{\delta_1^2}{\kappa^2}\right) D,
	\end{align*}
	and
	\begin{align*}
		J_5 &= - \varepsilon \int_\bbr \left(\tilde{u}_y\left(\frac{1}{v} - \frac{1}{\tilde{v}}\right)\right)_y \tilde{\psi}_{yy} dy \\
		&\leq \frac{1}{20}\mathbf{D}_2 + C \varepsilon \int_\bbr \tilde{u}_{yy}^2\tilde{\phi}^2 dy + C \varepsilon \int_\bbr \tilde{u}_y^2\tilde{\phi}_y^2 dy + C \varepsilon \int_\bbr |\tilde{u}_y\tilde{v}_y\tilde{\phi}|^2 dy \\
		&\leq \frac{1}{20}\mathbf{D}_2 + C \left(\frac{\delta_2^4}{\varepsilon} + \frac{\delta_1^2\varepsilon}{\kappa^2}\right)D_1 + C\frac{\delta_2^4}{\varepsilon^2} G^S + C\frac{\delta_2^5}{\lambda\varepsilon^2} G_1 + C\frac{\delta_1\varepsilon}{\kappa^3} G^S.
	\end{align*}
	By Young's inequality, Lemma \ref{lemma-rare}, Lemma \ref{lemma-shock} and Lemma \ref{rare-shock interact}, we obtain
	\begin{align*}
		J_6 &= \int_\bbr (F_{3y} + F_{4y})\tilde{\psi}_{yy} dy 
		\leq \frac{1}{20}\mathbf{D}_2 + C \varepsilon^{-1} \int_\bbr (F_{3y}^2 + F_{4y}^2) dy \leq \frac{1}{20}\mathbf{D}_2 + C \varepsilon \|u^r_{yy}\|^2 + C \varepsilon \|u^r_y\|_{L^4}^4 \\
		&\quad + C \varepsilon \|u^r_y(v^s)^{-X}_y\|^2 + C\varepsilon^{-1} \|v^r_y((v^s)^{-X} - v^*)\|^2 + C\varepsilon^{-1}\|(v^s)^{-X}_y(v^r - v^*)\|^2 \\
		&\leq \frac{1}{20}\mathbf{D}_2 + C \varepsilon \|u^r_{yy}\|^2 + C \varepsilon \|u^r_y\|_{L^4}^4 + C \delta_1^2 \delta_2^3 \kappa^{-2}  e^{-\frac{C\delta_2}{\varepsilon}\left(\tau + A\frac{\varepsilon}{\delta^2}\right)} + C \delta_1^2 \delta_2^4 \varepsilon^{-1} \kappa^{-1} e^{-\frac{C}{\kappa}\left(\tau + A\frac{\varepsilon}{\delta^2}\right)}\\
		&\quad + C \delta_1^2 \delta_2^2 \varepsilon^{-1}\kappa^{-1} e^{-\frac{C\delta_2}{\varepsilon}\left(\tau + A\frac{\varepsilon}{\delta^2}\right)} + C \delta_1^2 \delta_2^2 \varepsilon^{-1} \kappa^{-1} e^{-\frac{C}{\kappa}\left(\tau + A\frac{\varepsilon}{\delta^2}\right)} + C \delta_1^2 \delta_2^3 \varepsilon^{-2} e^{-\frac{C\delta_2}{\varepsilon}\left(\tau + A\frac{\varepsilon}{\delta^2}\right)} \\
		&\quad + C \delta_1^2 \delta_2^4 \varepsilon^{-3} \kappa e^{-\frac{C}{\kappa}\left(\tau + A\frac{\varepsilon}{\delta^2}\right)},
	\end{align*}
	where $D_2 = \varepsilon\|\tilde{\psi}_{yy}\|^2 \sim \mathbf{D}_2 = \varepsilon \left\|\frac{\tilde{\psi}_{yy}}{\sqrt{v}}\right\|^2$.
	Combining the above estimates, integrating over $\left[- A\frac{\varepsilon}{\delta^2}, \tau\right]$, and then combining with the estimates \eqref{estimate 0after} and \eqref{estimate phi1 after}, we can obtain the estimate \eqref{estimate psi1 after}.
\end{proof}

 \bigskip

\section*{Acknowledgment}
L.-A. Li was supported by Beijing Natural Science Foundation (No. 1254045), the National Natural Science Foundation of China (No. 12501295) and the Fundamental Research Funds for the Central Universities (No. 2243100008). 
Yi Wang was partially supported by NSFC grants (Grant No. 12421001 and 12288201) and CAS Project for Young Scientists in Basic Research, Grant No. YSBR-031.

\bigskip



\begin{thebibliography}{00}

\bibitem{BB} Bianchini, S., Bressan, A.: \textit{Vanishing viscosity solutions of nonlinear hyperbolic systems.} Ann. of Math. {\bf 161} (2), 223-342 (2005).

\bibitem{Bressan2000} Bressan, A.: \textit{Hyperbolic systems of conservation laws. The one-dimensional Cauchy problem.} Oxford University Press, Oxford, 2000.

\bibitem{Bressan2026} Bressan, A., Caravenna, L., and Shen, W.: \textit{Local asymptotic patterns for viscous approximations of conservation laws.} https://arxiv.org/abs/2605.00189, (2026).

\bibitem{CEJ} Chen, G., Endres, E. and Jenssen, H.: \textit{Pairwise wave interactions in ideal polytropic gases.} Arch. Ration. Mech. Anal. {\bf 204}, 787-836 (2012).

\bibitem{CKV} Chen, G., Kang, M.-J. and Vasseur, A.: \textit{From Navier-Stokes to BV  solutions of the barotropic Euler equations.} https://arxiv.org/abs/2401.09305, 2024.

\bibitem{CP} Chen, G.-Q., Perepelitsa, M.: \textit{Vanishing viscosity limit of the Navier-Stokes equations to the Euler equations for compressible fluid flow.} Comm. Pure Appl. Math. {\bf 63}, 1469-1504 (2010).

\bibitem{GX} Goodman, J., Xin, Z.: \textit{Viscous limits for piecewise smooth solutions to systems of conservation laws.} Arch. Rational Mech. Anal. {\bf 121}, 235-265 (1992).

\bibitem{HL} Hoff, D., Liu, T.-P.: \textit{The inviscid limit for the Navier-Stokes equations of compressible, isentropic flow with shock data.} Indiana Univ. Math. J. {\bf 38}, 861-915 (1989).

\bibitem{Hopf1950} Hopf, E.: \textit{The partial differential equation {$u_t+uu_x=\mu u_{xx}$}.} Comm. Pure Appl. Math. {\bf 3}, 201-230 (1950).

\bibitem{HJW} Huang, F., Jiang, S. and Wang, Y.: \textit{Zero dissipation limit of full compressible Navier-Stokes equations with a Riemann initial data.} Commun. Inf. Syst. {\bf 13}, 211-246 (2013).

\bibitem{HLW} Huang, F., Li, M. and Wang, Y.: \textit{Zero dissipation limit to rarefaction wave with vacuum for one-dimensional compressible Navier-Stokes equations.} SIAM J. Math. Anal. {\bf 44}, 1742-1759 (2012).

\bibitem{HWWW} Huang, F., Wang, W., Wang, Y. and Wang, Y.: \textit{Uniqueness of composite wave of shock and rarefaction in the inviscid limit of Navier-Stokes equations.} SIAM J. Math. Anal. {\bf 56}, 3924-3967 (2024).

\bibitem{HWWY} Huang, F., Wang, Y., Wang, Y. and Yang, T.: \textit{The limit of the Boltzmann equation to the Euler equations for Riemann problems.} SIAM J. Math. Anal. {\bf 45}, 1741-1811 (2013).

\bibitem{HWWY2015} Huang, F., Wang, Y., Wang, Y. and Yang, T.: \textit{Vanishing viscosity of isentropic Navier-Stokes equations for interacting shocks.} Sci. China Math. {\bf 58}, 653-672 (2015).

\bibitem{HWY-1} Huang, F., Wang, Y. and Yang, T.: \textit{Fluid dynamic limit to the Riemann solutions of Euler equations: I. Superposition of rarefaction waves and contact discontinuity.} Kinet. Relat. Models {\bf 3}, 685-728 (2010).

\bibitem{HWY-2} Huang, F., Wang, Y. and Yang, T.: \textit{Vanishing viscosity limit of the compressible Navier-Stokes equations for solutions to a Riemann problem.} Arch. Rational Mech. Anal. {\bf 203}, 379-413 (2012).

\bibitem{JNS} Jiang, S., Ni, G. and Sun, W.: \textit{Vanishing viscosity limit to rarefaction waves for the Navier-Stokes equations of one-dimensional compressible heat-conducting fluids.} SIAM J. Math. Anal. {\bf 38}, 368-384 (2006).

\bibitem{KV} Kang, M.-J., Vasseur, A.: \textit{Uniqueness and stability of entropy shocks to the isentropic Euler system in a class of inviscid limits from a large family of Navier-Stokes systems.} Invent. Math. {\bf 224}, 55-146 (2021).

\bibitem{KV1} Kang, M.-J., Vasseur, A.: \textit{Contraction property for large perturbations of shocks of the barotropic Navier-Stokes system.} J. Eur. Math. Soc. {\bf 23}, 585-638 (2021).

\bibitem{KV2} Kang, M.-J., Vasseur, A.: \textit{Well-posedness of the Riemann problem with two shocks for the isentropic Euler system in a class of vanishing physical viscosity limits.} J. Diff. Equa. {\bf 338}, 128-226 (2022).

\bibitem{KVW} Kang, M.-J., Vasseur, A. and Wang, Y.: \textit{Time-asymptotic stability of composite waves of viscous shock and rarefaction for barotropic {N}avier-{S}tokes equations.} Adv. Math. {\bf 419}, Paper No. 108963, 66 pp. (2023).

\bibitem{KM} Kawashima, S., Matsumura, A.: \textit{Asymptotic stability of traveling wave solutions of systems	for one-dimensional gas motion.} Comm. Math. Phys. {\bf 101}, 97-127 (1985).

\bibitem{L-W-W} Li, M., Wang, T. and Wang, Y.: \textit{The limit to rarefaction wave with vacuum for 1D compressible fluids with temperature-dependent transport coefficients.} Anal. Appl. (Singap.) {\bf 13}, 555-589 (2015).

\bibitem{Liu} Liu, T.-P.: \textit{Shock waves.} Graduate Studies in Mathematics, {\bf 215}, Amer. Math. Soc., Providence, RI, 2021.

\bibitem{M} Ma, S.: \textit{Zero dissipation limit to strong contact discontinuity for the 1-D compressible Navier-Stokes equations.} J. Diff. Equa. {\bf 248}, 95-110 (2010).

\bibitem{SYZ} Shi, X., Yong, Y. and Zhang, Y.: \textit{Vanishing viscosity for non-isentropic gas dynamics with interacting shocks.} Acta Math. Sci. Ser. B (Engl. Ed.), {\bf 36}, 1699-1720 (2016).

\bibitem{Smoller1994} Smoller, J.: \textit{Shock waves and reaction-diffusion equations.} Grundlehren der mathematischen Wissenschaften, {\bf 258}, Springer-Verlag, New York, second edition, 1994.

\bibitem{VW} Vasseur, A., Wang, Y.: \textit{The inviscid limit to a contact discontinuity for the compressible {N}avier-{S}tokes-{F}ourier system using the	relative entropy method.} SIAM J. Math. Anal. {\bf 47}, 4350-4359 (2015).

\bibitem{W} Wang, H.: \textit{Viscous limits for piecewise smooth solutions of the p-system.} J. Math. Anal. Appl. {\bf 299}, 411-432 (2004).

\bibitem{W-2} Wang, Y.: \textit{Zero dissipation limit of the compressible heat-conducting Navier-Stokes equations in the presence of the shock.}  Acta Math. Sci. Ser. B. {\bf 28}, 727-748 (2008).

\bibitem{X-1} Xin, Z.: \textit{Zero dissipation limit to rarefaction waves for the one-dimensional Navier-Stokes equations of compressible isentropic gases.} Comm. Pure Appl. Math. {\bf 46}, 621-665 (1993).

\bibitem{XZ} Xin, Z., Zeng, H.: \textit{Convergence to rarefaction waves for the nonlinear Boltzmann equation and compressible Navier-Stokes equations.} J. Diff. Equa. {\bf 249}, 827-871 (2010).

\bibitem{Y} Yu, S.-H.: \textit{Zero-dissipation limit of solutions with shocks for systems of hyperbolic conservation laws.} Arch. Rational Mech. Anal. {\bf 146}, 275-370 (1999).

\bibitem{ZPT} Zhang, Y., Pan, R. and Tan, Z.: \textit{Zero dissipation limit to a Riemann solution consisting of two shock waves for the 1D compressible isentropic Navier-Stokes equations.} Sci. China Math. {\bf 56}, 2205-2232 (2013).

\bibitem{ZPWT} Zhang, Y., Pan, R., Wang, Y. and Tan, Z.: \textit{Zero dissipation limit with two interacting shocks of the 1{D} non-isentropic {N}avier-{S}tokes equations.} Indiana Univ. Math. J. {\bf 62}, 249-309 (2013).



\end{thebibliography}
\end{document}